\documentclass[autodetect-engine,11pt]{article}
\usepackage[utf8]{inputenc}
\usepackage{mymacros}
\usepackage[all]{xy}
\usepackage{extpfeil}
\usepackage{hyperref}

\newcommand{\Ind}{\mathrm{Ind}}

\usepackage{amsmath}
\usepackage{amssymb}

\usepackage{color}
\usepackage[margin=3.5cm]{geometry}

\newcommand{\Core}{\mathrm{Core}}

\newcommand{\MHS}{\mathrm{MHS}}

\newcommand{\Vect}{\mathrm{Vect}}
\newcommand{\MHM}{\mathrm{MHM}}
\newcommand{\MTM}{\mathrm{MTM}}
\newcommand{\MHW}{\mathrm{MHW}}
\newcommand{\mon}{\mathrm{mon}}

\title{Hodge microsheaves on cotangent bundles and plumbings}
\author{Tatsuki Kuwagaki and Takahiro Saito}

\begin{document}
\maketitle

\begin{abstract}
We introduce and study the category of Hodge microsheaves which is a Hodge-theoretic version of the category of microsheaves for a certain class of holomorphic exact symplectic manifolds. We then study a Hodge-theoretic version of wrapped sheaves and discuss its applications in topology and representation theory. Namely, we study (1) Hain's Hodge structures on the cohomology of based loop spaces of algebraic varieties, and (2) the Koszul duality of Ginzburg algebras by Etg\"u--Lekili from a mixed geometric perspective.
\end{abstract}

\tableofcontents

\section{Introduction}\label{section:MotivationandIntroduction}
In this paper, we study a Hodge-version of the theory of microsheaves. We start with a lengthy explanation of our motivation coming from several area of mathematics.

\subsection{Motivation: Microsheaves}
A Hodge structure on a vector space is a certain decoration of the vector space. It is, for example, associated to a compact K\"ahler manifold, whose cohomology has a canonical Hodge structure, which is useful to study geometry.

To compute global something, it is always useful to localize. Hodge structure also has this feature, and the localized theory known as the theory of Hodge modules by Morihiko Saito~\cite{MHM} is quite strong, and has many applications in various areas of mathematics.

A Hodge module is a decoration of a constructible sheaf, and by the microlocal sheaf theory of Kashiwara--Schapira~\cite{KS}, such sheaves admit further localization called \emph{microlocalization}. Namely, one can view a sheaf on a manifold $M$ as an object living on its cotangent bundle $T^*M$. The recent developments of microlocal sheaf theory (e.g.~\cite{Guillermou, Jin,nadlershende}) strengthen the point of view: constructible sheaves have been understood as the simplest examples of \emph{microsheaves} on Lagrangian submanifolds in symplectic manifolds. A microsheaf over a Lagrangian (micro-)locally looks like the microlocalization of a constructible sheaf.

As localization of Hodge structure (=Hodge module) is useful, we expect that further localization of Hodge structure (= Hodge version of microsheaf) is useful. This is our first motivation.

\subsection{Motivation: Fukaya category}
Fukaya category is a category associated to a symplectic manifold. By the work of Ganatra--Pardon--Shende~\cite{GPSmicro}, a version of Fukaya category called a \emph{partially wrapped Fukaya category} of a Weinstein manifold is known to be equivalent to a category of microsheaves. So, it is also natural to discuss a Hodge version of Fukaya category as well.

An object of a Fukaya category is called a Lagrangian brane. Typically, a Lagrangian brane consists of a Lagrangian submanifold $L$, a local system on $L$, a grading structure, and a Spin-like structure. So, for a Hodge version, it is natural to consider the notion of \emph{Hodge brane}: A Hodge brane consists of a Lagrangian submanifold $L$, a VHS (= the Hodge version of a local system) on $L$, a grading structure, and a Spin-like structure.

If one tries to define a Fukaya-like category using Hodge branes, then they face an obvious difficulty: In the definition of Fukaya category, one uses parallel transport along the boundaries of holomorphic disks of sections of local systems in the brane data. However, we do not have any reasonable definition of parallel transport of sections of VHS. 

If one works with a hyperK\"ahler manifold and complex Lagrangian submanifolds, the situation changes. By the result of Solomon--Verbitsky~\cite{Solomon--Verbitsky}, there are almost no holomorphic disks in the situation. So one does not have to worry about the above difficulty (see also \cite{Joyce,gunningham2024deformationquantizationperversesheaves}), then one still has a hope to define ``Hodge--Fukaya category".

Anyway, in this paper, we will not directly treat Fukaya category, but we bypass it by using sheaf theory under the umbrella of Ganatra--Pardon--Shende~\cite{GPSmicro}. This gives another motivation to consider the Hodge version of microsheaves.

Another aspect to consider from the Fukaya-categorical point of view is that (wrapped) Fukaya category is not proper (i.e., hom-spaces are finite-dimensional) in general, while the category of constructible sheaves is. Such a non-proper version of constructible sheaves is introduced by Nadler~\cite{nadler2016wrappedmicrolocalsheavespairs}, as \emph{wrapped microlocal sheaves}. We then also are motivated to consider a wrapped version of 
Hodge modules.

\subsection{Motivation: Geometric representation theory}
In the last two subsection, we claim that one is naturally led to consider Hodge version of microsheaf/Fukaya category. In this subsection, we would like to say that the answer to the question should be useful, at least to study geometric representation theory.

Geometric representation theory replaces representation theory of some algebras with some geometry. A famous class of such geometry is known as \emph{conical symplectic resolutions}. After some works (e.g.~\cite{KashiwaraRouquier, Kuwabara}), Braden--Licata--Proudfoot--Webster~\cite{BLPW,BLPW2} initiated a unified treatment, in particular, defined a representation category \emph{category $\cO$} for each such geometry. An interesting conjecture surrounding the topic is \emph{symplectic duality}, which asserts that there exist many pairs of symplectic resolutions such that various dualities hold, such as \emph{Koszul duality} of category $\cO$. 

The classical example of such Koszul duality is for category $\cO$ of flag varieties, which is a certain subcategory of $\cD$-modules over flag varieties, and is equivalent to category $\cO$ of Lie algebras. The Koszul duality in this case was established by Beilinson--Ginzburg--Soergel~\cite{BGS}, who used mixed Hodge modules as a Hodge version of category $\cO$ to prove it.

Interpretations of category $\cO$ as microsheaves/Fukaya category have been anticipated and checked for some examples (e.g. \cite{Gammage}). So, it is natural to try to extend Beilinson--Ginzburg--Soergel's story to category $\cO$ of symplectic resolutions: Namely, define Hodge version of category $\cO$ and prove the Koszul duality by using it.

\subsection{Summary}
In this paper, we discuss a certain generalization of the theory of Hodge modules.

As we have mentioned, the theory of Hodge modules is a decoration of the theory of constructible sheaves. The point of view from wrapped Fukaya category enhances the theory of constructible sheaves in two-fold:
\begin{enumerate}
    \item It is defined for more general symplectic manifolds than cotangent bundles.
    \item It is not necessarily finite-dimensional due to wrapping.
\end{enumerate}
In this paper, we consider these kinds of generalizations.

In the following, we will explain the contents of this paper. In \S \ref{section:PreliminariesonHodge modules}, we first briefly explain Hodge theory. Then we discuss several related basic things. One important notion we introduce in the section is \emph{saturation}. This is a condition for a subcategory of a ``mixed version" of a given category, and is important to discuss the Koszul duality later.

In \S \ref{section:FourierTransform}, we recall Fourier transforms of $\cD$-modules, sheaves, and Hodge modules.

In \S \ref{section:Hodgemicrosheaves}, we first recall the theory of microsheaves, which is a generalization of the theory of constructible sheaves to more general symplectic manifolds. For some classes of symplectic manifold, we can express the category of microsheaves by using the gluing via Fourier transformation. This approach was initiated by Bezrukavnikov--Kapranov~\cite{BezrukavnikovKapranov} (see also Karabas--Lee~\cite{karabas2024wrappedfukayacategoryplumbings}). For the case of holomorphic exact symplectic manifolds, the gluing approach was largely generalized by C\^ot\'e--Kuo--Nadler--Shende~\cite{CKNSmicroRH} in the context of microlocal Riemann--Hilbert correspondence. Following their approach, we glue up mixed Hodge modules to obtain the category of Hodge microsheaves, which is closely related to McBreen--Webster's work~\cite{McBreenWebster} on abelian categories of Hodge microsheaves for multiplicative hypertoric manifolds.

In section~\ref{section:Hodgewrapping}, we introduce \emph{Hodge wrapping}, which is a Hodge-theoretic counterpart of wrapping of sheaves. Although, we do not have a general existence result, we give an effective way to compute them. One form of the theorem is roughly the following:
\begin{theorem}\label{thm:main0}
For a mixed Hodge module $\cE$ whose microsupport is in a complex analytic conic Lagrangian $\Lambda$, suppose that 
\begin{enumerate}
    \item there exists $t_0=0<t_1<t_2<...\to \infty$ such that the time $t_i$ Reeb flow image $\Phi_{t_i}(\frakF(\cE))$ of the underlying sheaf $\frakF(\cE)$ of $\cE$ is $\bC$-constructible for any $i$,
    \item there exists an inductive system of mixed Hodge modules $\cE=\cE_0\to \cE_1\to \cE_2\to \cdots $ such that the underlying sheaf $\frakF(\cE_i)$ of $\cE_i$ is $\Phi_{t_i}(\frakF(\cE))$.  
\end{enumerate}
Then $\colim_i \cE_i$ is the Hodge wrapping of $\cE$.
\end{theorem}
See the body of the paper for the precise/microsheaf version of the statement.

In section \S \ref{section:loopspace}, we study Hodge microsheaves on cotangent bundles. 
The category of microsheaves of a cotangent bundle supported on the zero section is generated by an object whose endomorphism is the chains of the based loop space~\cite{Abouzaid,GPSmicro}. This object corresponds to a cotangent fiber under the Ganatra--Pardon--Shende equivalence~\cite{GPSmicro}. If the object carries a Hodge structure, we can equip the chains of the based loop space with a Hodge structure. On the other hand, Hain~\cite{Hain} gives a Hodge structure on the chains of the based loop space via bar construction. So, it is natural to expect:
\begin{conjecture}
    Hain's Hodge structure is induced by a Hodge microsheaf whose underlying microsheaf corresponds to a cotangent fiber.
\end{conjecture}
Precisely speaking, we have to compute the endomorphism of Hodge microsheaves in a saturated category. See the 
body of the paper for details.

\begin{theorem}\label{thm:main1}
    The conjecture holds true for $\bP^n$.
\end{theorem}
The proof of Theorem~\ref{thm:main1} is based on Arai's result~\cite{Arai} and Theorem~\ref{thm:main0}.

In \S~\ref{section:Koszuldualityreview}, we first recall some basics of Koszul duality. Then we recall two philosophies for Koszul duality. The first one is due to Beilinson--Ginzburg--Sorgel~\cite{BGS}: The Koszul duality arises from mixed geometry. The second one is due to Ekholm--Etg\"u--Lekili~\cite{EtguLekili, EkholmLekili}: Core and cocore in Fukaya category often form a Koszul dual pair.

In \S~\ref{section:Koszuldualityforplumbing}, we explain how we can combine the above two philosophies. Namely, in the case when the symplectic manifold $X$ is the $A_n$-plumbing of $T^*\bP^1$. We obtain the following:
\begin{theorem}\label{thm: main2}
    The cocores of $X$ can be lifted to Hodge microsheaves. The (saturated) endomorphism algebra of the Hodge microsheaf exhibits the Koszul duality. 
\end{theorem}
This gives a mixed geometric proof of the Koszul duality of Etg\"u--Lekili~\cite{EtguLekili}. This is closely related to the $\widehat A_n$-plumbing case considered by \cite{McBreenWebster}, which is  explained in Appendix.

Note that the core of $X$ we use for Theorem~\ref{thm: main2} is intrinsic for $X$, and is different from the relative core used in the context of symplectic duality~\cite{BLPW,BLPWhypertoric}. We also discuss how we can recover the Koszul duality in \cite{BLPWhypertoric} from our formalism.

The proof of Theorem~\ref{thm: main2} is based on a very explicit description of the microsheaves corresponding to the cocores, and occupies Appendix. Although the proof is very long, we believe that such an explicit description of wrappings is not previously known and worth recording here.

\subsection*{Acknowledgment}

T.K. thanks Takumi Arai, Tatsuyuki Hikita, Dogancan Karabas, Michael McBreen, and Vivek Shende for related discussions. T.K. also thanks Toshiro Kuwabara and Yoshihisa Saito for introducing him some basic geometric representation theory at the very early stage of this project. T.K. is supported by JSPS KAKENHI  Grant Numbers 22K13912, 23K25765, and 20H01794.
T.S. thanks Tomohiro Asano and Yuichi Ike for helpful discussions on constructible sheaves.
T.S. thanks Genki Sato for his explanation of several facts in category theory.
T.S. also thanks Takuro Mochizuki, Claude Sabbah and Yota Shamoto for their insightful comments and encouragement.
T.S. is supported by JSPS KAKENHI Grant Numbers JP23K19012 and JP23K25765.

\subsection*{Notation}
\begin{itemize}
    \item $\Sh(X,\bK)$: the unbounded derived dg category of $\bK$-sheaves. For $\bK=\bC$, we denote it by $\Sh(X)$.
    \item We set $\Mod(\bK):=\Sh(\mathrm{pt},\bK)$: the unbounded derived dg category of $\bK$-modules.
    \item $\Sh(X, S)$: the unbounded derived dg category of weak constructible $\bC$-sheaves with the possible singularities in $S$.
    \item Six operations (and compositions of them) in this paper are derived functors unless specified.
    \item We use the following notational interpretation freely interchangeably:
    \begin{equation}
        H^0\Hom(-[i], -[k])=\Ext^{k-i}(-,-).
    \end{equation}
\end{itemize}

\section{Hodge modules}\label{section:PreliminariesonHodge modules}
\subsection{Mixed structure}
When understanding about mixed Hodge modules formally, it is convenient to use the terminology ``mixed structure on a given category".

\begin{definition}
Let $\cD$ be a category. A \emph{mixed structure} on $\cD$ consists of the following:
\begin{enumerate}
    \item An autoequivalence $\lb d\rb \colon \widehat{\cD}\rightarrow \widehat{\cD}$ of a category $\widehat \cD$. This is called ($d$-) Tate twist functor. We set $\lb d\rb(\cE):=\cE\lb d\rb$ for any $\cE\in \widehat\cD$.
    \item A functor $\frakF\colon \widehat \cD\rightarrow \cD$ and a natural isomorphism $\epsilon\colon \frakF\circ \lb d\rb \xrightarrow{\cong}\frakF$.
\end{enumerate}
When $\cD$ is a triangulated category, we further suppose that
\begin{enumerate}
    \item $\widehat\cD$ is triangulated, and
    \item $\frakF$ and $(d)$ is exact.
\end{enumerate}

A mixed structure on a dg category is one on the homotopy category.
\end{definition}
\begin{remark}
    This is a rather weaker version of the notions of ``mixed version" appeared in the literature~\cite{BGS, AcharKitchen, Rider}. In fact, this definition does not have much contents. The point of mixed geometry is to find a nontrivial and useful mixed structures.
\end{remark}

\subsection{A lightning introduction to mixed Hodge modules}
Here we would like to provide a rapid introduction to Hodge modules. For details, we refer to \cite{MHM,MHMproject}. 

We first start with the notion of Hodge structures.
\begin{definition}
Let $V$ be a $\bC$-vector space. A \emph{Hodge structure} of weight $n$ on $V$ consists of
\begin{enumerate}
    \item a $\bQ$-vector space with a specified isomorphism $V_\bQ\otimes_\bQ\bC\cong V$,
    \item a decreasing filtration $F^\bullet V$ on $V$
\end{enumerate}
such that 
\begin{equation}
    F^pV\oplus \overline{F^qV}=V
\end{equation}
for any $p, q$ with $p+q=n+1$. Here the overline is an automorphism of $V=V\otimes_\bQ\bC$ induced by the complex conjugation of $\bC$.
\end{definition}

\begin{example}[Tate Hodge structure]
    We set $V_\bQ=2\pi\sqrt{-1}\bQ\subset \bC=V$, $F^1V=V, F^0V=0$, and consider it as a Hodge structure of weight 2. We denote it by $\bQ(1)$.
\end{example}

Sometimes different weights are mixed in objects of our interests. In that case, we use the notion of mixed Hodges structures:
\begin{definition}
    Let $V$ be a $\bC$-vector space. A \emph{mixed Hodge structure} on $V$ consists of
    \begin{enumerate}
    \item a $\bQ$-vector space with a specified isomorphism $V_\bQ\otimes_\bQ\bC\cong V$,
    \item a decreasing filtration $F^\bullet V$ on $V$, and
    \item an increasing filtration $W_\bullet V_\bQ$ on $V_\bQ$
\end{enumerate}
such that each graded quotient $\Gr_i^WV$ is a Hodge structure of weight $i$ with the induced $\bQ$-structure and filtration.
\end{definition}
We denote the category of mixed Hodge structures by $\MHS^\heartsuit$.
We can easily see the following.
\begin{proposition}
    The following data form a mixed structure on the category $\Vect_\bQ$ of $\bQ$-vector spaces:
    \begin{enumerate}
        \item Tate twist functor $\lb 1\rb :=(-)\otimes \bQ(1)$.
        \item The forgetful functor $\frakF\colon \MHS^\heartsuit \to \Vect_\bQ$.
    \end{enumerate}
\end{proposition}

We next got to the smooth family version of Hodge structure, which is called variation of Hodges structures:
\begin{definition}
    Let $X$ be a complex manifold. Let $(\cM, \nabla)$ be a holomorphic flat connection on $X$. A \emph{variation of Hodge structures} of weight $n$ on $\cM$ consists of
\begin{enumerate}
    \item a $\bQ$-local system $\cL$ with a specified isomorphism $\cL\otimes_\bQ\bC\cong \mathrm{Sol}(\cM,\nabla)$ where $\mathrm{Sol}(\cM,\nabla)$ is the local system of flat sections, and
    \item a decreasing filtration $F^\bullet \cM$ of $\cM$
\end{enumerate}
with the following hold:
\begin{enumerate}
    \item $F^p\cM\oplus \overline{F^q\cM}=\cM$ for any $p, q$ with $p+q=n+1$ as $C^\infty$-bundles, and
    \item (Griffhts transversality) $\nabla_v F^pV\subset F^{p-1}V$ for any $v\in TX$.
\end{enumerate}
\end{definition}

Similarly, we have the notion of variation of mixed Hodge structures.
\begin{definition}
    Let $X$ be a complex manifold. Let $(\cM, \nabla)$ be a holomorphic flat connection on $X$. A \emph{variation of Hodge structures} (VHS for short) on $\cM$ consists of
    \begin{enumerate}
    \item a $\bQ$-local system with a specified isomorphism $\cL\otimes_\bQ\bC\cong \mathrm{Sol}(\cM,\nabla)$,
    \item a decreasing filtration $F^\bullet \cM$ of $\cM$, and
    \item an increasing filtration $W_\bullet \cL$ of $\cL$
\end{enumerate}
such that each graded quotient $\Gr_i^W\cM$ is a variation of Hodge structures of weight $i$ with the induced $\bQ$-structure and filtration.
\end{definition}
We denote the category of VHS on $X$ by $\mathrm{VHS}^\heartsuit(X)$.
We can similarly have the following.
\begin{proposition}
    The following data form a mixed structure on the category $\mathrm{Loc}_\bQ(X)$ of $\bQ$-local systems on $X$:
    \begin{enumerate}
        \item Tate twist functor $\lb 1\rb :=(-)\otimes \bQ(1)$.
        \item The forgetful functor $\frakF\colon \mathrm{VHS}^\heartsuit(X) \to \mathrm{Loc}_\bQ(X)$.
    \end{enumerate}
\end{proposition}

We finally come to the stage of Hodge modules. Roughly speaking, the notion of Hodge module is a singular family version of Hodge structures. Since a singular version of a flat connection is a $\cD$-module, a Hodge module is a certain structure on a $\cD$-module. The first approximation is the following:
\begin{definition}
Let $X$ be a complex manifold. An object of the category $\MHW^{c,\heartsuit}(X)$ is a tuple $(\cM, \cE,F, W)$ such that
\begin{enumerate}
    \item $\cM$ is a regular holonomic $\cD_X$-module,
    \item $\cE$ is a $\bQ$-perverse sheaf with a specified isomorphism $\cE\otimes_\bQ\bC\cong \mathrm{DR}(\cM)$, the de Rham image of $\cM$,
    \item $W$ is an increasing filtration on $\cE$,
    \item $F$ is a decreasing $\cO_X$-module filtration on $\cM$.
\end{enumerate}
A morphism in $\MHW^{c,\heartsuit}(X)$ is a morphism between $\bQ$-perverse sheaves preserving the filtrations.
\end{definition}

\begin{proposition}
    The following data form a mixed structure on the category $\Sh^\heartsuit(X,\bQ)$ of $\bQ$-sheaves on $X$:
    \begin{enumerate}
        \item Tate twist functor $\lb 1\rb :=(-)\otimes \bQ(1)$.
        \item The forgetful functor $\frakF\colon \mathrm{MHW}^{c,\heartsuit}(X) \to \Sh^\heartsuit(X,\bQ)$.
    \end{enumerate}
\end{proposition}

The category $\MHW^{c,\heartsuit}(X)$ is too big, and the work of Morihiko Saito~\cite{MHM} gives a certain nice subcategory of $\MHW^{c,\heartsuit}(X)$, \emph{the category of mixed Hodge modules $\MHM^{c,\heartsuit}(X)$}. Here is a rough description of it:
\begin{theorem}[\cite{MHM,tubach2024norihodgerealisationsvoevodsky}]
    There exists a full abelian subcategory $\MHM^{c,\heartsuit}(X)$ of $\MHW^{c,\heartsuit}(X)$ such that its derived dg category $\MHM^c(X)$ satisfies that
    \begin{enumerate}
        \item $\MHM^c(\{*\})$ is the bounded derived dg category of graded polarizable mixed Hodge structures,
        \item $\MHM^c(X)$ contains all the admissible graded polarizable variations of mixed Hodge structures,
        \item several important operations (e.g. six operations, nearby/vanishing cycle) work, and
        \item some nice theorems (e.g. the decomposition theorem) hold.
    \end{enumerate}
\end{theorem}
We do not define several adjectives (graded polarizable, admissible) here, see \cite{MHMproject}. 
Similarly, we have the following:
\begin{proposition}\label{prop:mhmismix}
    The following data induce a mixed structure on the category $\Sh_{\mathrm{constr}}(X, \bQ)$ of cohomologically constructible sheaves on $X$:
    \begin{enumerate}
        \item Tate twist functor $\lb 1\rb :=(-)\otimes \bQ(1)$.
        \item The forgetful functor $\frakF\colon \MHM^c(X) \to \Sh_{\mathrm{constr}}(X, \bQ)$.
    \end{enumerate}
    The restriction of data to $\MHM^{c, \heartsuit}(X)$ gives a mixed structure on the  category of perverse sheaves on $X$.
\end{proposition}

There is a further enlarged version $\MTM(X)$, the derived category of mixed twistor modules studied by Simpson, Sabbah, Mochizuki~\cite{mochizukiMTM}. We would like to package as follows:
\begin{theorem}
    \begin{enumerate}
        \item There exists a category $\MTM(X)$ which gives a mixed structure on the derived category of holonomic $\cD$-modules. 
        \item $\MHM^c(X)$ is fully faithfully embedded into $\MTM(X)$.
        \item There exists an autoequivalence $\lb\frac{1}{2}\rb$ on $\MTM(X)$ satisfying $\lb \frac{1}{2}\rb^2|_{\MHM^c(X)}=\lb 1\rb$.
    \end{enumerate}
\end{theorem}

\subsection{Infinite-dimensional Hodge modules}\label{section:Infinite-dimensionalHodge}
Let $X$ be a complex manifold.
For our purpose, we need an infinite-dimensional version of Hodge modules. For this purpose, we simply set as follows:
\begin{definition}
    \begin{enumerate}
        \item We set
        \begin{equation}
    \MHM(X):=\Ind(\MHM^c(X)).
\end{equation}
where $\Ind$ is the category of Ind-objects.
\item We define $\frakF\colon \MHM(X)\rightarrow \Sh(X,\bQ)$ by the composition of $\Ind(\frakF)\colon \Ind(\MHM^c(X))\rightarrow \Ind(\Sh(X,\bQ))$ and the colimit realization functor $\Ind(\Sh(X,\bQ))\rightarrow \Sh(X,\bQ)$.
\item For a conic Lagrangian $\Lambda\subset T^*X$, the category $\Sh_{\Lambda}(X,\bQ)$ is the subcategory of $\Sh(X,\bQ)$ spanned by the objects satisfying $\mathrm{SS}\subset \Lambda$. We set
\begin{equation}
    \MHM_\Lambda(X):=\frakF^{-1}(\Sh_{\Lambda}(X, \bQ)).
\end{equation}
    \end{enumerate}
\end{definition}

Let $\Lambda$ be a complex analytic conic Lagrangian. We denote the full subcategory of $\MHM^c(X)$ spanned by the objects whose microsupports are contained in $\Lambda$ by $\MHM^c_\Lambda(X)$. Then obviously, $\Ind (\MHM^c_{\Lambda}(X))\subset \MHM_\Lambda(X)$.

\subsection{Saturation}
\begin{definition}
Let $(\frakF\colon \widehat \cD\to \cD, \lb d\rb , \epsilon)$ be a mixed structure on abelian or triangulated category $\cD$. 
\begin{enumerate}
    \item We say an object $M\in \widehat\cD$ is \emph{saturated} if the morphism
    \begin{equation}
    \bigoplus_{n\in \bZ}\Ext^i_{\widehat \cD}(M, N\lb nd\rb )\xrightarrow{\frakF} \Ext^i_{\cD}(\frakF (M), \frakF (N))
\end{equation}
is isomorphic for any $N$ and $i\in \bZ$.
\item We say $\widehat \cD$ is \emph{saturated} if any object in $\widehat \cD$ is saturated.
\end{enumerate}
\end{definition}

As discussed in \cite{BGS}, the saturation property is crucial in the discussion of Koszul duality. However, the whole category of mixed Hodge modules do not satisfy the saturatedness. Here are easy counterexamples.
\begin{example}
\begin{enumerate}
    \item Suppose $V_1=\bC$ is the (unique) Hodge structure of weight $0$. Let $V_2=\bC^2$ be the (unique) Hodge structure of weight 1. Then 
    \begin{equation}
        0=\bigoplus_i\Hom_{\mathrm{MHS}}(V_1, V_2\lb i \rb)\neq \Hom_{\mathrm{Vect}_\bQ}(\frakF(V_1), \frakF(V_2))=\bQ^2.
    \end{equation}
    Hence $\mathrm{MHS}$ is not a saturated mixed structure of $\mathrm{Vect}_\bQ$.
    \item 
    We demonstrate another example explaining the non-saturatedness of $\MHS$.
    Suppose $V_1=\bC$ is the (unique) Hodge structure of weight $0$. Then $\bigoplus_i \Ext^1_{\MHS}(V,V(i))\neq 0$ because there are several nontrivial mixed Hodge structures of rank 2, whereas $\Ext^1_{\Vect_\bQ}(V,V)=0$.
\end{enumerate}
\end{example}

We first give the following lemma.
\begin{lemma}\label{lem:finite_generating_saturated} Let $(\frakF\colon \widehat \cD\to \cD, \lb d\rb , \epsilon)$ be a mixed structure on a triangulated category $\cD$. Let $\{\cE_i\}_{i\in I}$ be a set of objects of $\widehat\cD$. If the subcategory spanned by $\lc \cE_i[j]\relmid i\in I, j\in \bZ\rc$ is saturated, then the subcategory generated by $\lc \cE_i[j]\relmid i\in I, j\in \bZ\rc$ under the shifts, Tate twists, and taking cones is also saturated as well.
\end{lemma}
\begin{proof}
Since the generating set is closed under shifts, the closedness under the shifts is obvious. 

    Let $\widehat\cD'$ be a maximal saturated subcategory of $\widehat\cD$ containing $\cE_i$. For $\cE, \cF\in \widehat\cD'$. Then  We have that
\begin{equation}
    \bigoplus_{i}\Hom_{\widehat\cD}(\cE(j),\cF(k)(i))=\bigoplus_{i}\Hom_{\widehat\cD}(\cE,\cF(i+j-k))=\Hom_{\cD}(\frakF(\cE),\frakF(\cF)).
\end{equation}
Hence $\cE(j), \cF(k)\in \widehat \cD'$.

For $\cE,\cF,\cG\in \widehat\cD'$, we have
\begin{equation}
\begin{split}
    \bigoplus_{i}\Hom_{\widehat \cD}(\cG, \Cone(\cE\to \cF)(i))&\cong \Cone( \bigoplus_{i}\Hom_{\widehat \cD}(\cG, \cE(i))\to \bigoplus_i\Hom_{\widehat \cD}(\cG, \cF(i)))\\
    &\cong  \Cone( \Hom_{\cD}(\frakF(\cG), \frakF(\cE))\to \Hom_{\cD}(\frakF(\cG), \frakF(\cF)))\\
    &\cong \Hom_{\cD}(\frakF(\cG), \Cone(\frakF(\cE)\rightarrow \frakF(\cF)))\\
    &\cong \Hom_{\cD}(\frakF(\cG), \frakF(\Cone(\cE\to\cF))).
\end{split}
\end{equation}
In a similar way, we can prove the remaining equalities to prove $\Cone(\cE\to \cF)\in \widehat\cD'$. This completes the proof.
\end{proof}

\begin{example}[{\cite[\S4, Example (3)]{BGS}}]\label{ex:Hodge-Tate}
\begin{enumerate}
    \item 
Consider the subcategory $\cT^\heartsuit$ of $\MHS^\heartsuit$ spanned by the finite direct sums of  $\bQ(1)^{\otimes n}$ for any $n$. It forms an abelian category. It is obvious that $\cT^\heartsuit$ is a saturated mixed structure on $\Vect_\bQ$. We call an object in $\cT$ is Hodge--Tate. By Lemma~\ref{lem:finite_generating_saturated}, the derived category $\cT$ of $\cT^\heartsuit$ is also saturated.

\item\label{ex:Hodge-Tate2}
For $\bP^1$, take $3$ (or 2) different points $\lc l,m, r\rc$ (or $m,r$) on it. 

We consider the sub dg category $M'$ spanned by $\bC_{\bP^1}[1], (\bC_{\bP^1\bs\lc i\rc}\xrightarrow{\iota} \bC_{\bP^1})\in \MHM(\bP^1)$ where $i=l,m,r$ where $\iota$ is the canonical injection. Note that $(\bC_{\bP^1\bs\lc i\rc}\xrightarrow{\iota} \bC_{\bP^1})\cong \bC_i$. 

In each hom-space, we consider the following subspaces:
\begin{equation}
    \begin{split}
        \bQ \cdot \id\oplus \bQ\cdot [\bP^1] &\subset\bigoplus_j \Hom_{\MHM(\bP^1)}(\bC_{\bP^1}[1],\bC_{\bP^1}[1](j))\\
            \bQ \cdot u_i&\subset \bigoplus_j\Hom_{\MHM(\bP^1)}(\bC_{\bP^1}[1],(\bC_{\bP^1\bs\lc i\rc}\xrightarrow{\iota} \bC_{\bP^1})(j))\\
               \bQ\cdot \iota_i &\subset \bigoplus_j\Hom_{\MHM(\bP^1)}((\bC_{\bP^1\bs\lc i\rc}\xrightarrow{\iota} \bC_{\bP^1}),\bC_{\bP^1}[1](j))\\
                \bQ\cdot \id &\subset \bigoplus_j\Hom_{\MHM(\bP^1)}((\bC_{\bP^1\bs\lc i\rc}\xrightarrow{\iota} \bC_{\bP^1}),(\bC_{\bP^1\bs\lc i\rc}\xrightarrow{\iota} \bC_{\bP^1})(j))  
    \end{split}
\end{equation}
where $[\bP^1]$ is the fundamental class, $u_i$ is the morphism corresponding to the identity under $\Hom_{\MHM(\bP^1)}(\bC_{\bP^1}[1],\bC_{\bP^1}[1])\cong \Hom_{\MHM(\bP^1)}(\bC_{i}[1],\bC_i[1])$, and $\iota_i$ is the morphism induced by $\iota$. One can check that the left hand sides defines a wide subcategory $M^{\mathrm{pre
}}$ of $M'$. We denote the pretriangulated dg category of finite twisted complexes (i.e., iterated cones of $M^{\mathrm{pre}}$) by $M$. Since $M$ is generated by shifts, Tate twists, and taking cones by $M^{\mathrm{pre}}$, Lemma~\ref{lem:finite_generating_saturated} implies that $M$ is saturated. Also, by the construction, $M$ is a non-full subcategory of $\MHM^c(\bP^1)$.

\item \label{ex:Hodge-Tate3}For $\bP^n$, we take a hyperplane $H$ and a point $x\not\in H$. The we consider the objects $\bC_{\bP^n}[n], \bC_H[n-1], \bC_x$. Similarly, we can similarly take consider the subspaces of morphisms. One can observe that the category spanned by them is a wide saturated sub dg category. Then, by taking the dg category $M$ of finite twisted complexes of $M'$, we obtain a saturated mixed structure by Lemma~\ref{lem:finite_generating_saturated}. Again, this is a (non-full) subcategory of $\MHM^c(\bP^n)$.
\end{enumerate}
\end{example}

\subsection{Linear algebraic description of Hodge modules}\label{section:linearalgebraicHodge}
For the later purpose, we recall linear algebraic description of Hodge modules from \cite{SaitoMonodromic}, which is a mixed version of the classical results for perverse sheaves/D-modules.

\begin{definition}
    Let $\cC$ be a category. Let $\widehat\cC$ be an abelian category and a mixed structure on $\cC$. A \emph{mixed monodromic object} $C=(C_{(-1,0]}, T_s, N, C_{-1}, c,v)$ of $\widehat\cC$ consists of the following:
    \begin{enumerate}
        \item $C_{(-1,0]}, C_{-1}\in \widehat\cC$. 
        \item $T_s\in \Aut(C_{(-1,0]}), N\in \Hom(C_{(-1,0]}, C_{(-1,0]}(-1))$ 
        \item $c\in \Hom(\Ker(T_s-1), C_{-1}), v\in \Hom(C_{-1},\Ker(T_s-1)(-1))$
    \end{enumerate}
    such that
    \begin{enumerate}
        \item $T_s\circ N=N\circ T_s$,
        \item $C_{(-1,0]}=\bigoplus_{\alpha\in (-1,0]\cap \bQ}\Ker(T_s-\exp(-2\pi\sqrt{-1}\alpha))=:C_\alpha$,
        \item $vc=N|_{\Ker (T_s-1)}$.
    \end{enumerate}
    Morphisms between mixed monodromic objects are defined to be obvious compatible morphisms. We denote the derived category of mixed monodromic objects by $\mathrm{Mon}(\widehat{\cC})$.
\end{definition}

\begin{theorem}[{\cite{SaitoMonodromic}}]\label{thm:MHMlinearalgebra}
Let $X$ be a complex manifold. The category $\MHM^c_{\mon}(X\times \bC_t)$ of monodromic mixed Hodge modules on $X$ is equivalent to $\mathrm{Mon}(\MHM^c(X))$.
\end{theorem}

We denote the colimit closure of $\MHM^c_{\mon}(X\times \bC_t)$ in $\MHM(X\times \bC_t)$ by $\MHM_{\mon}(X\times \bC_t)$.
\subsection{Half-Tate twist}
Let $\widehat\cD$ be a category with an autoequivalence $(1)\colon \widehat\cD\rightarrow \widehat\cD$. In this setup, we define the square root $(\frac{1}{2})$ of $(1)$ as follows: We consider the direct product category $\sqrt{\widehat\cD}:=\widehat\cD\times\widehat \cD$. Then we set
\begin{equation}
    \lb \frac{1}{2}\rb \colon \sqrt{\widehat\cD}\rightarrow \sqrt{\widehat\cD}; (c,c')\mapsto (c'(1), c).
\end{equation}
Note that $(\frac{1}{2})^2=(1)\times (1)$. 

We have a fully faithful embedding $\widehat\cD\hookrightarrow \sqrt{\widehat\cD}; c\mapsto (c, 0)$.

\begin{example}\label{example:half-tateTwistor}
We consider the case when $\widehat\cD=\MHM^c(X)$. Note that $\MTM(X)$ contains $\MHM^c(X)$ and already has the square root of $(1)$. Indeed, we have
\begin{equation}
    \sqrt{\MHM^c(X)}\hookrightarrow \MTM(X),
\end{equation}
compatible with the half-Tate twist action.
\end{example}

Suppose $(\frakF\colon \widehat \cD\rightarrow \cD, (1), \epsilon)$ is a saturated mixed structure. We then have the composition
\begin{equation}
    \sqrt{\frakF}\colon \sqrt{\widehat\cD}=\widehat\cD\times \widehat\cD\xrightarrow{\oplus} \widehat\cD\xrightarrow{\frakF}\cD.
\end{equation}
\begin{lemma}
  The pair $(\sqrt{\frakF}\colon \sqrt{\widehat\cD}\rightarrow \cD, \lb\frac{1}{2}\rb)$ induces  a saturated mixed structure.   
\end{lemma}
\begin{proof}
    It is enough to show the saturatedness: For $c=(c_1, 0), c'=(c_2,0)\in \sqrt{\widehat\cD}$,
    \begin{equation}
    \begin{split}
        \bigoplus_{j}\Hom_{\sqrt{\widehat\cD}}(c, c'\lb \frac{j}{2}\rb)&\cong \bigoplus_{j}\Hom_{\sqrt{\widehat\cD}}(c, c'\lb j\rb)\\
        &\cong \bigoplus_{j}\Hom_{\widehat\cD}(c_1, c_2\lb j\rb)\\
        &\cong \Hom_{\cD}(\frakF(c_1), \frakF(c_2))\\
        &\cong \Hom_{\cD}(\sqrt{\frakF}(c), \sqrt{\frakF}(c')).
    \end{split}
    \end{equation}
    For $c=(c_1, 0), c'=(0,c_2)\in \sqrt{\widehat\cD}$,
    \begin{equation}
    \begin{split}
        \bigoplus_{j}\Hom_{\sqrt{\widehat\cD}}(c, c'\lb \frac{j}{2}\rb)&\cong \bigoplus_{j}\Hom_{\sqrt{\widehat\cD}}(c, c'\lb j+\frac{1}{2}\rb)\\
        &\cong \bigoplus_{j}\Hom_{\widehat\cD}(c_1, c_2\lb j\rb)\\
        &\cong \Hom_{\cD}(\frakF(c_1), \frakF(c_2))\\
         &\cong \Hom_{\cD}(\sqrt{\frakF}(c), \sqrt{\frakF}(c')).
    \end{split}
    \end{equation}
    The remaining cases are similar.
\end{proof}

\section{Fourier transformation}\label{section:FourierTransform}
\subsection{Fourier--Laplace transformation on $\cD$-modules}
We first define the Weyl algebra of $n$ variables. Let $z_1,...,z_n, \zeta_1,...,\zeta_n$ be indeterminates. We denote the free algebra generated by these elements by $\bC\la z_1,...,z_n, \zeta_1,...,\zeta_n\ra$. We consider the bi-sided ideal generated by 
\begin{equation}
    z_iz_j-z_jz_i, \zeta_j z_i-z_i\zeta_j-\delta_{ij}
\end{equation}
where $\delta_{ij}$ is Kronecker's delta. The quotient of the free algebra by this ideal is denoted by $\cD_n$. The derived category $\Mod(\cD_{\bC^n})$ of algebraic $\cD$-modules on $\bC^n$ is equivalent to the category of modules over $\cD_n$.

\begin{definition-lemma}
    We consider the map $\mathrm{FL}\colon \cD_n\rightarrow \cD_n$ defined by 
    \begin{equation}
        z_i\mapsto \zeta_i, \zeta_i\mapsto -z_i.
    \end{equation}
    Then this defines an automorphism of $\cD_n$. Hence it induces a derived category equivalence $\Mod(\cD_{\bC^n})\rightarrow \Mod(\cD_{\bC^n})$. We also denote it by $\mathrm{FL}$, and call it the Fourier(--Laplace) transform.
\end{definition-lemma}
We also have the relative version: Let $V\to X$ be a holomorphic vector bundle. Let $V^*$ be the dual bundle. Then we similarly have the relative version of Fourier transform
\begin{equation}
   \mathrm{FL}\colon \Mod(\cD_V)\xrightarrow{\sim} \Mod(\cD_{V^*}).
\end{equation}

Although Fourier transform does not preserve the regularity in general, we have the following :
\begin{theorem}[{Brylinski~\cite{Brylinski}}]
    Let $\cM$ be a regular holonomic $\cD_{V}$-module whose characteristic variety is invariant under the scaling $\bC^*$-action on the fibers of $V$ (``monodromic"). Then $\frakF(\cM)$ is again a regular monodromic $\cD_{V^*}$-module. In other words, $\mathrm{FL}$ induces an equivalence between the derived category of regular monodromic $\cD$-modules $\Mod^c_{\mon}(\cD_V)\rightarrow \Mod^c_{\mon}(\cD_{V^*})$.
\end{theorem}

Although this statement is finite-dimensional, we are interested in infinite-dimensional version. Let $\Mod_{\mon}(\cD_V)$ be the colimit-closure (which is equivalent to the category of ind-objects) of $\Mod^c_{\mon}(\cD_V)$. Since $\mathrm{FL}$ is an equivalence, it preserves the colimits. Hence we have 
\begin{corollary}
    \begin{equation}
        \Mod_{\mon}(\cD_V)\xrightarrow{\cong} \Mod_{\mon}(\cD_{V^*}).
    \end{equation}
\end{corollary}

\subsection{Fourier--Sato transformation on monodromic constructible sheaves}
In the situation of Brylinski's theorem, by the Riemann--Hilbert correspondence, we can induce an equivalence on the categories of monodromic $\bC$-constructible sheaves. One can formulate this induced autoequivalence without referring to $\cD$-side, which is called the Fourier--Sato transform. 

Let $V$ be a (real) vector bundle over $X$ and $V^*$ be its dual bundle. We set
\begin{equation}
    S:=\lc (x, x^*)\in V\times V^*\relmid x^*(x)\geq 0\rc.
\end{equation}
We denote the constant sheaf on $S$ by $\bK_S$. We denote the $i$-th projection of $V\times V^*$ by $p_i$. We set
\begin{equation}
    \mathrm{FS}\colon \Sh(V,\bK)\rightarrow \Sh(V^*, \bK); \cE\mapsto p_{2!}(p_1^{-1}\cE\otimes \bK_S)
\end{equation}
We denote the subcategory of constructible sheaves in $\Sh(V, \bK)$ whose microsupport is invariant under the $\bC^*$-scaling action on the fibers of $V$ by $\Sh^c_{\mon}(V, \bK)$.

\begin{theorem}[{\cite{KS, Brylinski}}]
    \begin{enumerate}
        \item $\mathrm{FS}$ induces an equivalence
        \begin{equation}
           \mathrm{FS}\colon  \Sh^c_{\mon}(V, \bK)\rightarrow \Sh^c_{\mon}(V^*, \bK).
        \end{equation}
        \item The de Rham functor of the regular RH-correspondence induces an equivalence
        \begin{equation}
            \mathrm{DR}\colon \Mod_{\mon}^c(\cD_V)\xrightarrow{\cong} \Sh^c_{\mon}(V, \bC).
        \end{equation}
        \item The following diagram is commutative:
        \begin{equation}
            \xymatrix{
            \Mod_{\mon}^c(\cD_V)\ar[r]_{\cong}^{\mathrm{FL}}\ar[d]_{\cong}^{\mathrm{DR}}&\Mod^c_{\mon}(\cD_{V^*})\ar[d]_{\cong}^{\mathrm{DR}}\\
            \Sh^c_{\mon}(V, \bC)\ar[r]_{\cong}^{\mathrm{FS}}&\Sh^c_{\mon}(V^*, \bC)
            }
        \end{equation}
    \end{enumerate}
\end{theorem}

We again consider the colimit closure of $\Sh^c_{\mon}(V, \bK)$ and denote it by $\Sh_{\mon}(V, \bK)$. Then we obtain the following:
\begin{theorem}
    \begin{enumerate}
        \item $\mathrm{FS}$ induces an equivalence
        \begin{equation}
           \mathrm{FS}\colon  \Sh_{\mon}(V, \bK)\xrightarrow{\cong}\Sh_{\mon}(V^*, \bK).
        \end{equation}
        \item The de Rham functor induces a functor
        \begin{equation}
            \mathrm{DR}\colon \Mod_{\mon}(\cD_V)\xrightarrow{} \Sh_{\mon}(V, \bC).
        \end{equation}
        \item The following diagram is commutative:
        \begin{equation}
            \xymatrix{
            \Mod_{\mon}(\cD_V)\ar[r]_{\cong}^{\mathrm{FL}}\ar[d]^{\mathrm{DR}}&\Mod_{\mon}(\cD_{V^*})\ar[d]^{\mathrm{DR}}\\
            \Sh_{\mon}(V, \bC)\ar[r]_{\cong}^{\mathrm{FS}}&\Sh_{\mon}(V^*, \bC)
            }
        \end{equation}
    \end{enumerate}
\end{theorem}

\subsection{Fourier transformation on monodromic Hodge modules} \label{jan7-1}
Reichelt--Walther~\cite{ReicheltWalther} introduced a notion of Fourier transform of monodromic Hodge modules by lifting that of monodromic $\cD$-modules. The second-named author~\cite{SaitoMonodromic} provided a different explicit definition of Fourier transform. Later, Chen--Dirks~\cite{chen2023vfiltrationhodgefiltrationfourier} proved that these definitions are equivalent.

We now give the definition of Fourier transform of mixed Hodge modules based on \cite{SaitoMonodromic, saito2023hodgefiltrationmonodromicmixed}.

\begin{definition}\label{fev5-1}
In the setup of Section~\ref{section:linearalgebraicHodge}, the Fourier transform is the following functor:
    \begin{align}
        &\mathrm{Mon}(\widehat \cC) \rightarrow \mathrm{Mon}(\widehat \cC)\\
        &(C_{(-1,0]}, T_s, N, C_{-1}, c,v) \mapsto (C_{-1}\oplus \bigoplus_{(-1,0)}C_{\alpha},1\oplus T_s^{-1}, c\circ v\oplus N, C_{1}(-1), -v,c).
    \end{align}
\end{definition}

\begin{theorem}[\cite{SaitoMonodromic}]
    When $\widehat{\cC}$ is $\MHM^c(X)$, the Fourier transform defined above is an equivalence $\MHM^c_{\mon}(X\times \bC_t)\xrightarrow{\simeq}\MHM^c_{\mon}(X\times\bC_t)$ lifting the Fourier transform of the category of monodromic $\cD$-modules on $X\times \bC_t$.
\end{theorem}

Let $\Lambda$ be a complex analytic conic Lagrangian in $T^*(X\times \bC_t)$ which is invariant under the $\bC^*$-action. 

We set $\MHM_{\mathrm{mon}, \Lambda}(X):=\Ind(\MHM^c_{\mathrm{mon},\Lambda}(X))$ where the subscript $\Lambda$ is about the restriction of the microsupport. Then we obtain the infinite-dimensional version
\begin{equation}
    \mathrm{FL}^{\mathrm{pre}}\colon \MHM_{\mon,\Lambda}(X\times \bC_t)\xrightarrow{\simeq}\MHM_{\mon,\Lambda}(X\times\bC_t).
\end{equation}
 By the construction, we have $(\mathrm{FL}^{\mathrm{pre}})^2=(-1)$. To make the functor unipotent, we set
\begin{equation}
    \mathrm{FL}:=\mathrm{FL}^{\mathrm{pre}}\lb\frac{1}{2}\rb\colon  \sqrt{\MHM_{\mon,\Lambda}(X\times \bC_t)}\xrightarrow{\simeq}\sqrt{\MHM_{\mon,\Lambda}(X\times\bC_t)}.
\end{equation}
This satisfies $\mathrm{FL}^2=\id$. The underlying functor on $\cD$-modules and sheaves are $\mathrm{FL}$ and $\mathrm{FS}$.

\begin{remark}[Twistor version]
We can similarly consider the Fourier transform for twisor $\cD$-modules. From the viewpoint of Example~\ref{example:half-tateTwistor}, it is natural to ask the relation between the twistor version and the one here: 
    There is a subtle difference between them.
Since any ``exponential $D$-module" is the underlying $D$-module of a mixed twistor $D$-module,
we can naturally define the Fourier transform of a mixed twistor $D$-module.
On the other hand,    
the category of mixed Hodge module is a full subcategory of the category of mixed twistor $D$-module.
So, for a monodromic mixed Hodge module
we regard it as a mixed twistor $D$-module and get the Fourier transform of it as a mixed twistor $D$-module.
However, this object may not be in the (sub)category of mixed Hodge modules, because the integrable structure ``$\lambda^2\partial_\lambda$" of its underlying $R$-module may have an eigenvalue $\alpha\lambda$ for a non-integer complex number $\alpha\in \CC$ in general, as seen in Lemma~5.20 of \cite{saito2023hodgefiltrationmonodromicmixed}.
If we change the integrable structure of the $R$-module,
it coincides with what we have defined above.
\end{remark}

\section{Hodge microsheaves}\label{section:Hodgemicrosheaves}
\subsection{Preliminaries on microsheaves}
In this section, we review the definition of the category of microsheaves of Shende~\cite{h-principle}, Nadler--Shende~\cite{nadlershende}, and related materials.

For a while, we will consider the case when the sheaf coefficient is $\bZ$. Let $V$ be a contact manifold. The contact distribution $\xi$ of $V$ is a symplectic bundle, so it has the classifying map $V\rightarrow BU$. The universal Kashiwara--Schapira sheaf introduced in \cite{h-principle,nadlershende} is classified by $U/O\rightarrow B\bZ\times B^2\bZ/2\bZ$. Hence we obtain a sequence of maps
\begin{equation}
    V\rightarrow BU\rightarrow B(U/O)\rightarrow B^2\bZ\times B^3\bZ/2\bZ.
\end{equation}
Each null homotopy $\frakp$ of the composed map is called Maslov data. It gives a sheaf of categories on $V$, which is denoted by $\mu \mathrm{sh}_{\frakp}:=\mu \mathrm{sh}_{\frakp,V}$. Note that the second component $V\rightarrow B^3\bZ/2\bZ$ is factored as $V\rightarrow BU\rightarrow B(U/O)\rightarrow B^2O\rightarrow B^3\bZ/2\bZ$, hence this has a canonical null homotopy~\cite{CKNSmicroPerverse}.

\begin{example}
    Let $(X,\lambda)$ be an exact symplectic manifold i.e., $\lambda$ is a 1-form such that $d\lambda$ is a symplectic form on $X$. Consider the contactization $X\times \bR$. When the context is clear, we denote the restriction of $\mu \mathrm{sh}_\frakp$ to $X\times \{0\}\cong X$ by $\mu \mathrm{sh}_\frakp$. For example, a section of the Lagrangian Grassmannian bundle of the stable symplectic tangent bundle of $X$ gives a Maslov data.
\end{example}

Let $(X,\lambda)$ be an exact symplectic manifold.
Fix a Maslov data $\frakp$. The dg category of microsheaves $\mu \mathrm{sh}_\frakp(X)$ is defined as the global section of $\mu \mathrm{sh}_{\frakp}$, so each object has the notion of support in $X$. We denote the subsheaf of $\mu \mathrm{sh}_\frakp$ spanned by the objects supported on $L$ by $\mu \mathrm{sh}_{\frakp,L}$.
One can recover the usual sheaves as follows:
\begin{theorem}[\cite{nadlershende}] For $X=T^*M$, there exists a Malsov data $\frakp_{f}$ such that $\mu \mathrm{sh}_{\frakp_{f}}(X)\cong \Sh(X,\bZ)$.
\end{theorem}

\subsection{Relation to Fukaya category}

Let $(X,\lambda)$ be an exact symplectic manifold. The Liouville vector field $v$ is defined by $d\lambda(v,-)=\lambda$. The \emph{core} $\Core(X)$ is defined by the set of the points in $X$ not escaping to infinity under the 
 Liouville flow ($=:$ the flow of the Liouville vector field).

If one further assumes that $(X,\lambda)$ is Liouville, then it means that there exists a compact submanifold with boundary $X_0$ such the complement of $X_0$ has a cylindrical form $\partial X_0 \times \bR_{>0}$ under the Liouville flow. A subset of $\partial X$ is called stop. For a fixed stop $\Lambda$, we define the \emph{relative core} $\Core(X,\Lambda)$ as the set of the points not escaping to the complement of $\Lambda$ under Liouville flow.

If one further assumes $X$ to be Weinstein, it implies that $\Core(X)$ is isotropic. If $\Lambda$ is a Legendrian, $\Core(X,\Lambda)$ is isotropic as well.

Fix a Maslov data $\frakp$.
We consider the subsheaf of $\mu \mathrm{sh}_{\frakp}$ whose objects are supported in $\Core(X,\Lambda)$, and denote it by $\mu \mathrm{sh}_{ \frakp,\Core(X,\Lambda)}$. This sheaf is supported on $\Core(X,\Lambda)$. The compact objects of the global sections is denoted by $ \mu \mathrm{sh}^w_{\frakp, \Core(X,\Lambda)}(\Core(X,\Lambda))$.

On the other hand, for the choice $\frakp$ and a stop $\Lambda$, we can define the partially wrapped Fukaya category of $X$ stopped at $\Lambda$ which is an infinity category. We denote it by $\cW(X, \frakp, \Lambda)$.

\begin{theorem}[Ganatra--Pardon--Shende~\cite{GPSmicro}]
Let $X$ be a Weinstein manifold and $\Lambda$ be a Legendrian in $\Lambda\subset \partial X$. Fix a Maslov data $\frakp$.
We have an equivalence of $\infty$-categories:
    \begin{equation}
        \mathrm{GPS}\colon \Mod(\cW(X, \frakp,\Lambda)^{op}) \xrightarrow{\cong} \mu \mathrm{sh}_{\frakp, \Core(X,\Lambda)}(X) .
    \end{equation}
\end{theorem}

\subsection{Complex exact symplectic manifold and complex Lagrangians}
As we have seen, to define microsheaves, we have to choose a Maslov data. In this section, we consider the holomorphic setup, where the choice is canonical. We follow the explanation of \cite{CKNSmicroPerverse, CKNSmicroRH}.

Let $V$ be a complex contact manifold. We denote the complex symplectization $\pi\colon \widetilde V\rightarrow V$, which is a $\bC^*$-fibration. It is observed in \cite{CKNSmicroPerverse} that $\widetilde V$ carries a canonical null homotopy of $\widetilde V\rightarrow B^2\bZ$. As a result, we obtain a canonical choice of microsheaf category $\mu \mathrm{sh}_{\frakp_{can}}$ on $\widetilde V$. We denote the subsheaf of $\pi_*\mu \mathrm{sh}_{\frakp_{can}}$ consisting of pointwise representable by sheaves by $\bP\mu \mathrm{sh}_V$ (see \cite{CKNSmicroPerverse} for the definition). Here and after we omit $\frakp_{can}$ from the notations.

In \cite{CKNSmicroRH}, this unique canonical Maslov data is presented as a concrete gluing data as follows: Fix a Darboux covering $V=\bigcup U_\alpha$ with a contact embedding $U_\alpha\hookrightarrow \bP X_\alpha$ for some complex projectized tangent bundle $\bP X_\alpha$ for each $\alpha$. For each cover $U_\alpha$, we consider the twisted microsheaf $\bP\mu \mathrm{sh}^{w_2(U_\alpha)}_{U_\alpha}$ twisted by the second Stiefel--Whitney class $w_2(U_\alpha)$. On each overlapping region, we have a complex contact morphism, which carries a canonical quantization. By gluing up along these quantized transformation, we obtain our (complex) microsheaf category. 

By the above reason, when we write a microsheaf category for some complex geometric setup, we omit the Maslov data from the notation.

Now we consider a holomorphic exact symplectic manifold $X$. Namely, it is a complex manifold equipped with a holomorphic $1$-form $\lambda$ such that $d\lambda$ is a holomorphic symplectic form. In this situation, $V=X\times \bC$ is canonically a complex contact manifold. Then we can define $\bP\mu \mathrm{sh}_X:=\bP\mu \mathrm{sh}_V|_{X\times \{0\}}$ canonically.

If one further assume that the Liouville vector field of $X$ is integrated into a $\bC^*$-action, then $\bP\mu \mathrm{sh}_X$ admits a $\bZ[t^\pm]$-action and we can specialize it to $t=1$. We denote the specialization by $\mu_\bC \mathrm{sh}_X$, which has an embedding into $\mu \mathrm{sh}_X$.

\begin{remark}
    Although the gluing description of the canonical microsheaf category involves twisting by $w_2$, our actual computation happens for (some combinations of) complex projective spaces where $w_2$ are zero. For this reason, we tacitly suppress $w_2$ from our notation.
\end{remark}

\subsection{Hodge microsheaves} \label{jan7-2}
Here we give a definition of the category of microsheaves by using gluing via Fourier transforms. A similar approach can be found in~\cite{McBreenWebster}.

Let $X$ be an $2n$-dimensional holomorphic exact symplectic manifold.
Suppose $L\subset X$ is a holomorphic Lagrangian subvariety with $\lambda|_L=0$. Then $L\times 0\subset X\times \bC$ is a holomorphic Legendrian. Then we denote the subsheaf of $\mu_\bC \mathrm{sh}_X$ consisting of objects supported on $\pi^{-1}(L\times 0)$ by $\mu_\bC \mathrm{sh}_L$. We similarly define $\bP \mu \mathrm{sh}_L$.

To describe $\bP\mu \mathrm{sh}_{L}$ and $\mu_\bC \mathrm{sh}_{L}$ in terms of Fourier transform, we consider the following condition.

\begin{definition}
     Let $L$ be a holomorphic Lagrangian subvariety in $X$ with $\lambda|_L=0$ of the form $\bigcup_{i\in I} L_i$ where 
    \begin{enumerate}
        \item  $I$ is a finite set and $L_i$ is a smooth Lagrangian submanifold,
        \item the intersections are clean, and
        \item $\bigcap_{i\in J}L_i$ is codimension $|J|$ in $L$ for any $J\subset I$.
    \end{enumerate}
    We say $L$ is a \emph{Lagrangian core of Fourier type}.
\end{definition}

Let $L=\bigcup_{i\in I} L_i$ be a Lagrangian core of Fourier type. We consider the following category $\cC_I$: An object is a pair $(i, J)$ with $J\subset I$ and $i\in J$. The morphisms are generated by the following:
\begin{enumerate}
    \item For $i, i'\in J$ with $i\neq i'$, mutually inverse isomorphisms $f_{ii'}\colon (i, J)\leftrightarrow (i', J)\colon f_{i'i}$.
    \item For $J\subset J'\subset I$, a morphism $r_{JJ'}\colon (i,J) \rightarrow (i, J')$. 
\end{enumerate}
Note that $\bigcap_{j\in J}L_j$ is a smooth $(n-|J|)$-dimensional submanifold of $L_i$ for any $i\in J\subset I$ by the assumption. For each $L_i$ $(i\in I)$, we consider the following stratification $\cS_i$: We define it inductively from low dimensional strata. The $\ell$-dimensional strata is defined by the complement of $(\ell-1)$-dimensional strata in $\bigcap_{i\in J} L_i$ with $|J|=n-\ell$. For $i\in J$, we also denote the induced stratification on $\bigcap_{i\in J}L_i$ by $\cS_J$.

For each $J$ and $j\in J$, we consider the normal bundle $\pi_{(j,J)}\colon T_{\bigcap_{i\in J}L_i}L_j\rightarrow \bigcap_{i\in J}L_i$. We also have a stratification $\cS_{(j,J)}:=\pi_{(j,J)}^{-1}(\cS_J)$ on $T_{\bigcap_{i\in J}L_i}L_j$. 
We consider the following assignment: For each $(j, J)$, we assign the category of sheaf $\Sh(T_{\bigcap_{i\in J}L_i}L_j, \cS_{(j,J)})$. For $f_{jj'}$, we assign the Fourier transform $\Sh(T_{\bigcap_{i\in J}L_i}L_j,\cS_{(j,J)})\rightarrow \Sh(T_{\bigcap_{i\in J}L_i}L_{j'},\cS_{(j',J)})$. For $r_{JJ'}$, we assign the specialization functor $\Sh(T_{\bigcap_{i\in J}L_i}L_j, \cS_{(j,J)})\rightarrow \Sh(T_{\bigcap_{i\in J'}L_i}L_{j}, \cS_{(j,J')})$. These together define a diagram in the category of categories. Taking the limit, we obtain a category.

One can see the obtained category is equivalent to $\mu_\bC \mathrm{sh}_{L}$ as follows: Each $\bigcap_{i\in J}L_i$ admits a tubular neighborhood $U_J$ such that $\mu \mathrm{sh}_L(U_J)\cong \Sh(T_{\bigcap_{i\in J}L_i}L_j, \cS_{(j,J)})$. If one shrinks more, it becomes a Darboux coordinate whose position variables are on $L_j$.

Each such Darboux coordinate $U$ induces a contact Darboux coordinate $U\times \bC$ on $X\times \bC$. Moreover, $\mu_\bC \mathrm{sh}_{\pi^{-1}(L)} (U\times \bC)\cong \mu \mathrm{sh}_L(U)$. Also, each Fourier transform induces a quantized contact transform. Hence it recovers the gluing description of \cite{CKNSmicroRH} explained in the last section. In particular, we have $\mu_\bC \mathrm{sh}_L(X)\cong \mu \mathrm{sh}_L(X)$.

By replacing $\Sh$ with $\sqrt{\MHM}$, we define $\mu \MHM_L(X)$. This is the category of \emph{Hodge microsheaves}.
Like Proposition~\ref{prop:mhmismix}, we obtain the following.
\begin{proposition}
    The following data induce a mixed structure on the category $\mu\mathrm{sh}_L(X,\bQ)$:
    \begin{enumerate}
        \item Tate twist functor $\lb \frac{1}{2}\rb :=(-)\otimes \bQ\lb\frac{1}{2} \rb$.
        \item The forgetful functor $\frakF\colon \mu\MHM_L(X) \to \mu\mathrm{sh}_L(X, \bQ)$.
    \end{enumerate}
\end{proposition}

As we have remarked in section~\ref{section:Infinite-dimensionalHodge}, we sometimes use a smaller model satisfying saturatedness.
We use the following convention:
\begin{definition}[Saturated system]
For a Lagrangian core of Fourier type $L$ of $X$, a \emph{saturated system} is an assignment of saturated (non-full) subcategories $M_{(j,J)}^c\subset\sqrt{\MHM^c(T_{\bigcap_{i\in J}L_i}L_j, \cS_{(j,J)})}$ satisfying the following: the morphisms appeared in the diagram of the definition of $\mu\MHM_L(X)$ can be restricted to the subcategories $M_{(j,J)}^c$'s.
\end{definition}
For a saturated system, we can glue similarly. We denote such a glued category $\mu \mathrm{M}^c_L(X)$ (for specified $M_{(j,J)}^c$).
\begin{lemma}\label{lem:saturatedmicro}
For a saturated system, the category is $\mu \mathrm{M}^c_L(X)$ is also saturated.
\end{lemma}

\subsection{Example: $A_n$-plumbing of $T^*\bP^1$}
Here we exemplify the construction in the last section of the case of $A_n$-plumbing of $T^*\bP^1$. This is the derived version of \cite{BezrukavnikovKapranov}. See also \cite{karabas2024wrappedfukayacategoryplumbings}.

Let $X$ be the $A_n$-plumbing of $T^*\bP^1$. For its definition, we refer to \cite{EtguLekili} and the references therein. It is also realized as the minimal resolution of $A_n$-singularity, so it carries a structure of holomorphic exact symplectic manifold.

The core $C$ of it is described as follows. We have $n$ copies $\{\bP_i\}_{i=1,...,n}$ of $\bP^1$ with three different marked points $l_i, r_i,m_i$ for each $\bP_i$. We prepare the set $\{p_i\}_{i=1,...,n-1}$ and consider the pushout
\begin{equation}
    C=\bP_1\cup_{p_1}\bP_2\cup_{p_2}\cdots \cup_{p_{n-1}}\bP_n
\end{equation}
with the maps $p_i\rightarrow n_i$ and $p_i\rightarrow s_{i+1}$. This is the core. We also consider the relative core
\begin{equation}
    C_{\lc \boldsymbol{m}\rc}:=C\cup \bigcup_i T^*_{m_i}\bP^1_i.
\end{equation}

We then consider the homotopy pull-back diagram
\begin{equation}
\begin{split}
    &\mu \mathrm{sh}_C(X)\\
    &=\Sh(\bP_1, \{n_1\})\times_{\Sh(\bP_1\bs\{s_1\}, \{n_1\})}\Sh(\bP_2,\{n_2, s_2\})\times_{\Sh(\bP_2\bs \{s_2\}, \{n_2\})}\cdots \times_{\Sh(\bP_{n-1}\bs\{s_{n-1}\},\{n_{n-1}\} )}\Sh(\bP^1_n, \{s_{n}\})\\
    &=: \cB_1\times_{\cC_1}\times \cB_2\times_{\cC_2}\cdots \times_{\cC_{n-1}}\cB_n
\end{split}
\end{equation}
with respect to the Dwyer--Kan model structure of the category of dg-categories.
Here $\Sh(\bP_i, \{n_i,s_i\})\rightarrow\Sh(\bP_i\bs\{s_i\}, \{n_i\})$ is given by the specialization, and  $\Sh(\bP_{i+1}, \{n_{i+1}, s_{i+1}\})\rightarrow\Sh(\bP_i\bs\{s_i\}, \{n_i\})$ is given by the composition of the restriction and Fourier transform.
Recall that any object (i.e., any dg category) is fibrant in this model structure. So, all the categories involved here are fibrant. To compute the homotopy pull-back, we need to replace $\cB_i\rightarrow \cC_i$ by a fibration $\cA_i\rightarrow \cC_i$ with a quasi-equivalence $\cB_i\rightarrow \cA_i$ for each $i<n$. We also set $\cA_n=\cB_n$.
Hence the desired homotopy pull-back is the genuine pull-back
\begin{equation}\label{eq:genuinepullback}
    \cA_1\times_{\cC_1}\cA_2\times_{\cC_2}\cdots \times_{\cC_{n-1}}\cA_n.
\end{equation}

Now let us describe an object in the homotopy pull-back. We will see that a set of objects $\{b_i\in \cB_i\}$ with $b_i|_{\cC_i}\simeq b_{i+1}|_{\cC_i}$ defines an object of the homotopy pull-back. In fact, by the functor $\cB_i\rightarrow \cA_i$, we denote the image of $b_i$ by $a_i$. By the definition of the fibration in this model structure, we have a lift of an isomorphism $a_2|_{\cC_1}\rightarrow a_1|_{\cC_1}$ to $a_2\rightarrow a_2'$ in $\cA_2$. We next consider an isomorphism $a_3|_{\cC_2}\rightarrow a_2'|_{\cC_2}$ and lift it to $a_3\rightarrow a_3'$. Repeating this, we obtain objects $a_1, a_2', a_3',...$ which defines an object of the genuine pull-back (\ref{eq:genuinepullback}). Similarly, we can compute $\mu\mathrm{sh}_{C_{\lc\boldsymbol{m}\rc}}(\bP^1)$.

The above argument also works well for the category of microsheaves.
\begin{example}\label{ex:saturatedHodgesheafforplumbing}
    By using saturated model $M$ of \ref{ex:Hodge-Tate2} of Example~\ref{ex:Hodge-Tate}, we obtain a saturated mixed structure on $\mu \mathrm{sh}_{C_{\lc \boldsymbol{m}\rc}}(X)$. We denote it by $\mu \mathrm{M}^c_{C_{\lc \boldsymbol{m}\rc}}(X)$.
\end{example}

\section{Hodge wrapping}\label{section:Hodgewrapping}
\subsection{Wrappings}
Let $M$ be a manifold.
Let $\Lambda_1\subset \Lambda_2$ be conic Lagrangians in $\pi_M\colon T^*M\rightarrow M$. For a compact object $\cE_2\in \Sh_{\Lambda_2}(M)$, we can consider the functor $\Hom_{\Sh(M)}(\cE_2, -)\colon \Sh_{\Lambda_1}(M)\rightarrow \Mod(\bC)$.
Note that $\cE_2\not\in \Sh_{\Lambda_1}(X)$ in general. However,  by the representability theorem, the functor is corepresented by a compact object $\cE_1$ in $\Sh_{\Lambda_1}(X)$. We say $\cE_1$ is the \emph{wrapping of $\cE_2$}. This point of view was first introduced by Nadler~\cite{nadler2016wrappedmicrolocalsheavespairs}. Later, it is identified with geometric wrapping by \cite{Kuo, GPSmicro}.

Since we have an obvious inclusion $\rho\colon \Sh_{\Lambda_1}(X)\subset \Sh_{\Lambda_2}(X)$ and its left adjoint by the adjoint functor theorem, we can also say that $\cE_1$ is the image under the left adjoint $\rho^l$. This left adjoint is called \emph{stop removal functor} in \cite{GPSmicro}.

For a point $p\in \Lambda$, one can define the microstalk functor $\Sh_{\Lambda}(X)\rightarrow \Mod(\bC)$, which estimates the microsupport of the objects at $p$. For example, we can construct as follows: Take a small open subset $U$ of $\pi_M(p)$ and a smooth function $f$ on it with $f(\pi(p))=0$ and $df$ intersects $\Lambda$ transversely. Then the functor is $\Hom(\bC_{\lc x\in U\relmid f(x)\geq 0 \rc},-)$.

Again, $\bC_{\lc x\in U\relmid f(x)\geq 0 \rc}\not\in \Sh_{\Lambda}(X)$ in general. The corresponding wrapping in $\Sh_{\Lambda}(X)$ is called a microlocal skyscraper sheaf (or microstalk, for short). This was first introduced by Nadler~\cite{nadler2016wrappedmicrolocalsheavespairs}. The following explains the importance of microlocal skyscraper sheaves:
\begin{proposition}[\cite{nadler2016wrappedmicrolocalsheavespairs}]
    The set of microskysraper sheaves in $\Sh_{\Lambda}(X)$ is a set of compact generators.
\end{proposition}
In the context of wrapped Fukaya category~\cite{GanatraSectorial}, microlocal skyscrper sheaves correspond to cocores/linking disks.

Since the microlocal skyscraper sheaves together generate the whole category $\Sh_{\Lambda}(X)$, considering Hodge version should be important as well:
\begin{definition}
Let $X$ be a complex manifold.
    Let $M_1\subset M_2\subset \MHM(X)$ be subcategories. An object $\cE_1\in M_1$ is called the $(M_1,M_2)$-\emph{Hodge wrapping} of $\cE_2\in M_2$ if there exists a functor isomorphism
    \begin{equation}
        \Hom_{M_1}(\cE_1,-)\cong \Hom_{M_2}(\cE_2,-)|_{M_1}.
    \end{equation}
\end{definition}
Then it is natural to ask the relationship between Hodge wrapping and usual wrapping. In this section, we explore this question.

\subsection{Conjectures on complex wrappings}
Here we recall some ideas considering wrapping using sheaf quantizations of Hamiltonian isotopies.

Let $M$ be a differentiable manifold. We fix a Riemannian metric $g$ on $M$. We denote the geodesic flow by $\phi_t^g$, which is a contact isotopy on the cosphere bundle $S^*M$. Equivalently, $\phi_t^g$ is a conic Hamiltonian isotopy on $T^*M\bs T^*_MM$ where $T^*_MM$ is the zero section.

Recall that the derived category of $\bC$-valued sheaves is denoted by $\Sh(M)$. By the work of Guillermou--Kashiwara--Schapira, one can ``quantize" $\phi_t^g$. Namely, the following holds:
\begin{theorem}[\cite{GKS}]
Let $\phi_t$ be any conic Hamiltonian isotopy on $T^*M\bs T^*_MM$.
    There exists an autoequivalence $\Phi_t$ of $\Sh(M)$ such that $\SSr(\Phi_t(\cE))\bs T^*_MM=\phi_t(\SSr(\cE)\bs T^*_MM)$ holds for any $\cE\in \Sh(M)$ where $\SSr$ is microsupport.
\end{theorem}
In the following, we only consider the case when $\phi_t=\phi_t^g$ for some $g$.
In the case of $\phi_t^g$, we have a canonical map $\Phi_t\rightarrow \Phi_{t'}$ for any $t\leq t'$, called the continuation map (see, e.g. \cite{Kuo}). Hence $\{\Phi_t\}$ forms an inductive system.
\begin{theorem}[\cite{Kuo}]
For conic Lagrangians $\Lambda_1=T^*_MM\subset \Lambda=\Lambda_2$, the wrapping $\cE_1\in \Sh_{\Lambda_1}(M)$ of $\cE_2\in \Sh_{\Lambda_2}(M)$ is isomorphic to $\colim_t\Phi_t(\cE_2)$. 
\end{theorem}

Now let us consider the case when $M$ carries a complex structure. In this setup, inside $\Sh(M)$, we have the derived category of $\bC$-constructible sheaves $\Sh_{\bC-c}(M)$. Now we can formulate our conjecture. 

\begin{conjecture}\label{conj:1}
Let $M$ be a compact complex manifold. Let $\cE$ be an object of $\Sh_{\bC-c}(M)$. There exists an inductive system $\{\cE_i\}$ in $\Sh_{\bC-c}(M)$ such that $\lim_{t\rightarrow +\infty}\Phi_{t}(\cE)=\lim_{i\rightarrow +\infty}\cE_i$. 
\end{conjecture}

In particular, we are interested in the following case:
\begin{conjecture}
Let $M$ be a compact complex manifold and $x$ be a point in $M$. There exists an inductive system $\{\cE_i\}$ in $\Sh_{\bC-c}(M)$ such that $\lim_{t\rightarrow +\infty}\Phi_{t}(\bC_x)=\lim_{i\rightarrow +\infty}\cE_i$. 
\end{conjecture}
One possible approach to this conjecture is the following: Describe $\{\Phi_t(\bC_x)\}$ explicitly by using an explicit metric, and find a sequence $t_1,t_2, \cdots \rightarrow +\infty$ such that $\Phi_{t_i}(\bC_x)\in \Sh_{\bC-c}(M)$. This approach is taken in \cite{Arai} by Takumi Arai.

We can also consider a microsheaf version of the conjecture. When $X$ is a holomorphic exact symplectic manifold, C\^ot\'e--Kuo--Nadler--Shende~\cite{CKNSmicroPerverse} introduced a t-structure ``perverse microsheaves" on $\mu \mathrm{sh} (X)$. We denote the finite microstalk part of the triangulated closure of it by $\mu \mathrm{sh}_{\bC-c}(X)$.

\begin{conjecture}
Suppose $X$ is a holomorphic exact symplectic manifold with compact core. 
    Let $L$ be a complex Lagrangian submanifold. Then there exists an inductive system $\{\cE_i\}$ in $\mu \mathrm{sh}_{\bC-c}(X)$ such that $\lim_{i\rightarrow +\infty}\cE_i\cong \mathrm{GPS}(L)$.
\end{conjecture}
\begin{remark}
    The above conjecture seems to be false in general, see Appendix. One can consider the conjecture as a guiding principle.
\end{remark}

When $X=T^*M$, $L=T^*_NM$, we have $\mathrm{GPS}(L)=\lim_{t\rightarrow +\infty}\Phi_t(\bK_N)$ by the results of \cite{GPSmicro} and Kuo~\cite{Kuo}. Hence this conjecture is a generalization of Conjecture~\ref{conj:1}.

We now explain our motivation for the above conjectures from Hodge theory.

\begin{conjecture}[A naive (non-mathematical) conjecture]\label{conj:Hodge}
For any complex Lagrangian $L$, the object $\mathrm{GPS}(L)$ has a lift to a Hodge microsheaf.
\end{conjecture}

Conjecture \ref{conj:Hodge} is related to Conjecture \ref{conj:1} as follows. By the result of Jin~\cite{Jin}, a complex Lagrangian in $T^*M$ defines an object $\cE$ of $\Sh_{\bC-c}(M)$. Suppose Conjecture~\ref{conj:1} holds for $\cE$. Then we have an inductive system $\{\cE_i\}\in \Sh_{\bC-c}(M)$. We then find/lift the sequence to $\mathrm{MHM}(M)$ (of course this step is also nontrivial and conjectural). Then take the colimit in the ind-category of a certain subcategory of $\mathrm{MHM}(M)$, which will give us a desired object conjectured in Conjecture \ref{conj:Hodge}. In this sense, confirming Conjecture \ref{conj:1} will give a strong evidence for Conjecture \ref{conj:Hodge}.

In the following sections, we explain the case such a Hodge lift of the wrapping is actually the Hodge wrapping.

\subsection{General machinery}
We introduce several preparatory lemmas for the next section.
We first discuss the saturatedness in infinite-dimensional setup.
\begin{lemma}\label{lem:indsaturation}
Let $(\frakF\colon \widehat\cD\to \cD, (d),\epsilon)$ be a mixed structure on a triangulated category. Let $\widehat\cD^c$ be a triangulated subcategoryof $\widehat \cD$. We suppose the following:
    \begin{enumerate}
    \item Both $\cD$ and $\widehat \cD$ are cocomplete.
        \item $\frakF$ commutes with coproducts.
        \item Any object $E\in \widehat\cD^c$ is saturated and compact in $\widehat\cD$, and $\frakF(E)$ is also compact in $\cD$.
    \end{enumerate}
Then any $E\in \widehat\cD^c$ is saturated in the colimit closure of $\widehat\cD^c$.
\end{lemma}
\begin{proof}
For $E\in \widehat\cD^c$ and an object $\colim_i G_i\in \widehat \cD$ with $G_i\in \widehat\cD^c$ in the colimit closure, 
we have 
    \begin{equation}
    \begin{split}
        \bigoplus_j\Hom_{\widehat\cD}(E, \colim_i G_i(j))&\cong\bigoplus_j\colim_i\Hom_{\widehat\cD}(E,  G_i(j)) \\
        &\cong \colim_i\bigoplus_j\Hom_{\widehat\cD}(E,  G_i(j)) \\
        &\cong\Hom_{\cD}(\frakF(E), \colim_i\frakF(G_i))
    \end{split}
    \end{equation}
    where the first equality follows from 3, the second equality is obvious, and the third equality follows from 2 and 3.
\end{proof}

\begin{lemma}\label{lem:limitsaturation}
Let $(\frakF\colon \widehat\cD\to \cD, (d),\epsilon)$ be a mixed structure on a triangulated category. 
    Suppose 
    \begin{enumerate}
    \item Both $\cD$ and $\widehat \cD$ are cocomplete.
        \item $E_i\in \widehat\cD$ is saturated and compact,
        \item the colimit $E=\colim_i E_i$ is compact in $\widehat \cD$, and
        \item $\frakF$ commutes with colimit.
    \end{enumerate}
    Then $E=\colim_i E_i$ is saturated as well.
\end{lemma}
\begin{proof}
For $G\in \widehat\cD$, we have
        \begin{equation}
        \begin{split}
        \bigoplus_j\Hom_{\widehat\cD}(E, G(j))
        &\cong \Hom_{\widehat\cD}(E, \bigoplus_jG(j))\\
        &\cong \lim_i\Hom_{\widehat\cD}(E_i, \bigoplus_jG(j))\\
        &\cong \lim_i\bigoplus_j\Hom_{\widehat\cD}(E_i, G(j))\\
        &\cong \lim_i\Hom_{\widehat\cD}(\frakF(E_i), \frakF(G))\\
        &\cong \Hom_{\widehat\cD}(\colim_i\frakF(E_i), \frakF(G))\\
        &\cong \Hom_{\cD}(\frakF(E), \frakF(G))
        \end{split}
    \end{equation}
    where the first equality follows from the compactness of $E$, the second follows from the definition of $E_i$, the third follows from the compactness of $E_i$, the fourth follows from the saturatedness of $E_i$, the fifth is obvious, and the sixth follows from 4.
    This completes the proof.
\end{proof}
To verify 2 in Lemma~\ref{lem:limitsaturation} in practice, we need the following lemmas.

\begin{lemma}\label{colimitcompact}
Let $(\frakF\colon \widehat\cD_2\to \cD, (d),\epsilon)$ be a mixed structure on a triangulated category. Let $\widehat \cD_1$ to be a triangulated full subcategory of $\widehat\cD_2$. For $E\in \widehat \cD_1$, suppose that
    \begin{enumerate}
        \item $E$ is represented by $E=\Cone(E''\rightarrow E')$ in $\widehat{\cD} _2$ such that $E'$ is compact.
        \item $E''=\colim_i E_i''$ with $\Hom_{\cD}(\frakF(E_i''), -)|_{\frakF(\widehat \cD_1)}=0$ and $E_i''$ is saturated. 
    \end{enumerate}
    Then $E$ is compact in $\widehat \cD_1$.
\end{lemma}
\begin{proof}
We first note that
\begin{equation}
    \Hom_{\widehat\cD_2}(E''_i, -)|_{\widehat\cD_1}\subset \Hom_{\cD}(\frakF(E''_i), -)|_{\frakF(\widehat\cD_1)}=0
\end{equation}
by 2. Hence $\Hom_{\widehat\cD_2}(E'', -)|_{\widehat\cD_1}=0$ as well. Then we have $\Hom_{\widehat\cD_1}(E,-)=\Hom_{\widehat\cD_2}(E',-)|_{\widehat \cD_1}$. The compactness of $E'$ completes the proof.
\end{proof}

\subsection{Hodge wrapping versus wrapping}
Now let's go back to the sheaf-theoretic setup.

Let $\Lambda_1\subset \Lambda_2$ be conic Lagrangians in $T^*X$. We take a saturated (non-full) subcategory $M_2^c\subset \MHM^c_{\Lambda_2}(X)$ and set $M_2:=\Ind (M_2^c)$. We take $M_1:=M_2\cap \MHM_{\Lambda_1}(X)$. We then have two mixed structures $\frakF_1\colon M_1\rightarrow \Sh_{\Lambda_1}(X)$ and $\frakF_2\colon M_2\rightarrow \Sh_{\Lambda_2}(X)$. We record the following.
\begin{lemma}
    \begin{enumerate}
        \item $M_1$ and $M_2$ are cocomplete.
        \item $M_2^c$ is saturated in $M_2$.
        \item The functors $\frakF_1$ and $\frakF_2$ are both exact.
    \end{enumerate}
\end{lemma}
\begin{proof}
    1 and 3 are obvious. 2 follows from Lemma~\ref{lem:indsaturation}.
\end{proof}

As noted in the above, we have the left adjoint $\rho^l$ of the obvious inclusion $\rho\colon \Sh_{\Lambda_1}(X)\rightarrow \Sh_{\Lambda_2}(X)$. By the adjoint, we have
\begin{equation}
    \id \in \Hom_{\Sh_{\Lambda_1}(X)}(\rho^l(\cE), \rho^l(\cE))\cong \Hom_{\Sh_{\Lambda_2}(X)}(\cE, \rho^l(\cE))\ni u_\cE
\end{equation}
for any $\cE\in \Sh_{\Lambda_2}(X)$.

\begin{lemma}\label{lem:microstalkrep}
Let $E_1\in M_1$ and $E_2\in M_2$. Suppose
\begin{enumerate}
    \item $\frakF_1(E_1)=\rho^l(\frakF_2(E_2))$.
    \item There exists a morphism $f\colon E_2\rightarrow E_1\in M_2$ such that $\frakF_2(f)=u:=u_{\frakF_2(E_2)}$.
    \item $E_1$ (resp. $E_2$) is saturated in $M_1$ (resp. $M_2$).
\end{enumerate}
Then we have
\begin{equation}
    \Hom_{M_1}(E_1,-)\cong \Hom_{M_2}(E_2,-)|_{M_1}
\end{equation}
i.e., $E_1$ is the $(M_1,M_2)$-Hodge wrapping of $E_2$.
\end{lemma}
\begin{proof}
For any $G\in M_1$, we have the following diagram:
\begin{equation}
    \xymatrix{
    \bigoplus_j\Hom_{M_{1}}(E_1, G(j)) \ar[r]^{\bigoplus_j \Hom(f,G(j))}\ar[d]^{\cong}&\bigoplus_j\Hom_{M_{2}}(E_2, G(j))\ar^{\cong}[d]\\
    \Hom_{\Sh_{\Lambda_1}(X)}(\rho^l(\frakF(E_2)), \frakF(G)) \ar[r]^{\cong}_{\Hom(u, \frakF(G))}&\Hom_{\Sh_{\Lambda_2}(X)}(\frakF(E_2), \frakF(G))\\
    }
\end{equation}
The right vertical arrow is an isomorphism by the saturatedness 3 of $E_2$. The left vertical arrow is an isomorphism by the saturatedness 3 of $E_1$ and 1. The diagram is commutative by 2. This yields the isomorphism
\begin{equation}
    \Hom_{M_{1}}(E_1, G)\xrightarrow{\circ f}\Hom_{M_{2}}(E_2, G).
\end{equation}
This completes the proof.
\end{proof}

\begin{corollary}
Suppose the setting of Lemma~\ref{lem:microstalkrep}.
If we have a left adjoint $\rho^l_M\colon M_2\rightarrow M_1$ of the obvious inclusion $\rho_M\colon M_1\subset M_2$, then $E_1\cong \rho^l_{M}(E_2)$.
\end{corollary}

Now we give a more concrete theorem:
\begin{theorem}\label{theorem:Hodgewrapping}
Take $\cE'\in M_2^c$.
Suppose there exists an inductive system of objects $\cE_i$ in $M_2^c$ such that
\begin{enumerate}
    \item  $\cE_0=\cE'$,
    \item  $\frakF(\colim_i \cE_i)=\rho^l(\frakF(\cE'))$, and 
    \item $\Hom(\Cone(\frakF(\cE')\rightarrow \frakF(\cE_i)),-)|_{\Sh_{\Lambda_1}(X)}=0$.
\end{enumerate}
Then $\cE:=\colim_i \cE_i\in M_1$ is the $(M_1,M_2)$-Hodge wrapping of $\cE'$.
\end{theorem}
\begin{proof}
We first prove that $\cE$ is compact in $M_1$. 
We set $\cE_i'':=\Cone(\cE'\to \cE_i)[-1]$ and $\cE'':=\colim_i\cE_i''$. Then we have $\cE=\Cone(\cE''\to \cE')$. By the assumptions $\cE'$ is compact in $M_2$. By 3, $\Hom(\Cone(\frakF(\cE')\rightarrow \frakF(\cE_i)),-)|_{\frakF_1(M_1)}=0$.
Also, $\cE_i''\in M_2^c$ means it is saturated. Hence, all the conditions of Lemma~\ref{colimitcompact} are satisfied, we conclude that $\cE$ is compact.

Then it follows from Lemma~\ref{lem:limitsaturation}, $\cE$ is saturated. Then by Lemma~\ref{lem:microstalkrep}, we conclude that $\cE$ is the Hodge wrapping of $\cE''$. 
\end{proof}

The theorem above together with a strong version of Conjecture~\ref{conj:1} gives an effective way to find Hodge wrappings as follows:
\begin{corollary}\label{cor:Hodgewrapping}
Suppose $\Lambda_1=\varnothing$.
Fix a metric $g$ on $X$. 
For an object $\cE'\in M_2^c$, suppose that 
\begin{enumerate}
    \item there exists $t_0=0<t_1<t_2<...\to \infty$ such that $\Phi_{t_i}(\frakF(\cE'))\in \Sh_{\bC-c}(X)$ for any $i$,
    \item there exists an inductive system $\cE'=\cE_0\to \cE_1\to \cE_2\to \cdots \in M_2$ such that $\frakF(\cE_i)=\Phi_{t_i}(\cE')$ for any $i$.
\end{enumerate}
Then $\colim_i \cE_i$ is the $(M_1,M_2)$-Hodge wrapping of $\cE'$.
\end{corollary}
\begin{proof}
We check the conditions in Theorem~\ref{theorem:Hodgewrapping}. 1. 
The object $\cE'$ is assumed to be compact in $M_2$. 2. Since the functor $\frakF$ preserves colimits, we have
\begin{equation}
    \frakF(\colim_i \cE_i)=\colim_i \Phi_{t_i}(\frakF(\cE_i))=\rho^l(\frakF(\cE'))
\end{equation}
where the last equality is by \cite{Kuo}. 3. This is a general property of the wrapping. Hence the conditions in Theorem~\ref{theorem:Hodgewrapping} are satisfied, and the proof is complete.
\end{proof}

\subsection{The case of Hodge microsheaves}

Let $X$ be a holomorphic exact symplectic manifold and $L_1\subset L_2$ be Lagrangian cores of Fourier type. 
We fix a saturated system for $L_2$. Then we obtain the category $\mu \mathrm{M}_{L_2}^c(X)$. We set $\mu \mathrm{M}_2:=\Ind(\mu \mathrm{M}^c_{L_2}(X))$.

We similarly set 
\begin{equation}
    \mu\mathrm{M}_1:=\lc \cE\in \mu \mathrm{M}_2\relmid \supp(\cE)\subset L_1\rc.
\end{equation}
We similarly define as follows:
\begin{definition}
    An object $\cE_1\in  M_1$ is the \emph{wrapping} of $\cE_2\in M_2$ if there exists and isomorphism $\Hom_{M_1}(\cE_1,-)\cong \Hom_{M_2}(\cE_2,-)|_{M_1}.$
\end{definition}
We can similarly prove the following:
\begin{theorem}\label{theorem:microHodgewrapping}
Take $\cE'\in M_2^c$. 
Suppose there exists a sequence of objects $\cE_i$ in $M_2^c$ such that
\begin{enumerate}
    \item  $\cE_0=\cE'$,
    \item  $\frakF(\colim_i \cE_i)=\rho^l(\frakF(\cE'))$, and 
    \item $\Hom(\Cone(\frakF(\cE')\rightarrow \frakF(\cE_i)),-)|_{\mu\mathrm{sh}_{L_1}(X)}=0$.
\end{enumerate}
Then $\cE:=\colim_i \cE_i\in M_1$ is the $(M_1, M_2)$-Hodge wrapping of $\cE''$.
\end{theorem}

\section{Hodge microsheaves and Hodge structure on loops}\label{section:loopspace}
In this section, we discuss a relationship between Hain's Hodge structure~\cite{Hain} on loop spaces with our Hodge modules.
\subsection{Hain's loop Hodge structure}
We first recall the bar construction. 

Let $X$ be a complex manifold. We denote the Sullivan model of $X$ over $\bK$ by $A^\bullet_\bK X$. 
\begin{theorem}[Chen~\cite{Chen}]
    If $X$ is simply connected, the bar construction $B(A^\bullet_\bK X)$ is quasi-isomorphic to the de Rham algebra of the based loop space $\Omega X$.
\end{theorem}
Based on this, Hain gives the following construction: 
\begin{theorem}[Hain~\cite{Hain}]
We assume that $X$ is a simply connected algebraic variety.
\begin{enumerate}
    \item There exists a mixed Hodge complex $M_\bK^\bullet X$ which is also a differential graded algebra (=multiplicative mixed Hodge complex) such that the cohomology recovers the mixed Hodge--de Rham algebra of $X$.
    \item There exists an extension of bar construction to any multiplicative mixed Hodge complex such that the result gives a mixed Hodge complex.
\end{enumerate}
\end{theorem}
As a result, we get a mixed Hodge structure on $H^\bullet(\Omega X)$. 

\subsection{Hodge structure from wrapping}
Now let us further argue Conjecture~\ref{conj:Hodge} when $L=T^*_xX$. In this case, $\mathrm{GPS}(L)=\lim_{t\rightarrow+\infty}\Phi_t(\bC_x)$. Combining the results of Kuo~\cite{Kuo} and Abouzaid~\cite{Abouzaid} (or Ganatra--Pardon--Shende~\cite{GPSmicro}), we have
\begin{equation}
    \End_{\Sh(X)}(\lim_{t\rightarrow +\infty}\Phi_t(\cE))\cong \End_{\cW(T^*X)}(T^*_xX)\cong C_{-*}(\Omega_xX)
\end{equation}
where the rightmost term is the dga of chains of the based loop space. 

For $\MHM_{T^*_XX\cup T^*_xX}(X)$, we fix a saturated model $\mathrm{M}(X)$.
Let $\widetilde \cE$ be the Hodge wrapping of $\bC_x$. Then $\End_{\Sh(X)}(\lim_{t\rightarrow +\infty}\Phi_t(\cE))$ also has a saturated Hodge version $\bigoplus_i\End_{\mathrm{M}(X)}(\widetilde\cE, \widetilde\cE(i))$ where $(i)$ is Tate twist. On the other hand, Hain~\cite{Hain} constructed a Hodge structure on $C^*(\Omega_xX)$ via bar construction when $X$ is simply connected. In this situation, we conjecture the following:
\begin{conjecture}\label{conj:LoopHodgeconj}
    \begin{equation}
        \End^{-n}_{\mathrm{M}(X)}(\cE, \cE(i))^*\cong C^n(\Omega_xX)_i
    \end{equation}
    where $ C^n(\Omega_xX)_i$ is $(i, n-i)$-piece of Hain's Hodge structure.
\end{conjecture}
\begin{remark}
    Probably, this conjecture does not hold for varieties with non--Hodge--Tate cohomology. To remedy this, we have to modify the category (or the notion of saturation) so that it is sensitive to non--Hodge--Tate cohomology.
\end{remark}

One possible way to attack this conjecture is to relate $ \End_{\Sh(X)}(\lim_{t\rightarrow +\infty}\Phi_t(\cE))$ to bar construction, which is currently not clear for the present authors.

\subsection{Examples}
In this section, we use the saturated mixed structure $M_2^c$ of \ref{ex:Hodge-Tate3} in Example~\ref{ex:Hodge-Tate} on $\Sh(\bP^n, \lc H,x\rc)$. We set $M_2:=\Ind(M_2^c)$ and set $M_1:=M_2\cap \MHM_{T^*_XX}(X)$.

\begin{theorem}
    There exists a $(M_1,M_2)$-Hodge wrapping of $\bC_x\in \MHM(\bP^n)$.
\end{theorem}
\begin{proof}
For $\bP^n$, the case $n=1$ can be seen from the computation in the later section. For $n>1$, one can adapt our argument to the result of \cite{Arai}. The result of \cite{Arai} is a stronger version of Conjecture~\ref{conj:1} summarized as follows: For the Fubini--Study metric $g$ on $\bP^n$, there exists a sequence $0=t_0<t_1<t_2<\cdots \to \infty$ such that $\Phi_{t_i}(\bC_x)\in \Sh_{\bC-c}(X)$. Moreover each $\Phi_{t_i}(\cE_i)$ admits the twisted complex presentation formed by
\begin{equation}\label{eq: sequence1}
    \bK_{\bP^n}, \bK_{\bP^n}[1], \bK_{\bP^n}[2n], \bK_{\bP^n}[2n+1], \bK_{\bP^n}[4n], \bK_{\bP^n}[4n+1],...
\end{equation}
with a twisted differential. The corresponding iterated extension is given by (1) the hyperplane class when the degree difference is 1 and, (2) the fundamental class when the degree difference is $2n-1$. Hence one can lift them to the corresponding map between mixed Hodge modules. In other words, one obtains a lift indicated in 2 of Corollary~\ref{cor:Hodgewrapping}, which completes the proof.
\end{proof}

By investigating the above proof furhter, we obtain the following.
\begin{theorem}
    Conjecture~\ref{conj:LoopHodgeconj} holds true for $X=\bP^n$.
\end{theorem}
\begin{proof}
Since (1) the hyperplane class has pure weight $2$ and (2) the fundamental class has pure weight $2n$, the lift of (\ref{eq: sequence1}) is formed by 
\begin{equation}
    \bK_{\bP^n}, \bK_{\bP^n}[1](1), \bK_{\bP^n}[2n](n+1), \bK_{\bP^n}[2n+1](n+2), \bK_{\bP^n}[4n](2n+2), \bK_{\bP^n}[4n+1](2n+3),....
\end{equation}    
Hence (degree, weight) is 
\begin{equation}
    (0,0), (1,2), (2n, 2n+2), (2n+1, 2n+4), (4n, 4n+4), (4n+1, 4n+6),....
\end{equation}

On the other hand, Hain's bar construction can be computed as follows: The de Rham algebra of $\bP^n$ is formal, hence we can consider it as $\bC[x]/x^{n+1}$ with $\deg x=2$ and the weight of $x=2$. We denote the $i$-th graded part by $A^i$. We set $\overline{A}^i=A^{i+1}$ ($i\geq 1$) and $\overline{A}^0=0$. We define the reduced bar complex by $\bigoplus_{i\geq 0}(\bigoplus \overline{A}^j)^{\otimes i}$. The grading is induced by the tensor grading and the differential is the usual bar differential. One can directly see that the cohomology class is represented by
\begin{equation}
    1, x, x\otimes x^n, x\otimes x^n\otimes x, x\otimes x^n\otimes x\otimes x^n ,x\otimes x^n\otimes x\otimes x^n\otimes x,....  
\end{equation}
Note that $x$ has $(\text{degree},\text{weight})=(1,2)$ and $x^n$ has $(\text{degree},\text{weight})=(2n-1,2n)$ in $\overline{A}$. Hence (degree, weight) of the bar construction is read as 
\begin{equation}
    (0,0), (1,2), (2n, 2n+2), (2n+1, 2n+4), (4n, 4n+4), (4n+1, 4n+6),....
\end{equation}
This agrees with the above computation, hence completes the proof.
\end{proof}

\section{Koszul duality in mixed geometry and symplectic geometry}\label{section:Koszuldualityreview}
\subsection{A quick review of Koszul duality}
Let $\mathbf{k}$ be a semisimple ring and $A=\bigoplus_{i\geq 0}A_i$ a graded $\mathbf{k}$-algebra with $A_0=\mathbf{k}$.
We say that $A$ is \emph{Koszul} in the classical sense if there exists a graded projective resolution of $\mathbf{k}$ (as a graded $A$-module)
\[\dots \to P_2\to P_1\to P_0\to \mathbf{k}\to 0\]
such that $P_i=AP_{i,i}$ for the degree $i$ part $P_{i,i}$ of $P_i$.
This condition is equivalent to the vanishing:
\[\mathrm{Ext}_A^k(\mathbf{k},\mathbf{k}(-s))=0\quad (k\neq s),\]
where $(-s)$ stands for the shift of the grading.
For a Koszul algebra,
we define the Koszul dual algebra of $A$ as
\begin{equation}
    A^!:=\Ext_{A}^*(\mathbf{k},\mathbf{k}).
\end{equation}
Then, it is known that $A^!$ is again Koszul and we have an isomorphism $(A^{!})^!\simeq A$. 
If $A$ is Koszul, we obtain the important conclusion:
there exists an equivalence of categories between the bouded derived categories of finitely generated $A$-modules and that of $A^!$-modules.

If one tries to generalize the notion of graded algebra to the dg-setup, there are two ways of thinking of it:
\begin{enumerate}
    \item One can consider the notion of differential graded algebra itself is a generalization of the notion of graded algebra.
    \item As a graded algebra is an algebra with a grading, one can consider a dga with another grading.
\end{enumerate}
According to these two ways, there are two ways to generalize the Koszul duality to the dg-setup.

In He-Wu~\cite{HeWu}, the generalization along the first line is described as follows:
For a (connected) dga $B=\bigoplus_{i\geq 0}B_i$ with $B_0=\mathbf{k}$,
we say that $B$ is \emph{Koszul} if 
\[\mathrm{Ext}_{B}^i(\mathbf{k},\mathbf{k})=0\quad (i\neq 0),\]
where $\mathrm{Ext}^i_{B}$ is the derived functor of the Hom-functor $\Hom_{B}$ in the category of dg $B$-modules, not in the category of graded $B$-modules.
Then, the same conclusion can be obtained for $B$ as above for $A$.
Remark that 
for a Koszul algebra $A$ if we regard it as a dga with a trivial differential,
then one can check that $A$ is a Koszul dga.

The generalization along the second line is also described in \cite{HeWu}, which will be recalled in the next section. In \cite{Luoetal}, this way is further generalized to $A_\infty$-algebras,
and they obtained similar consequences in this case.

In this paper, we use the second one: Koszulity for ``a dga with another grading (Adams grading)", explained in the next subsection.
Note that a Koszul algebra in the classical sense is also Koszul in the sense of the next section.

\subsection{He--Wu's dg Koszul duality}
For our application, we will use the result of \cite{HeWu}. We briefly recall their result here. Let $A=\bigoplus_{i,j\in \bZ}A^{i,j}$ be a graded dga with degree $(1,0)$ differential $d$. We call $j$ \emph{Adams grading}.

\begin{definition}
    We say $A$ is \emph{Adams connected} if the followings are satisfied:
    \begin{enumerate}
        \item $A^{i,j}=0$ for $i<0$,
        \item $A^{i,j}=0$ for $j<0$,
        \item $A^{i,0}=0$ for $i>0$, 
        \item $A^{0,j}=0$ for $j>0$, 
        \item $A^{0,0}=\mathbf{k}$.
    \end{enumerate}
    Note that we have the projection $A\rightarrow A^{0,0}=\mathbf{k}$ i.e., $A$ is augmented.
\end{definition}
We note the following simple lemma:
\begin{lemma}\label{lem:adams-connected}
    Let $A=\bigoplus_{i,j\in \bZ}A^{i,j}$ be a graded dga such that its cohomology algebra is Adams connected. Then there exists an Adams-connected dga quasi-isomorphic to $A$.
\end{lemma}
\begin{proof}
We first remark that $A$ is quasi-isomorphic to a dga satisfying 1 and 4 of the conditions. This follows from the fact that any cohomologically connected dga is quasi-isomorphic to a connected dga. We provide a sketch of the proof for the reader's convenience. We focus on the first grading and omit the second grading for a while. We first take homogeneous cocycle representative of the basis of the cohomology algebra $H^\bullet(A)$; $e_1^1, e_1^2,...., e_2^1,...., e_3^1,...$ where the subscript express its degree. We consider a free dga $A_1$ generated by the formal symbol $\{\widetilde e_i^j\}$ corresponding to $e_{i}^j$ with the trivial differential. Then there exists an obvious map $F_1\colon A_1\rightarrow A$ which is a quasi-isomorphism in degree 0 and 1. The map $F_1$ factors through the cocycle $C^\bullet(A)$, hence induces a map $[F_1]\colon H^2(A_1^\bullet)=A_1^2\rightarrow H^2(A)$. The kernel of $[F_1]$ is generated by elements of the form $\sum a_{ij}\widetilde e_1^i\widetilde e_1^j$ and satisfies $\sum a_{ij}e_1^ie_1^j=dh$ for some $h\in A^1$. We now consider the free dga $A_2$ generated by $\widetilde e_i^j$ and $\widetilde h$'s with the differential $d\widetilde h=\sum a_{ij}\widetilde e_1^i\widetilde e_1^j$. Then we again obtain $F_2\colon A_2\rightarrow A$ which is a quasi-isomorphism in $\deg =0,1,2$. We then consider the obvious surjective map $F_2\colon H^3(A_1)\rightarrow H^3(A)$. We again consider the cocycle representative of the kernel of $F_2$ and adjoin some elements. By repeating these procedures, we obtain a graded dga satisfying 1 and 4. This graded dga is quasi-isomorphic to the original one by the construction.

Now we assume that 1 and 4 are satisfied by $A$.
    We first note that $A$ is quasi-isomorphic to $\bigoplus_{j\geq 0, i}A^{i,j}$. We also have a quasi-isomorphism 
    \begin{equation}
        \mathbf{k}\cdot 1\oplus \bigoplus_{j>0,i}A^{i,j}\hookrightarrow \bigoplus_{j\geq 0, i}A^{i,j}. 
    \end{equation}
    Then the left hand side satisfies the conditions for Adams-connectedness except for 1 and 4. We reset $A$ by the left hand side. We consider the bi-sided ideal $I$ generated by $d(\bigoplus_j A^{-1,j})$. Then this is Adams connected.
\end{proof}

\begin{definition}\label{def:AdamsKoszul}
    Let $A$ be an Adams connected dga. We say $A$ is \emph{Adams Koszul} if $\Ext^{i,j}_A(\mathbf{k},\mathbf{k})=0$ for $i\neq 0$.
\end{definition}

\begin{theorem}[{Adams version of \cite[Theorem 3.8]{HeWu}, see also \S 6 of loc.cit.}]
    Let $A$ be an Adams Koszul dga. Then we have
    \begin{equation}
        \End_B(\mathbf{k})\cong A
    \end{equation}
    for $B=\End_A(\mathbf{k})$.
\end{theorem}

\subsection{Beilinson--Ginzburg--Soergel's mixed geometry philosophy}
In this section, we explain Beilinson--Ginzburg--Soergel's philosophy~\cite{BGS}. In one sentence, it says that
\begin{center}
    \emph{Adams grading comes from mixed geometry.}
\end{center}
We explain a historic example related to geometric representation theory. Let $\frakg$ be a complex semisimple Lie algebra. Let $\frakb$ be a Borel subalgebra, and $\frakh$ be a Cartan subalgebra. Let $U(\frakg)$ be the universal enveloping algebra. Let $\cO$ be the full subcategory of $U(\frakg)$-modules which are (1) semisimple over $\frakh$,
and (2) locally finite over $\frakb$.

Let us fix a dominant integral weight $\lambda\in \frakh^*$. It gives a central character $Z(U(\frakg))\rightarrow \bC$. Then we can define $\cO_\lambda$ as the full subcategory of $\cO$ spanned by the objects with the above central action. Then the only simples of $\cO_\lambda$ are of the form $L(x\cdot \lambda)$ for $x\in W$ where $W$ is the Weyl group of $\frakg$. We denote the projective cover of $L(x\cdot \lambda)$ by $P(x\cdot \lambda)$.

Let $\frakq$ be a parabolic subalgebra containing $\frakb$. We can consider $\cO^\frakq$ consisting of $\frakq$-locally finite modules. Let $S_\frakq\subset W$ be the simple reflections corresponding to $\frakq$. We consider $L_x^\frakq$ be the module $L(x^{-1}w_0\cdot 0)$ for $x\in S_\frakq$. We also denote the projective cover of it by $P_x^\frakq$. Then we have
\begin{theorem}[{Beilinson--Ginzburg--Soergel Koszul duality theorem~\cite{BGS}}]\label{thorem:BGS}
\begin{equation}
\begin{split}
    \End_{\cO_\lambda}(\bigoplus P(x\cdot \lambda))&\cong \Ext^\bullet_{\cO^\frakq}(\bigoplus L_x^\frakq,\bigoplus L_x^\frakq)\\
    \End_{\cO^\frakq}(\bigoplus P_x^\frakq)&\cong \Ext_{\cO_\lambda}^\bullet(\bigoplus L(x\cdot \lambda),\bigoplus L(x\cdot \lambda))
\end{split}
\end{equation}
Moreover the right hand sides are Koszul dual to each other i.e., For $k:=\Ext^0_{\cO^\frakq}(\bigoplus L_x^\frakq)$, we have
\begin{equation}
    \begin{split}
        \Ext^\bullet_{\Ext^\bullet_{\cO^\frakq}(\bigoplus L_x^\frakq,\bigoplus L_x^\frakq)}(\mathbf{k},\mathbf{k})&\cong\Ext_{\cO_\lambda}^\bullet(\bigoplus L(x\cdot \lambda),\bigoplus L(x\cdot \lambda)),  \\
        \Ext^\bullet_{\Ext_{\cO_\lambda}^\bullet(\bigoplus L(x\cdot \lambda),\bigoplus L(x\cdot \lambda))}(\mathbf{k},\mathbf{k})&\cong\Ext^\bullet_{\cO^\frakq}(\bigoplus L_x^\frakq,\bigoplus L_x^\frakq). \\
    \end{split}
\end{equation}
\end{theorem}
This kind of Koszul duality gives a numerous applications, for example, in multiplicity calculations. So, it's worth having the Koszul duality in representation theory.

Since the algebra appeared in the theorem is Koszul, they carry an additional Adams grading. Where does it come from? Beilinson--Ginzbrug--Soergel's philosophy says that \emph{it comes from mixed geometry}.

Namely, consider a setup where we have some representation category $\cO$ with a geometric realization. Namely, there exists a variety $X$ and a stratification $\cW$ such that $\cO\cong \mathrm{Perv}(X, \cW)$, where the latter is the category of perverse sheaves stratified by $\cW$. Suppose that $\cO$ has a projective generator $P$, hence has an equivalence. Suppose also that $\End(P)$ has a Koszul grading. In the setup of Theorem~\ref{thorem:BGS}, for example, one can think as $\cO=\cO_\lambda$, $P=\bigoplus P(x\cdot \lambda)$, the grading is given by $\Ext^\bullet_{\cO^\frakq}(\bigoplus L_x^\frakq)$, and the geometric realization is given by the Beilinson--Bernstein localization. Then what is the graded version on the geometric (perverse sheaf) side?

Now we have a graded version of $\mathrm{Perv}(X, \cW)$, namely, the category of mixed Hodge modules $\MHM^c(X, \cW)$. Suppose an object $\cP$ corresponding to $P$ can be upgraded into an object $\widetilde\cP$ of $\MHM^c(X, \cW)$. Suppose moreover that we are in a nice situation such that 
\begin{equation}
    \End_{\mathrm{Perv}(X, \cW)}(\widetilde\cP,\widetilde\cP)\cong \bigoplus_i\Hom_{\MHM^c(X, \cW)}(\widetilde\cP, \widetilde \cP(i))
\end{equation}
holds.
This gives an additional grading structure on $\End(\cP)$. BGS philosophy says that this additional grading is a good candidate of the Adams grading in general, and is actually so in the setup of Theorem~\ref{thorem:BGS}.

\begin{example}
    [sl(2)] Consider the case when $\frakg= sl(2)$.
In this case, the corresponding variety is $\bP^1$ and $\cW$ is the Schubert stratification is given by $\bP^1=\bC\cup \{\infty\}$. The category $\cO_0$ has two simples $L(0)$ and $L(-2)$. The corresponding perverse sheaves are $\cL(0)=\bC_{\bP^1}[1]$ and $\cL(-2)=\bC_\infty$ where the latter is the skyscraper sheaf at $\infty$. The projective covers are $\cP(0)=\bC_{\bP^1\bs \infty}$ and $\cP(-2)$. The latter has no convenient name.
\end{example}

\subsection{Core--Cocore duality}
In this section, we explain Etg\"u--Ekholm--Lekili's philosophy. In one sentence, it says that
\begin{center}
    \emph{A pair of core and cocore often forms a Koszul dual pair.}
\end{center}

Let $X$ be a Weinstein manifold with the contact boundary $Y$. 
Let $\Lambda$ be a Legendrian subset of $Y$. In the setup, Ekholm--Lekili~\cite{EkholmLekili} define two algebras from the Legendrian. Here we follow the Lagrangian perspective.

Let $L$ be the relative skeleton. We assume that $L$ is decomposed into smooth Lagrangians $\bigcup_{i\in I} L_i$. For a point $x\in L_i$, one can consider cocore/linking disk at $x$ and denote it by $C_i$, which is a Lagrangian intersecting $L$ only at $x$ (see \cite{GPSmicro} for the precise definition). In general, the duality we are interested in is a relation between $\End_{Fuk}(\oplus L_i)$ and $\End_{\cW}(\bigoplus C_i)$ where the first one is taken in the infinitesimally wrapped Fukaya category and the latter is taken in the (partially) wrapped Fukaya category.

The following is a false conjecture:
\begin{conjecture}[]
$\End_{Fuk}(\oplus L_i)$ and $\End_{\cW}(\bigoplus C_i)$ are Koszul dual. Namely, there is a duality between core (i.e., skeleton) and cocore (or linking disk).
\end{conjecture}
This conjecture is not true, but in many examples, the statement holds true.

We first review the case of Ekholm--Lekili~\cite{EkholmLekili}. 

Let $X$ be a Liouville domain with the boundary $Y$. Let $L$ be an immersed exact Lagrangian submanifold with the boundary Legendrian $\Lambda$. We assume $L$ can be divided into embedded Lagrangians $L=\bigcup_{\nu\in \Gamma}L_\nu$ intersecting transversely. Accordingly, $\Lambda$ is also decomposed as  $\Lambda=\bigcup_{\nu\in \Gamma}\Lambda_\nu$. We fix a partition $\Gamma=\Gamma^+\cup \Gamma^-$. We consider $\Gamma^+$ as the stop and $\Gamma^-$ as the surgery component.

Make the surgery along $\Gamma^-$. We obtain a closed Lagrangian $S_\nu$ for each $L_\nu$. We set $L_\Lambda:=\bigcup_{\nu\in \Gamma^+}L_\nu\cup \bigcup_{\nu\in \Gamma^-}S_\nu$. The cocore is denoted by $C_\Lambda:=\bigcup_{\nu\in \Gamma}C_\nu$. We set $LA^*:=\Hom_{Fuk}(L_\Lambda, L_\Lambda)$ and $CE^*:=\Hom_{\cW}(C_\Lambda, C_\Lambda)$. There is another algebra $CE_{||}^*$, which is quasi-isomorphic to $CE^*$ when $\Lambda_\nu$ is simply connected for all $\nu\in \Gamma^+$ (see \cite{EkholmLekili} for the definition).

\begin{theorem}[Ekholm--Lekili~\cite{EkholmLekili}]
Suppose $\Lambda_\nu$ is simply connected for all $\nu\in \Gamma^+$.
Suppose $BLA^*$ is locally finite, simply--connected as a $\mathbf{k}=\bigoplus_{\nu\in \Gamma^+} \bC e_\nu$-bimodule. Then $LA^*$ and $CE^*$ are Koszul dual.
\end{theorem}

\begin{example}[sl(2)]
Let us go back to the case when $X=T^*\bP^1$ with $L=T^*_\infty \bP^1\cup \bP^1$. We set $L_+:=T^*_{\infty}\bP^1$ and $L_-:=\bP^1$. We also set $\Gamma^+:=\{+\}$ and $\Gamma^-:=\{-\}$. Then $\Lambda_+=S^1$ is not simply connected. So we cannot apply the above theorem. But the conclusion holds: $LA^*$ and $CE^*$ are isomorphic to $\bC[x]/x^2$ with $\deg x=1$ and the latter algebra is self-Koszul dual.

So, this case has two explanations of Koszul duality; mixed geometry and core-cocore duality.
\end{example}

The earlier result by Etg\"u--Lekili~\cite{EtguLekili} is also along the line of core-cocore duality, but missing an explanation from mixed geometry. We will explain this example in the next section.

\section{Hodge microsheaves on plumbings of $T^*\bP^1$}\label{section:Koszuldualityforplumbing}

In this section, we study Hodge microsheaves on $A_n$-plumbings of $T^*\bP^1$. Depending on the chocie of core, we argue two kinds of results: (1) If the core is a nodal chain of $\bP^1$'s, we recover the Koszul duality result by Etg\"u--Lekili~\cite{EtguLekili}. (2) If the core is the union of a nodal chain of $\bP^1$'s and one $\bC$, we recover the Koszul dulaity result in the context of symplectic duality~\cite{BLPW,BLPW2}.

\subsection{Etg\"u--Lekili's result} \label{dec27-5}
Let $\Gamma$ be a tree (In the next subsection, we focus on the case when $\Gamma=A_n$). Let $X_\Gamma$ be the $\Gamma$-plumbing of $T^*\bP^1$. As described in \cite{EtguLekili}, the space can be made into a Weinstein manifold. The core of $X_\Gamma$ is the nodal curve consisting of $\{\bP^1_v\}_{v\in V(\Gamma)}$ with the intersection complex $\Gamma$ where each $\bP^1_v$ is $\bP^1$ corresponding to a vertex $v\in V(\Gamma)$ of $\Gamma$.

By the result of \cite{GanatraSectorial, Chantraineetal}, the wrapped Fukaya category $\cW(X_\Gamma)$ of $X_\Gamma$ is generated by the cocores. Each cocore can be identified with a cotangent fiber $T^*_v$ of each $\bP^1_v$ in the core. On the other hand, the compact Fukaya category $F(X_\Gamma)$ of $X_\Gamma$ is generated by $\bP^1_v$'s in the core.

The first main result of Etg\"u--Lekili is the identification of the endomorphism ring of these categories. For this purpose, we introduce two dga associated to $\Gamma$.

We set $\mathbf{k}$ to be the semisimple ring given by $\bigoplus_{v\in V(\Gamma)} \bK e_v$. The ring $A_\Gamma$ is a $\bZ$-graded $\mathbf{k}$-algebra generated by degree 1 elements $e_{v_1v_2}$ associated to each adjacent pair $(v_1, v_2)\in V(\Gamma)\times V(\Gamma)$ and a degree 2 element $w_v$ associated to each $v\in V(\Gamma)$. The relations are given by
\begin{enumerate}
    \item $e_{vw}e_{wv}=w_v$,
    \item the other products are zero.
   \end{enumerate}

\begin{proposition}[\cite{EtguLekili}]
    Suppose $\Gamma$ is of type AD. Then $\End_{F(X_\Gamma)}(\bigoplus_{v\in V(\Gamma)}\bP^1_v)\cong A_\Gamma$.
\end{proposition}

Next, choose an orientation of $\Gamma$ and regard it as a quiver. Let $\widehat \Gamma$ be the extended quiver consisting of:
\begin{enumerate}
    \item The same vertices as $\Gamma$,
    \item \begin{enumerate}
        \item the arrows of the double of $\Gamma$ of degree $(1,-1)$, namely, $g$ and $g^*$ for each edge $g$ of $\Gamma$, and
        \item loop $h_v$ at each vertex $v\in V(\Gamma)$ of degree $(1,-2)$.
    \end{enumerate}
\end{enumerate}
The Ginzburg dga $\cG_\Gamma$ is a doubly graded dga defined by the path algebra of $\widehat \Gamma$ over $\mathbf{k}$ together with differential of degree $(1,0)$ defined by $dg=dg^*=0, dh_v=\sum_{\text{$g^*$ starts from $v$}}gg^*-\sum_{\text{$g$ starts from $v$}}g^*g$. When we view $\cG_\Gamma$ as a single-graded dga (usual dga), we take the total degree.

\begin{proposition}[\cite{EtguLekili}]
    Suppose $\Gamma$ is of type AD. Then $\End_{W(X_\Gamma)}(\bigoplus_{v\in V(\Gamma)}T^*_v)\simeq \cG_\Gamma$ as $A_\infty$-algebras.
\end{proposition}

The main Koszul duality statement of \cite{EtguLekili} is the following:
\begin{theorem}[\cite{EtguLekili}]
Suppose $\Gamma$ is of type $AD$.
    Viewing the second grading of $\cG_\Gamma$ as Adams grading, $A_\Gamma$ is Koszul dual to $\cG_\Gamma$.
\end{theorem}
Going back to the interpretation in terms of Fukaya categories, one can consider this theorem as a manifestation of core--cocore Koszul duality. However, in this story, the Adams grading on $\cG_\Gamma$ lacks its geometric origin. In the next section, we give a conjectural mixed geometry interpretation of this grading and provide some evidences.

\subsection{Koszul duality of $\cG_{A_n}$ from Hodge microsheaves}\label{mar24-10}
For $\Gamma=A_n$, we will consider $\mu \mathrm{sh}_C(X_\Gamma)$ where $C$ is the core, which is a nodal chain of $\bP^1$. By Ganatra--Pardon--Shende's theorem~\cite{GPSmicro}, we have an object $\cH^\infty_j\in \mu \mathrm{sh}_C(X_\Gamma)$ corresponding to the cocore $T_j^*\in \cW(X_\Gamma)$ for $j\in \Gamma$. In other words, this is the wrapping of a skyscraper sheaf on a non-nodal point of $\bP^1$ of $j$.

By definition and Etg\"u--Lekili's result~\cite{EtguLekili,EtguLekili2}, we have 
\begin{corollary}We have a quasi-isomorphism 
      \begin{equation}
        \End_{\mu \mathrm{sh}_C(X_\Gamma)}(\bigoplus_j \cH_j^\infty)\simeq \cG_\Gamma
    \end{equation}
    as $A_\infty$-algebras.  
\end{corollary}

We use the saturated mixed structure explained in Example~\ref{ex:saturatedHodgesheafforplumbing}. We put $\mu \mathrm{M}_{C_{\lc \boldsymbol{m}\rc}}(X_\Gamma):=\Ind(\mu \mathrm{M}^c_{C_{\lc \boldsymbol{m}\rc}}(X_\Gamma))$ and $\mu \mathrm{M}_{C}(X_\Gamma):=\mu\mathrm{M}_{C_{\lc \boldsymbol{m}\rc}}(X_\Gamma)\cap \mu\MHM_{C}(X_\Gamma)$.

Under this situation, we have the following:
\begin{theorem}\label{thm:comparisonGinzburng}
There exists an object $\cH_j^{\infty, H}\in \mu\mathrm{M}_C(X_\Gamma)$ such that $\frakF(\cH_j^{\infty, H})=\cH_j^\infty$. Moreover, $\cH_j^{\infty, H}$ is the $(\mu \mathrm{M}_{C}(X_\Gamma),\mu \mathrm{M}_{C_{\lc \boldsymbol{m}\rc}}(X_\Gamma))$-Hodge wrapping of $\bK_{m_j}$.
\end{theorem}
This will be proved in appendix. Here we would like to show the Koszul duality of $\End_{\mu \mathrm{sh}(X_\Gamma)}(\cH_j^\infty)$ by using the mixed grading. 

We set $\cH^{\infty, H}:=\bigoplus_j\cH_j^{\infty, H}$ and $\cH^{\infty}:=\bigoplus_j\cH_j^{\infty}$.
Recall that we have
\begin{equation}
     \Hom_{\mu \mathrm{M}_C(X_\Gamma)}(\cH^{\infty, H},\bigoplus_{s\in \bZ}\cH^{\infty, H}(s/2))\simeq  \End_{\mu \mathrm{sh}_C(X_\Gamma)}(\cH^\infty)
\end{equation}
by the saturatedness. The left hand side has a doubly graded decomposition
\begin{equation}
    \Hom_{\mu \mathrm{M}_C(X_\Gamma)}(\cH^{\infty, H}, \bigoplus_{s\in \bZ}\cH^{\infty, H}(s/2))=\bigoplus_{k\in\ZZ}\Hom^k_{\mu \mathrm{M}_C(X_\Gamma)}(\cH^{\infty, H}, \bigoplus_{s\in \bZ}\cH^{\infty, H}(s/2))=:\bigoplus_{a,b}B^{a,b}
\end{equation}
where $a=k-s$ and $b=-s$. We consider $B$ as an Adams-graded dga where the Adams grading is $b$.

\begin{lemma}\label{jan8-8}
    One can take $B=\bigoplus_{a,b}B^{a,b}$ to be an Adams-connected dga among the quasi-isomorphism class.
\end{lemma}
\begin{proof}
  This lemma will be proved in Subsection~\ref{dec20-5} by direct computations.  
\end{proof}

We denote the derived category of Adams-graded $B$-modules by $\Mod(B)$.
Then $\Mod(B)$ carries an autoequivalence $\la 1\ra :=\lb -1\rb[-1]$, where $\lb -1\rb$ (resp. $[-1]$) is the shift of the Adams-grading (cohomological degree) by $-1$.
For $\cE\in \mu \mathrm{M}_C(X_\Gamma)$,
we define the bidegree of the complex
$\bigoplus_{s\in \bZ}\Hom_{\mu \mathrm{M}_C(X_\Gamma)}(\cH^{\infty, H}, \cE(s/2))$
so that the cohomological (resp. Adams) degree of $\Hom^k_{\mu \mathrm{M}_C(X_\Gamma)}(\cH^{\infty, H}, \cE(s/2))$ is $k-s$ (resp. $-s$).
Then, $\bigoplus_{s\in \bZ}\Hom_{\mu \mathrm{M}_C(X_\Gamma)}(\cH^{\infty, H}, \cE(s/2))$ is an Adams-graded $B$-module.

\begin{lemma}\label{jan9-1}
    We have an equivalence
    \begin{equation}\label{jan9-3}
        \mu \mathrm{M}_C(X_\Gamma)\simeq \Mod(B); \cE\mapsto  
        \bigoplus_{s\in \bZ}\Hom_{\mu \mathrm{M}_C(X_\Gamma)}(\cH^{\infty, H}, \cE(s/2)).
    \end{equation}
    The equivalence sends the half-Tate twist to the shift $\la 1\ra$ defined above.
    Moreover, the augmentation module on the right hand side corresponds to the direct sum of the rank 1 constant sheaves on each sphere under this equivalence.
\end{lemma}
\begin{proof}
Since $\bigoplus_j \cH_j^\infty$ generates $\mu \mathrm{sh}(X_\Gamma)$, the lift $\bigoplus_s\bigoplus_j\cH_j^{\infty,H}(s/2)$ generates $\mu \mathrm{M}_C(X_\Gamma)$ as well. The continuity of the functor is also evident. By the usual Morita-theorem like argument, we get an equivalence.

The second assertion will be proved by direct computations in Subsection~\ref{dec20-5}.
\end{proof}

As a corollary of this lemma, we have
\begin{equation} \label{jan7-3}
   \bigoplus_{s\in \ZZ}\Hom_{\mu \mathrm{M}_C(X_\Gamma)}(\bigoplus_{i=1}^n\bK_{\bP_i^1},\bigoplus_{i=1}^n\bK_{\bP_i^1}(s/2))\simeq
   \bigoplus_{s\in \ZZ}\Hom_{B}(\mathbf{k}, \mathbf{k}\la s\ra).
\end{equation}
where $\mathbf{k}$ is the augmentation module and $\la s\ra:=\la 1\ra^s$.
In particular, we have
\begin{equation} \label{mar19-1}
   \bigoplus_{s\in \ZZ}\mathrm{Ext}^k_{\mu \mathrm{M}_C(X_\Gamma)}(\bigoplus_{i=1}^n\bK_{\bP_i^1},\bigoplus_{i=1}^n\bK_{\bP_i^1}(s/2))
   \simeq
   \bigoplus_{s\in \ZZ}\mathrm{Ext}^{k-s}_{B}(\mathbf{k}, \mathbf{k}\lb -s\rb).
\end{equation}
We regard the left hand side of (\ref{jan7-3}) as
an Adams graded dga by setting the degree of $\Hom^k_{\mu \mathrm{M}_C(X_\Gamma)}(\bigoplus_{i=1}^n\bK_{\bP_i^1},\bigoplus_{i=1}^n\bK_{\bP_i^1}(s/2))$
as $(k,s)$.
By explicitly computing the left hand side of (\ref{mar19-1}),
we will prove the following corollary in Subsection~\ref{jan8-10}.

\begin{corollary}\label{apr-1}
    The algebra $B$ is an Adams Koszul dga in the sense of \cite{HeWu}.
\end{corollary}

Finally, we discuss the algebra structure of endomorphism algebras of $\scH^{\infty,H}$ and
$\bigoplus_{i=1}^n\bK_{\bP_i^1}$,
and their Koszul duality.
The following lemma will also be proved in Subsection~\ref{jan8-10}.

\begin{lemma}\label{jan7-4-2}
The left hand side of (\ref{jan7-3}) is isomorphic to the formal dga $A_\Gamma$ in \cite{EtguLekili}.
\end{lemma}

Then, applying the theory of \cite{HeWu}:
the Adams-version of \cite[Theorem 3.8]{HeWu}, we conclude
\begin{corollary}
    We have a quasi-isomorphism between Adams-graded dgas:
    \begin{equation}
        \bigoplus_{s\in \ZZ}\Hom_{A_\Gamma}(\mathbf{k}, \mathbf{k}\lb s\rb)\simeq B.
    \end{equation}
\end{corollary}

On the other hand, by the result of \cite{EtguLekili}, the Koszul dual of $A_\Gamma$ is $\cG_\Gamma$. So, $\cG_\Gamma$ and $B$ are quasi-isomorphic. In other words, we have:
\begin{corollary}
    There exists a quasi-isomorphism between Adams-graded dgas:
    \begin{equation}
    \cG_\Gamma\simeq \Hom_{\mu \mathrm{M}_C(X_\Gamma)}(\cH^{\infty,H}, \bigoplus_{s\in \bZ}\cH^{\infty, H}(s/2)).
\end{equation}
\end{corollary}
This corollary explains the mixed-geometric origin of the additional grading on the Ginzburg dga.

\subsection{Category $\cO$ of $A_n$-plumbings of $T^*\bP^1$}\label{section:BLPW}
Recall that the core $C$ in the last subsection is a nodal chain of $\bP^1$'s of $A_n$-type. Let $v$ be one of one-valent vertices of $\Gamma$. We put 
\begin{equation}
    C'=C\cup T^*_{x}\bP^1_v
\end{equation}
where $x$ is a non-nodal point on $\bP^1_v$. In this subsection, we use $C'$ to define Hodge microsheaves.

Since $X_\Gamma$ is a symplectic model of crepant resolution of $\bC^2/(\bZ/(n+1)\bZ)$, this is a conical symplectic resolution. Moreover $C'$ here is the relative core in the context~\cite[Example 3.4]{BLPW2}. The category $\mu \mathrm{sh}_{C'}(X_\Gamma)$ is identified with a block of the category $\cO$ of $X_\Gamma$, one can see this from comparing the explicit description given in \S~\ref{appendix:BLPW} and \cite[Example 4.11]{BLPWhypertoric} (or one may also use the microlocal Riemann--Hilbert~\cite{CKNSmicroRH} to deduce it). Moreover, $\mu \mathrm{sh}_{C'}(X_\Gamma)$ (more precisely, a t-structure of it) is Koszul again by \cite[Example 4.11]{BLPWhypertoric}. We give a proof of this fact by using Hodge microsheaves in Appendix~\ref{appendix:BLPW}. Parallel to Theorem~\ref{thm:comparisonGinzburng} and Corollary|\ref{apr-1}, we have the following
\begin{theorem}
\begin{enumerate}
    \item There exists an object $\cH_j^{\infty, H}\in \mu\mathrm{M}_{C'}(X_\Gamma)$ such that $\frakF(\cH_j^{\infty, H})=\cH_j^\infty$. Moreover, $\cH_j^{\infty, H}$ is the $(\mu \mathrm{M}_{C'}(X_\Gamma),\mu \mathrm{M}_{C'_{\lc \boldsymbol{m}\rc}}(X_\Gamma))$-Hodge wrapping of $\bK_{m_j}$.
    \item The algebra
    \begin{equation}
        \bigoplus_{j,k,s}\Hom_{\mathrm{M}_{C'}(X_\Gamma)}(\cH_j^{\infty, H}, \cH_k^{\infty, H}(s/2))
    \end{equation}
    is a Koszul algebra.
\end{enumerate}
\end{theorem}
The details will be provided in \S~\ref{appendix:BLPW}.

\section{Appendix: Proof of Theorem~\ref{thm:comparisonGinzburng}}
This appendix is for the proof of Theorem~\ref{thm:comparisonGinzburng}. First, we explicitly compute microlocal skyscrapers (=Ganatra--Pardon--Shende counterpart of cocores, or corepresentatives of microstalks). The construction is quite lengthy, but we believe that it is worth recording here, since there are not much written explicit construction of microlocal skyscrapers in the Weinstein context.
Next, based on the explicit construction, we lift it to the Hodge setup. We then complete the proof of Theorem~\ref{thm:comparisonGinzburng}.

In Subsection~\ref{dec20-1}, we organize the notations related to $\mu \mathrm{sh}_C(X_\Gamma)$.
In Subsection~\ref{dec20-2}, we define a building block $\block{k}$ to construct the microlocal skyscrapers and state a vanishing result: Lemma~\ref{m29-2} for it.
In Subsection~\ref{nov23-1}, we prove Lemma~\ref{m29-2} in four cases, using the facts placed in Subsection~\ref{dec20-3}.
In Subsection~\ref{dec20-4}, we construct the object $\scH_j^\infty\in \mu \mathrm{sh}_C(X_\Gamma)$ and show that it is in fact the microlocal skyscraper.
In Subsection~\ref{dec20-5}, we discuss the ``mixed Hodge structure" on $\scH_j^\infty$.

\subsection{Notations and some remarks on $\mu \mathrm{sh}_C(X_\Gamma)$}
\label{dec20-1}
Recall the notation for the $A_n$-plumbing of $T^*\PR$.
We fix three different points $l,r,m$ of the complex projective line $\PR$.
We set $V_r:=\PR\setminus \{l\}(\simeq \bC)$, $V_l:=\PR\setminus \{r\}(\simeq \bC)$, $W:=\PR\setminus \{l,r\}(\simeq \CS)$ and
$T:=\PR\setminus \{m\}(\simeq \bC)$.
Let $\PR_1,\dots,\PR_n$ be the $n$ copies of $\PR$ with $l,r,m$.
By gluing, we construct the $A_n$-plumbing $X_\Gamma$ of $T^*\PR_1,\dots,T^*\PR_n$.
We use the following convention of the gluing: the core is a nodal curve $\Ancore:=\mathrm{Core}(X_\Gamma)$ with $n-1$ nodal points consisting of the intersection between $\PR_i$ and $\PR_{i+1}$ ($1\leq i\leq n-1$) at the points $r\in \PR_i$ and $l\in \PR_{i+1}$.
Then, by the definition of plumbing, for $1\leq i\leq n-1$ we have a natural identification:
\begin{align}\label{m27-2}
    T_r\PR_{i}\simeq T_{l}^*\PR_{i+1} 
\end{align}
where $T_{r}\PR_{i}$ (resp. $T_l^*\PR_{i+1}$) is the tangent fiber at $r\in \PR_i$ (resp. cotangent fiber at $l\in \PR_{i+1}$). 
Recall that the category of microlocal sheaves 
$\mu \mathrm{sh}_C(X_\Gamma)$ can be identified with the homotopy limit of dg-categories
\begin{align*}
    \mu \mathrm{sh}_C(X_\Gamma)=\Sh(\PR_1,\{r\})\times_{\Sh(T_{r}\PR_1,0)}\Sh(\PR_2,\{l,r\})\times_{\Sh(T_{r}\PR_2,0)}\dots \times_{\Sh(T_r\PR_{n-1},0)} \Sh(\PR_n,\{l\}),
\end{align*}
where $\Sh(\PR_1,\{r\})\to \Sh(T_{r}\PR_1,0)$ is defined as the specialization functor at the point $r\in \PR_1$,
and $\Sh(\PR_{2},\{l,r\})\to \Sh(T_{r}\PR_1,0)$ is defined as the composition of the specialization functor at $l$ and the Fourier transformation, i.e. the microlocalization functor,
and other $\times_{\Sh(T_{r}\PR_i,0)}$ are defined similarly.
Similarly, we consider the category of microlocal sheaves $\mu \mathrm{sh}_{C_{\{\boldsymbol{m}\}}}(X_\Gamma)$ where $C_{\{\boldsymbol{m}\}}=\mathrm{Core}(X_\Gamma, \bigcup_i\partial T^*_m \bP^1_i)$ which has the form 
\begin{align*}
    \mu \mathrm{sh}_{C_{\{\boldsymbol{m}\}}}(X_\Gamma)=\Sh(\PR_1,\{m,r\})\times_{\Sh(T_{r}\PR_1,0)}\Sh(\PR_2,\{l,m,r\})\times_{\Sh(T_{r}\PR_2,0)}\dots \times_{\Sh(T_r\PR_{n-1},0)} \Sh(\PR_n,\{l,m\}).
\end{align*}

In this subsection, we see the basic properties of $\mu \mathrm{sh}_{C_{\{\boldsymbol{m}\}}}(X_\Gamma)$.

We first prepare a notation for objects in $\Hzsh{\PR}$.

\begin{definition}\label{m22-1-2}
    Let $F_l$ (resp. $F_r$, $F_m$) be an object of $\Sh(V_l)$ (resp. $\Sh(V_r)$, $\Sh(W)$).
    Assume that the isomorphisms $F_m\simto F_l|_{W}$ and $F_m\simto F_r|_{W}$ on $W$ are given.
    We regard these objects and morphisms as the ones on $\PR$ extended by zero.
    Then, we take $(F_l,F_r)\in \Sh(\PR)$ as a cone of the morphism 
    \[F_m\to F_l\oplus F_r.\]
\end{definition} 
Note that we have the isomorphisms
    \begin{align*}
        (F_l,F_r)|_{V_l}&\simeq F_l\\
        (F_l,F_r)|_{V_r}&\simeq F_r,\ \mbox{and}\\
        (F_l,F_r)|_{W}&\simeq F_m.
    \end{align*}

The objects of $ \mu \mathrm{sh}_{C_{\{\boldsymbol{m}\}}}(X_\Gamma)$ is just a fiber product:
\begin{align}\label{s10-1-2}
    \mathrm{Ob}(\Sh(\PR_1,\{m,r\}))\times_{\mathrm{Ob}(\Sh(T_{r}\PR_1,0))}\dots \times_{\mathrm{Ob}(\Sh(T_r\PR_{n-1},0))} \mathrm{Ob}(\Sh(\PR_n,\{l,m\}))
\end{align}
where $\mathrm{Ob}$ expresses the set of isomorphism classes in the homotopy category.
More concretely,
an object in $\mu \mathrm{sh}_{C_{\{\boldsymbol{m}\}}}(X_\Gamma)$ can be expressed as a tuple:
\begin{align}\label{m27-3n-2}
    H:=((F^1,G^1),(F^2,G^2),\dots,(F^n,G^n)),
\end{align}
with the following conditions:
\begin{enumerate}
    \item $F^i$ (resp. $G^i$) is an object in $\Sh(V_l,\{l,m\})$ (resp. $\Sh(V_r,\{r,m\})$)
    and $(F^i,G^i)\in \Sh(\PR_i,\{l,r,m\})$ is an expression under Definition~\ref{m22-1-2}.
    \item Under the identification (\ref{m27-2}), we have an isomorphism in $\Sh(T_0V_r,0)$
    \begin{align}\label{m27-4n-2}
\mathrm{FL}(\nu_r(G^i))\simeq \nu_l(F^{i+1}),
\end{align}
where $\nu_r$ (resp. $\nu_l$) is the specialization functor at $r$ (resp. $l$) and $\mathrm{FL}(-)$ is the Fourier transformation.
\end{enumerate}

For the object $H$ expressed as (\ref{m27-3n-2}),
we define $H^{\circ}$ as the ``flipped" object: 
\begin{align}\label{m28-3}
    H^{\circ}=((G^n,F^n),(G^{n-1},F^{n-1}),\dots,(G^1,F^1)).
\end{align}

Next, we consider morphisms between
two objects $H=((F^1,G^1),\dots,(F^n,G^n))$ and $H'=(({F'}^{1},{G'}^{1}),\dots,({F'}^n,{G'}^n))$ in $\mu \mathrm{sh}_{C_{\lc \boldsymbol{m}\rc}}(X_\Gamma)$.
In general, the space of morphisms $H^0\HOM_{\Ancore}(H,H')$ cannot be expressed as a naive fiber product.
Let $U_i$ be the union of $V_r\subset \PR_{i-1}$ and $V_l\subset \PR_{i}$ for $2\leq i\leq n$, $U_1$ is $V_l\subset \PR_1$ and $U_{n+1}$ is $V_r\subset \PR_n$.
Note that $U_i\cap U_{i+1}=W\subset \PR_{i}$.
We define 
\begin{align*}
\mu \mathrm{sh}_{C_{\lc\boldsymbol{m}\rc}}(U_i):=\left\{\begin{array}{cc}
     \Sh(V_r, \lc r,m\rc)\times_{\Sh(T_r\PR,0)}\Sh(V_l, \lc m,l\rc)&2\leq i\leq n \\
     \Sh(V_l,\lc m\rc)& i=1\\
     \Sh(V_r, \lc m\rc)&i=n+1.
\end{array}\right.
\end{align*}

We note that 
\[\mu \mathrm{sh}_{C_{\lc \boldsymbol{m}\rc}}(X_\Gamma)=\mu \mathrm{sh}_{C_{\lc \boldsymbol{m}\rc}}(U_1)\times_{\Sh(W)} \mu \mathrm{sh}_{C_{\lc \boldsymbol{m}\rc}}(U_2)\times_{\Sh(W)}\cdots \times_{\Sh(W)}\mu \mathrm{sh}_{C_{\lc \boldsymbol{m}\rc}}(U_{n+1}).\]
Therefore, we have a long exact sequence
\begin{align}\label{s11-2-2}
    \dots \to &H^k\HOM_{\Ancore}(H,H')\to \bigoplus_{i=1}^{n+1}H^k\mathrm{Hom}_{U_i}(H,H')
    \to 
    \bigoplus_{i=1}^{n}H^k\mathrm{Hom}_{U_{i}\cap U_{i+1}}(H,H')\\
    \to &H^{k+1}\HOM_{\Ancore}(H,H')\to \dots.\notag
\end{align}
Here and the below, we omit $\Sh$, $\mu \mathrm{sh}$, and $\boldsymbol{m}$ from the subscript of hom-spaces. 

In the above sequence, hom-spaces over $U_{i}\cap U_{i+1}$ are easy to understand, since they are usual sheaves. Moreover, in our case, 
$H^k\mathrm{Hom}_{U_i}(H,H')$ can be also expressed in a simple form as follows.
We note the following elementary fact.
\begin{lemma}\label{s11-6-2}
    Consider a Cartesian diagram of complexes $A^\bullet,B^\bullet,C^\bullet$ and $D^\bullet$ of vector spaces:
\[\xymatrix{A^\bullet\ar[r] \ar[d]&B^\bullet \ar[d]\\ C^\bullet\ar[r] &D^\bullet }.\]
Then, if $H^k(B^\bullet)\to H^k(D^\bullet)$ is surjective for any $k\in \ZZ$,
we have an exact sequence:
\[0\to H^k(A^\bullet)\to H^k(B^\bullet)\oplus H^k(C^\bullet)\to H^k(D^\bullet)\to 0.\]
\end{lemma}
Applying this lemma to $\mu \mathrm{sh}_C(U_i)$, we obtain the following corollary:
\begin{corollary}\label{jan24-1cor}
    There exists a short exact sequence:
\begin{align*}
    0\to H^k\mathrm{Hom}_{U_i}(H,H')\to& 
H^k\mathrm{Hom}_{V_r}(G^i,{G'}^i)\oplus H^k\mathrm{Hom}_{V_l}(F^{i+1},{F'}^{i+1})\\
\to &H^k\mathrm{Hom}_{T_rV_r}(\nu_r(G^i),\nu_r({G'}^i))\to 0.
\end{align*}
for $1\leq i\leq n-1$ 
if one of the following conditions holds:
\begin{enumerate}
    \item \label{oct30-1} $G^i$ and ${G'}^i$ on $V_r$ have singular points at most at $0$,
    \item \label{oct30-2} $F^{i+1}$ and ${F'}^{i+1}$ on $V_l$ have singular points at most at $0$.
\end{enumerate}
\end{corollary}

This means that a morphism in $H^k\mathrm{Hom}_{U_i}(H,H')$ can be expressed as a pair 
\begin{align}\label{s11-5}
    (f_l,f_r)\in H^k\mathrm{Hom}_{V_r}(G^i,{G'}^i)\oplus H^k\mathrm{Hom}_{V_l}(F^{i+1},{F'}^{i+1}), 
\end{align} 
with $\nu_r(f_l)=\mathrm{FL}(\nu_l(f_r))$.
Although the assumption in the above corollary does not hold in general, we have the same consequence in our situation (see Lemma~\ref{dec20-6}).

So,
the morphism 
$\bigoplus_{i=1}^{n+1}H^k\mathrm{Hom}_{U_i}(H,H')
    \to 
    \bigoplus_{i=1}^{n}H^k\mathrm{Hom}_{U_{i}\cap U_{i+1}}(G^i|_{W},{G'}^i|_{W})$ in the exact sequence (\ref{s11-2-2})
    is expressed as
    \begin{align}\label{s11-7-2}
\xymatrix{
\bigoplus_{i=1}^{n+1}H^k\mathrm{Hom}_{U_i}(H,H')
    \ar[r]&
    \bigoplus_{i=1}^{n}H^k\mathrm{Hom}_{U_{i}\cap U_{i+1}}(H,H')\\
(b_1,(b_{2,l},b_{2,r}),\dots,(b_{n,l},b_{n,r}),b_{n+1})\ar@{|->}[r]\ar@{}[u]|-{\rotatebox{90}{$\in$}}& 
(b_1|_{W}-b_{2,l}|_{W},b_{2,r}|_{W}-b_{3,r}|_{W},\dots, b_{n,r}|_{W}-b_{n+1}|_{W})\ar@{}[u]|-{\rotatebox{90}{$\in$}}.
}        
    \end{align}

In order to study the morphisms in a (classical) derived category of sheaves,
we prepare some notations.
For a distinguished triangle $A'\to A\to A''\to A'[1]$ and an object $B$ in $\Sh(X)$ for a complex manifold $X$,
for any $k\in \bZ$ we have the long exact sequence:
\[\dots \to \HOM(A[1],B[k])\to \HOM(A'',B[k])\to \HOM(A,B[k])\to \HOM(A',B[k])\to  \dots,\]
    where $\HOM$ means $\HOM_{\Hzsh{X}}$.

\begin{definition}\label{s10-3-2}

\begin{enumerate}
    \item 
For $i,j\in \ZZ$ ($i\neq j$), 
if there exist subspaces $L'_k\subset \HOM(A',B[k])$ and $L''_k\subset \HOM(A'',B[k])$ with an exact sequence
\[0\to L''_k\to \HOM(A,B[k])\to L'_k\to 0\] for $k= i,j$, and $\HOM(A,B[k])=0$ for $k\neq i,j$,
then we write
\[{{\xymatrix{A' \ar[d]\ar[rd]^{\overbrace{i,\dots,i}^{d_i'},\ \overbrace{ j,\dots,j}^{d_j}}& &\\ A\ar[d]\ar[r]& B,&\mbox{or}\qquad \\ A''\ar[ur]_{\underbrace{i,\dots,i}_{d''_i},\ \underbrace{j,\dots,j}_{d''_j}}&&}}}
{\xyhidaridt{A'}{A}{A''}{B,}{i\times d'_i,\ j\times d_j'}{i\times d''_i,\ j\times d''_j}}\] 
where 
$d'_k=\dim{L'_k}$ and $d''_k=\dim{L''_k}$.

\item \label{j25-2-2}
For $i, j\in \ZZ$ ($i\neq j$),
if $\HOM(A,B[k])=0$ for $k\neq i,j$ (for example, in the situation above),
we also write
\[\xymatrix@C=50pt{A\ar[r]^{i\times d_i,\ j\times d_j}&B},\]
where $d_i=\dim \HOM(A,B[i])$ and $d_j=\dim \HOM(A,B[j])$.

\end{enumerate}

\end{definition}

We adopt the same definition also for 
similar situations (see Example~\ref{jan28-103}),
in particular, for a distinguished triangle $B'\to B\to B''\to B'[1]$.
Remark that we only use the above diagrams only for distinguished triangles;
whenever we use the above diagrams, the vertical sequence $A'\to A\to A''$ is always a distinguished triangle.

\begin{example}\label{jan28-103}
The diagram:
\[{\xyhidaridt{A'}{A}{A''}{B}{0,1,1}{2}}\]
means the following.
\begin{enumerate}
    \item There is one degree $0$ morphism $f\colon A\to B[0]$ up to multiplicative constant which corresponds to $A'\to B[0]$, i.e. $f\circ (A'\to A)\neq 0$.
\item There are two degree $1$ morphisms $A\to B[1]$ which form a basis of $\HOM(A,B[1])$ and each morphism corresponds to $A'\to B[1]$.  
\item There is one degree $2$ morphism $g\colon A\to B[2]$ up to multiplicative constant which corresponds to $A''\to B[2]$, i.e. there is a morphism $h\colon A''\to B[2]$ such that $g=h\circ (A\to A'')$.
\item There is no non-zero morphism $A\to B[k]$ for $k\neq 0,1,2$.

\item Remark that there may be more morphisms 
from $A'$ or $A''$ to $B[k]$ for some $k$.
For example, if there is a non-zero morphism $f\in \HOM(A'',B[1])$,
then the diagram above implies that
there is a morphism $g\colon A'[1]\to B[1]$ such that the following diagram commutes:
\[\xymatrix{{A''}\ar[r]^{f}\ar[d]&{B[1]}.\\ A'[1]\ar[ur]_{g}&}\]

\end{enumerate}
\end{example}

\begin{remark}
In this paper, the ambiguity of constant multiplication of a morphism is not important in most cases.
So, for $A,B\in \sh{X}$, 
if $\dim\HOM(A,B)=1$,
we always take and fix a non-zero morphism $A\to B$ and refer to it as ``the morphism $A\to B$" without giving any specific symbols.
Once we confirm $\dim\HOM(A,B)=1$,
the symbol $A\to B$ appearing in an equation or a diagram means ``the morphism $A\to B$". 
For example, for $C\in \sh{X}$ with $\dim\HOM(B,C)=1$,
when we consider the composition of
``the morphism $A\to B$" and ``the morphism $B\to C$",
we just write it as $(B\to C)\circ (A\to B)$.
Even if $\dim\HOM(A,B)>1$,
after fixing one morphism $A\to B$ (typically with some properties), we use a similar convention.
\end{remark}

\subsection{Basic objects in $\Sh(\CC)$ and the object $\block{k}$}
\label{dec20-2}

In this subsection,
we define some basic sheaves on $\CC$.
Then we define the object $\block{k}$ in $\mu \mathrm{sh}_C(X_\Gamma)$,
which will be a ``building block" to construct the microlocal skyscrapers. 
Let $V$ be a complex plane $\bC$ with two different points $0,m$,
i.e. $V_l$ or $V_r$.
We set $W:=V\setminus \{0\}$ and $T:=V\setminus \{m\}$.
The sheaf $\CV$ is the constant sheaf on $V$, $\CW$ is the zero-extension of the constant sheaf on $W$ to $V$ and $\CC_0$ is the skyscraper sheaf supported at $0$.

\begin{definition}\label{oct31-1}
\begin{enumerate}
    \item For $s\geq 1$, we denote by $\scL_s$ the $\bC$-local system whose monodromy matrix is the unipotent Jordan block of size $s$ on $W$.
    \item We denote by $A_s$ the zero extension of $\scL_s$ to $V$. We regard it as an object in $\Sh(V,0)$.
    \item We set $B_s:=\rR\Gamma_{W}A_s\in \Sh(V,0)$.
\end{enumerate}
\end{definition}
The followings are standard.
\begin{lemma}
\begin{enumerate}
    \item For $s\geq 2$, we have the following distinguished triangles in $\Sh(V,0)$:
    \begin{align}
        \CW\to A_{s}\to A_{s-1}\to  \CW[1], \label{m16-3}\\
        A_{s-1}\to A_{s}\to \CW\to  A_{s-1}[1]. \label{jan27-101}
    \end{align}
    \item For $A_s$ (resp. $\CV$) on $V$ and $B_s$ (resp. $\CC_0[-1]$) on the dual vector space $V^{*}$,
    two objects are transformed into each other by the Fourier transformation.
    
        \item \label{m16-1} We have the following diagrams:
        \begin{align}\label{m16-2-2}
            \xyhidaridt{\CW}{A_s}{A_{s-1}}{\CV,}{1}{0}\qquad \xyhidaridt{\RG{W}\CV}{B_s}{B_{s-1}}{\bC_0,}{0}{-1}
        \end{align}
    where the vertical sequence in the first diagram is the distinguished triangle (\ref{m16-3})
    and the one in the second diagram is obtained by applying the functor $\RG{W}(-)$ to (\ref{m16-3}).
    Two diagrams are transformed into each other (up to shifts) by the Fourier transformation.
\end{enumerate}
\end{lemma}

\begin{definition}  \label{nov6-6}
\begin{enumerate}    
   \item 
    We set $P_1:=\CV$.
    For $s\in \ZZ_{\geq 2}$, we take $P_s\in \Sh(V,0)$ which fits into the distinguished triangle:
   \[\bC_V\to P_s\to A_{s-1}\to \bC_V[1],\]
   where $A_{s-1}\to \CV[1]$ is a nonzero morphism defined by the diagram~(\ref{m16-2-2}).
\item For $s\in \ZZ_{\geq 2}$, we take $Q_s\in \sh{V,0}$ which fits into the distinguished triangle:
   \[\bC_0[-1]\to Q_s\to B_{s-1}\to \bC_0,\]
    where $B_{s-1}\to \bC_0$ is a nonzero morphism in the diagram~(\ref{m16-2-2}).
    \end{enumerate}
\end{definition}

\begin{lemma}
\begin{enumerate}
\item \label{j24-3}
 For $P_s$ on $V$ and $Q_s$ on the dual vector space $V^{*}$,
    two objects are transformed into each other by the Fourier transformation.

    \item \label{j24-4} We have distinguished triangles:
    \begin{align}
        \label{j24-5}\CC_0[-1]\to A_s\to P_s\to \CC_0,\\
        \label{j24-6}\CV\to B_s\to Q_s\to \CV[1].
    \end{align}
    
\end{enumerate}
\end{lemma}

\begin{proof}
The first statement \ref{j24-3} follows from the definition.

For \ref{j24-4}, 
Applying the octahedral axiom for the commutative diagram:
\[\xymigisita{A_{s-1}}{\CW[1]}{\CV[1]},\]
we obtain the distinguished triangle (\ref{j24-5}).
We also get (\ref{j24-6}) in the same way (or by applying the Fourier transformation to (\ref{j24-5})).

\end{proof}

\begin{remark}\label{nov5-3}
\begin{enumerate}
    \item 
    The object $P_s$ is isomorphic to the underived push-forward of $A_s$  along the inclusion $W\hookrightarrow V$.
    \item \label{nov5-2}
    The objects $A_s[1]$, $B_s[1]$, $P_s[1]$ ($s\geq 1$), $Q_s[1]$ ($s\geq 2$)
    and $\CC_0$ are perverse sheaves on $V$.
    It is known that any perverse sheaf on $V$ with possible singular points at $0$ can be decomposed into a direct sum of several perverse sheaves of these types.
    This fact is a special case of Lemma~\ref{nov5-1} below.
\end{enumerate}
\end{remark}

Next, we define some objects in $\sh{V,0}$ which has singular points also at $m$.

\begin{definition-lemma}\label{oct31-4}
    \begin{enumerate}
    \item We define objects $\ovA{1}, \undA{1}\in \shVzm$ as
    \[\ovA{1}:=\CWT,\quad \undA{1}:=\RG{T}\CW.\]
\item\label{s10-5} For $s\in \ZZ_{\geq 2}$, we take $\ovA{s}\in \shVzm$ (resp. $\undA{s}\in \shVzm$) inductively so that it fits into the distinguished triangle:
        \[\ovA{s-1}\to \ovA{s}\to \CW\to \ovA{s-1}[1]\]
        \[(\mbox{resp.}\  \CW\to \undA{s}\to \undA{s-1}\to \CW[1]),\]
where $\CW\to \ovA{s-1}[1]$ (resp. $\undA{s-1}\to \CW[1]$) is defined by the following diagrams
    \[\xymigidt{\CW}{\ovA{s-2}}{\ovA{s-1},}{\CW}{\emptyset}{1}\quad  \xyhidaridt{\CW}{\undA{s-1}}{\undA{s-2}}{\CW,}{1}{\emptyset}\]
where $\ovA{0}:=\CC_m[-1]$, $\undA{0}:=\RG{m}\CW[1]$ if $s=2$.

\item \label{s10-6}For $s\geq 1$, 
there is a morphism $\RG{T}\CW\to \ovA{s-1}[1]$ such that 
the diagram below commutes:
\[\xymatrix{\RG{T}\CW\ar[r]\ar[rd]& \ovA{s-1}[1] \ar[d]\\&\CW[1]},\]
where $\RG{T}\CW\to \CW[1]$ is the morphism $\und{A_1}\to \CW[1]$ appeared in \ref{s10-5} for $s=2$. 
Remark that the dimension of the space of such morphisms is grater than $1$.   
We fix one such morphism and take $\ovunA{s}\in \shVzm$ so that it fits into the distinguished triangle:
\[\ovA{s-1}\to \ovunA{s}\to \RG{T}\CW\to \ovA{s-1}[1].\]

\item We define $\ove{B_s}\in \shVzm$ as
\[\ove{B_s}:=\RG{W}\ove{A_s}.\]

\item We set $\ove{P_1}:=\bC_T\in \shVzm$.
For $s\geq 2$, we take
$\ove{P_s}\in \shVzm$ inductively so that it fits into the distinguished triangle:
\[\ove{P_{s-1}}\to \ove{P_s}\to \CW\to \ove{P_{s-1}}[1],\]
where $\CW\to \ove{P_{s-1}}[1]$ is the one corresponding to $\CW\to \ove{A_{s-1}}[1](=(\ove{P_{s-1}})_{W}[1])$.

\item 
We set $\und{P_1}:=\RG{T}\CV\in \shVzm$.
For $s\geq 2$, we take $\und{P_s}\in \shVzm$ so that it fits into the distinguished triangle:
\[\CV\to \und{P_s}\to \undA{s-1}\to \CV[1],\]
where the morphism $\undA{s-1}\to \CV[1]$ is the one corresponding to
$\undA{s-1}\to \CW[1]$.

\item 
For $s\geq 2$, by using the morphism $\RG{T}\CW\to \ove{P_{s-1}}[1]$ (corresponding to 
$\RG{T}\CW\to\ove{A_{s-1}}[1]$),
we take $\und{\ove{P_s}}\in \shV$ so that it fits into
\[\ove{P_{s-1}}\to \und{\ove{P_s}}\to \RG{T}\CW\to \ove{P_{s-1}}[1].\]

\item We take $\wt{\und{P_1}}\in \shVzm$ or $\sh{\PR,\{0,m,\infty\}}$ so that it fits into the distinguished triangle:
\begin{align*}
    \CC_T\to \wt{\und{P_1}}\to \CC_m[-1]\to  \CC_T[1]
\end{align*}
where $\CC_m[-1]\to  \CC_T[1]$ is a nontrivial morphism unique up to scaling.

\item We set
\begin{align*}
    \ove{\calL_s}:=\ove{A_s}|_{W}, \quad   \und{\calL_s}:=\und{A_s}|_{W},\ \mbox{and}\quad \und{\ove{\calL_s}}:=\ovunA{s}|_{W}. 
\end{align*}

\item \label{oct31-2}We have
\begin{align*}
    A_s|_{W}\simeq B_s|_{W}&\simeq P_{s}|_{W}\simeq \calL_s,\\
     Q_s|_{W}&\simeq \calL_{s-1},\\ 
\ove{A_s}|_{W}\simeq \ove{B_s}&|_{W}\simeq \ove{P_{s}}|_{W}  \simeq \ove{\calL_s},\quad\mbox{and} \\
\und{A_s}|_{W}\simeq &\und{P_{s}}|_{W} \simeq \und{\calL_s}.
\end{align*}

    \end{enumerate}
\end{definition-lemma}

\begin{remark}
\begin{enumerate}
\item 
All the objects $\ove{A_s}[1]$, $\und{A_s}[1]$, $\ove{\und{A_s}}[1]$, $\ove{B_s}[1]$, 
$\ove{P_s}[1]$, $\und{P_s}[1]$, $\ove{\und{P_s}}[1]$ are perverse sheaves on $V$.

\item
     As remarked, the space of morphism $\RG{T}\CW\to \ovA{s-1}[1]$ with the property in \ref{s10-6} is not $1$-dimensional and there is an ambiguity in the definition of $\ovunA{s}$ in this sense. 
     Accordingly, the definition of $\mathcal{H}_k^i$ which appears later also contains an ambiguity.
Nevertheless, the ``limit object" $\mathcal{H}^\infty_k$ is no longer so.
\end{enumerate}
 \end{remark}

Let us determine the specialization at $0$ of the objects defined above. 
Let $t$ be a coordinate of $V(=\CC)$ such that $t=0$ corresponds to the point $0$.
Recall that a perverse sheaf $\scF$ on $V$ with singular points at most at $0$ is determined by the tuple $(\psi_t\scF,\phi_{t,1}\scF,\mathrm{can},\Var)$, where $\psi_t\scF$ is the nearby cycle of $\scF$ (with the monodromy automorphism),
$\phi_{t,1}\scF$ is (resp. $\psi_{t,1}\scF$) the unipotent vanishing (resp. nearby) cycle of $\scF$, $\can$ is the morphism $\can \colon \psi_{t,1}\scF\to \phi_{t,1}\scF$
and
$\Var$ the morphism $\Var \colon \phi_{t,1}\scF\to \psi_{t,1}\scF$.

\begin{lemma}[Proposition.8.6.3 of \cite{KS}]
For a perverse sheaf $\scF$ on $V$ (which may have singular points at points other than $0$),
the specialization $\nu_0(\scF)$ (the object on a tangent fiber $T_0V$) at $0$
is again a perverse sheaf on $T_0V$ with the only singular point at $0$.
Moreover, under the isomorphism $T_0V\simeq V$ by the coordinate $t$,
$\nu_0(\scF)$ corresponds to the tuple 
$(\psi_{t}\scF, \phi_{t,1}\scF, \can, \Var)$.
\end{lemma}

By using this lemma, we have the following.
\begin{lemma}
\label{oct31-3}
We have
\begin{align*}
    \nu_0(\ove{A_s})\simeq     \nu_0(\und{A_s})\simeq     \nu_0(\ove{\und{A_s}})\simeq A_s,\\
    \nu_0(\ove{B_s})\simeq B_s,\quad\quad\quad \mbox{and}\\
    \nu_0(\ove{P_s})\simeq \nu_0(\und{P_s})\simeq \nu_0(\ove{\und{P_s}})\simeq P_s.
    \end{align*}
where each right hand side is regarded as the object on the tangent fiber $T_0V(\simeq \CC)$ (with the origin $0$).
\end{lemma}

We construct objects in $\mu \mathrm{sh}_C(X_\Gamma)$ by gluing up the objects defined in the above.
Let us use the notation in Definition~\ref{m22-1-2}.
By \ref{oct31-2} of Definition-Lemma~\ref{oct31-4},
we obtain the objects in $\sh{\PR,\{l,m,r\}}$:
\[(A_s,B_s),\quad (\und{P_s},\und{A_s}),\quad (\ove{B_s},\ove{P_s}), \mbox{\ etc.}\]
Moreover, by Lemma~\ref{oct31-3},
we obtain the following isomorphisms:
\[\mathrm{FL}(\nu_0(\und{A_s}))\simeq B_s,\quad \mathrm{FL}(\nu_0(\ove{P_{s}}))\simeq Q_s, \mbox{\ etc.}\]
Using these facts, we define $\block{j}$ in $\mu \mathrm{sh}_C(X_\Gamma)$, which is a ``building block" of the object corresponding to the cotangent fiber through the Ganatra--Pardon--Shende equivalence.
\begin{definition}\label{s11-8}
Here we use the expression (\ref{m27-3n-2}).
\begin{enumerate}
    \item \label{s11-8-1}
    If $n\geq 4$ and $1<j<n/2$,
we define an object $\block{j}\in  \mu \mathrm{sh}_{C_{\{\boldsymbol{m}\}}}(X_\Gamma)$ as
\begin{align*}
    ((P_1,Q_2),(P_2,Q_3),\dots, &(P_{j-1},Q_{j}),\\
    &(\und{P_j},\und{A_j}), (B_j,A_j), \dots,(B_j,A_j),(\ove{B_j},\ove{P_j}),\\
    &\qquad \qquad \qquad \qquad \qquad \qquad (Q_j,P_{j-1}),(Q_{j-1},P_{j-2}),\dots,(Q_2,P_1)),
\end{align*}
where $(\und{P_j},\und{A_j})$ is on the $j$-th $\PR$ and $(\ove{B_j},\ove{P_j})$ is on $(n-j+1)$-th $\PR$.
For $n/2<j<n$, we define $\block{j}$ as $\block{n-j+1}^{\circ}$, where $\circ$ is defined as (\ref{m28-3}).

\item \label{s11-9} If $n$ is even: $n=2n_0$, we define $\block{n_0}\in  \mu \mathrm{sh}_{C_{\{\boldsymbol{m}\}}}(X_\Gamma)$ as
\begin{align*}
    ((P_1,Q_2),\dots, &(P_{n_0-1},Q_{n_0}),
    (\und{P_{n_0}},\und{A_{n_0}}),(\ove{B_{n_0}},\ove{P_{n_0}}),
    (Q_{{n_0}-1},P_{{n_0}-2}),\dots,(Q_2,P_1)),
\end{align*}
and $\block{{n_0}+1}$ as $\block{{n_0}}^{\circ}$

\item \label{s11-8-3} If $n\geq 2$, we define $\block{1}\in  \mu \mathrm{sh}_{C_{\{\boldsymbol{m}\}}}(X_\Gamma)$ as
\begin{align*}
    ((\und{P_1},\und{A_1}),(B_1,A_1),\dots,(B_1,A_1),(\ove{B_1},\ove{P_1})),
\end{align*}
and $\block{n}$ as $\block{1}^{\circ}$.

\item \label{s11-8-4} If $n$ is odd: $n=2{n_0}+1$ (${n_0}\in \ZZ_{\geq 1}$), we define $\block{{n_0}+1}\in  \mu \mathrm{sh}_{C_{\{\boldsymbol{m}\}}}(X_\Gamma)$ as
\begin{align*}
      ((P_1,Q_2),\dots, &(P_{{n_0}},Q_{{n_0}+1}),
    (\und{\ove{P_{{n_0}+1}}},\und{\ove{P_{{n_0}+1}}}),
    (Q_{{n_0}+1},P_{{n_0}}),\dots,(Q_2,P_1)).
\end{align*}

\item \label{s11-8-5} If $n=1$,
we define $\block{1}\in  \mu \mathrm{sh}_{C_{\{\boldsymbol{m}\}}}(X_\Gamma)$ as
\[\block{1}=(\wt{\und{P_1}},\wt{\und{P_1}}).\]

\item In each of the above cases,
we define $\block{j}'\in  \mu \mathrm{sh}_{C_{\{\boldsymbol{m}\}}}(X_\Gamma)$ as follows: in the case $n\geq 2$ and $1< j< n/2$,
then $\block{j}'$ is
\begin{align*}
    ((P_1,Q_2),(P_2,Q_3),\dots, &(P_{j-1},Q_{j}),\\
    &({P_j},{A_j}), (B_j,A_j), \dots,(B_j,A_j),(\ove{B_j},\ove{P_j}),\\
    &\qquad \qquad \qquad \qquad \qquad \qquad (Q_j,P_{j-1}),(Q_{j-1},P_{j-2}),\dots,(Q_2,P_1)).
\end{align*}
For the other cases, we define $\block{j}'$ similarly.

\end{enumerate}
    
\end{definition}

\begin{lemma}\label{m29-2}
    For any $H\in \mu \mathrm{sh}_C(X_{\Gamma})$ and $1\leq j\leq n$,
    we have the vanishing:
    \begin{align}\label{jl18-4}
        H^{k}\HOM_{\mu \mathrm{sh}_C(X_{\Gamma})}(\block{j},H)=0\qquad (k\in \ZZ).
    \end{align}
\end{lemma}
Most of this appendix is devoted to proving this theorem.
We restate the assertion in a concrete form as follows.

\begin{lemma}\label{dec20-6}
Lemma~\ref{m29-2} is equivalent to the morphism (\ref{s11-7-2}) substituting $\block{j}$ for $H$ and $H$ for $H''$:
    \begin{align}\label{oct31-5}
\xymatrix{
\bigoplus_{i=1}^{n+1}H^k\mathrm{Hom}_{U_i}(\block{j},H)
    \ar[r]&
    \bigoplus_{i=1}^{n}H^k\mathrm{Hom}_{U_{i}\cap U_{i+1}}(\block{j},H)\\
(b_1,(b_{2,l},b_{2,r}),\dots,(b_{n,l},b_{n,r}),b_{n+1})\ar@{|->}[r]\ar@{}[u]|-{\rotatebox{90}{$\in$}}& 
(b_1|_{W}-b_{2,l}|_{W},b_{2,r}|_{W}-b_{3,r}|_{W},\dots, b_{n,r}|_{W}-b_{n+1}|_{W})\ar@{}[u]|-{\rotatebox{90}{$\in$}}
}        
    \end{align}
is isomorphic for $k\in \ZZ$,
where we use the notation defined in (\ref{s11-5}).
\end{lemma}

\begin{proof}
Recall the remark just below Lemma~\ref{s11-6-2}.
Except for the case $n=2m$ (\ref{s11-9} of Definition~\ref{s11-8}),
$\block{j}$ (and $H\in \mu \mathrm{sh}_C(X_\Gamma)$) satisfies either the condition \ref{oct30-1} or \ref{oct30-2} of Corollary~\ref{jan24-1cor}.
Hence, the assertion follows in this case.

Consider the case $n=2n_0$.
We can define a morphism $\ove{B_{n_0}}\to B_{n_0}$ such that the specialization of this morphism is the identity.
In fact, this is made by applying $\RG{W}$ to the morphism $\ove{A_{n_0}}\to A_{n_0}$.
Then, for any morphism $B_{n_0}\to F$ for $F\in \shVz$,
the specialization of the composition $(B_{n_0}\to F)\circ (\ove{B_{n_0}}\to B_{n_0})$ is the original $B_{n_0}\to F$ (regarded as an object on $T_0V$).
Hence, we can apply Lemma~\ref{s11-6-2} and the argument just below it to our situation and we thus obtain the assertion.
\end{proof}

\subsection{Proof of Lemma~\ref{m29-2}} \label{nov23-1}

In this subsection,
we show Lemma~\ref{m29-2}.
Since we will use many lemmas for the proof, we place all the items (including definitions and lemmas) in the next subsection, which we will use freely in this subsection.

Let us show the surjectivity and injectivity of (\ref{oct31-5})
for 
\[H=((F^1,G^1),\dots,(F^n,G^n)).\]
Without loss of generality, we may assume $k=0$.
We write 
\[\block{j}=((X^1,Y^1),\dots,(X^n,Y^n)).\]

\subsubsection{For $n\geq 4$ and $2\leq  j\leq    n/2 $}
We omit the proof for the case $j=n/2$ (for even $n$) which is similar. So we assume $j< n/2$. This is the case of \ref{s11-8-1} of Definition~\ref{s11-8}.

\subsubsection*{Surjectivity}\ 
Take an $n$-tuple of morphisms
\[(a_1,\dots,a_n)\in \bigoplus_{1\leq i\leq n}H^0\HOM_{U_{i}\cap U_{i+1}}(\block{j},H),\]
i.e., an element of the right hand side of (\ref{oct31-5}).
If $(a_1,\dots,a_n)$ and another $(\wt{a_1},\dots,\wt{a_n})$ are equal in the cokernel of (\ref{oct31-5}),
we write 
\[(a_1,\dots,a_n)\sim (\wt{a_1},\dots,\wt{a_n}).\]
To show the surjectivity is equivalent to showing
\[(a_1,\dots,a_n)\sim (0,\dots,0).\]
By Proposition~\ref{nov5-1}, we fix a decomposition
\begin{align*}
    F^i=\bigoplus_{k}\scF^{i,k}[k],\quad     G^i=\bigoplus_{k}\scG^{i,k}[k],
\end{align*}
where each $\scF^{i,k}, \scG^{i,k}$ is a direct sum of unipotent perverse sheaves.
By Lemma~\ref{dec4-1}, 
we have a constraint on the nilpotency and they satisfy the condition $(N_s)$ in Definition~\ref{dec25-2}.
We use them implicitly below.
Since each $X^i[1]|_W$, $Y^i[1]|_W$ is a perverse sheaf,
$a_i\colon X^i|_W\to \scF^{i,k}|_W[k]$ and $a_i\colon Y^i|_W\to \scG^{i,k}|_W[k]$ vanish if $k\neq -1,0$ by Lemma~\ref{nov22-1}.
So, we write each $a_i$ as a direct sum:
\[a_i=\begin{cases}
    &a_{i,l}'+a_{i,l}''\in \Hom(X^i|_{W}, \scF^{i,-1}[-1]|_{W})\oplus \Hom( X^i|_{W}, \scF^{i,0}|_{W})\\
    &a_{i,r}'+a_{i,r}'' \in \Hom(Y^i|_{W}, \scG^{i,-1}[-1]|_{W})\oplus \Hom( Y^i|_{W}, \scG^{i,0}|_{W})
\end{cases}.\]
Since any automorphism of $F^i|_W\cong G^i|_W$ has a triangular form with respect to the degree (by the similar results to Lemma~\ref{nov22-1}), we always have an identification $a_{i_l}''=a_{i,r}''$.
Hence $a_{i,l}'=0$ then so is $a_{i,r}'$, and vice versa.
In this case, we just write 
\[a_i=a_i''.\]

By 1 of Definition~\ref{s11-8}, the domain of $a_{j,l}'$ is $\underline{A_j}|_W$. Also, $\cF^{i,-1}[-1]|_W$ and $\cG^{i,-1}[-1]|_W$ are a direct sum of local systems of the form of $A_s|_W$. 
So, due to (\ref{jl2-2}) of Lemma~\ref{nov7-12},
we have $a_{j,l}'=a_{j,r}'=0$, so we have $a_j=a_{j,l}''=a_{j,r}''$.
By applying Lemma~\ref{j24-1} to $a_{j-1,l}'$, 
we have
\begin{align}\label{jl8-9}
    (a_1,\dots,a_n)= (a_1,\dots,a_{j-1},a_{j,l}'',\dots,a_n)\sim (a_1,\dots,\wt{a_{j-2}},a_{j-1}'',a_{j,l}'',\dots,a_n),
\end{align}
where $\wt{a_{j-2}}$ is some $X^{j-2}|_{W}\to \scF^{j-2}|_{W}$.
We omit the tilde of $\wt{a_{j-2}}$ in the below.
Repeating this procedure, the right hand side of (\ref{jl8-9}) takes the form of
\begin{align*}
(a_1'',\dots,a_{j-1}'',a_j'',\dots,a_n).
\end{align*}
By using Lemma~\ref{j27-2} repeatedly, we have
\begin{align*}
(a_1,\dots,a_n)\sim (0,\dots,0,a_{j-1}'',a_j'',\dots,a_n).  
\end{align*}
By Lemma~\ref{nov7-5}, we get
\begin{align*}
(a_1,\dots,a_n)\sim (0,\dots,0,a_j'',\dots,a_n).
\end{align*}
Set $j':=n-j+1$. 
Since $a_i$ for $j\leq i\leq j'$ obviously can be extended to 
$Y^i(=A_{j})\to G^i$ on $V_r$ by the zero extension,
we obtain
\begin{align*}
(a_1,\dots,a_n)\sim (0,\dots,0,a_{j'-1},\dots,a_n).
\end{align*}
By Lemma~\ref{jl8-21},
we have
\begin{align*}
(a_1,\dots,a_n)\sim (0,\dots,0,0,a_{j'},\dots,a_n).
\end{align*}
By Lemma~\ref{jl8-22},
we get
\begin{align*}
(a_1,\dots,a_n)\sim (0,\dots,0,0,0,a_{j'}'',{a_{j'+1}},\dots,a_n).
\end{align*}
Then, using Lemma~\ref{j24-1} repeatedly again,
we arrive at the form of 
\begin{align*}
(a_1,\dots,a_n)\sim (0,\dots,0,0,0,a_{j'}'',\dots,a_{n}'').
\end{align*}
By Lemma~\ref{j27-2} and Lemma~\ref{jl3-14},
we finally get the form of
\begin{align*}
(a_1,\dots,a_n)\sim (0,\dots,0,a_{j'}'',0,\dots,0).
\end{align*}
Then, by Lemma~\ref{jl18-23}, we have
\begin{align*}
(a_1,\dots,a_n)\sim (0,\dots,0,0,0,\dots,0).
\end{align*}
This proves the surjectivity.

\subsubsection*{Injectivity}
Take an $n+1$-tuple of morphisms from the left hand side of (\ref{oct31-5}),
\[(b_1,\dots,b_{n+1})\in \bigoplus_{1\leq i\leq n+1}H^0\HOM_{{U_i}}(\block{j},H),\]
such that its image is zero.
For $n\geq 2$ and $2\leq i\leq n$, we write $b_i=(b_{i,l},b_{i,r})$,
where $b_{i,l}\colon Y^{i}\to G^i$ and $b_{i,r}\colon X^{i+1}\to F^{i+1}$.
For $i=1$ or $i=n+1$, we write $b_{1}=b_{1,r}$ and $b_{n+1}=b_{n+1,l}$.
As in the proof of the surjectivity,
we can decompose $b_{i,l}$ and $b_{i,r}$ as
\begin{align*}
    b_{i,l}=b_{i,l}'+b_{i,l}'',\mbox{ and}\  b_{i,r}=b'_{i,r}+b''_{i,r},
\end{align*}
where $b_{i,l}'\colon Y^i\to \scG^{i,-1}[-1]$, $b_{i,l}''\colon Y^i\to \scG^{i,0}$, $b_{i,r}'\colon X^i\to \scF^{i,-1}[-1]$ and $b_{i,r}''\colon X^i\to \scF^{i,0}$.
If $b_{i,l}'=0$ and $b_{i,r}'=0$,
we just write 
\[b_{i}=b_{i}''.\]
By Lemma~\ref{jl8-30},
we have
\[b_{j}=b_{j}''.\]
By the assumption on $b$, we have
\[\mathrm{FL}(\nu_l(b_{j,r}))|_{W}=b_{j,l}|_{W}=b_{j-1,r}|_{W}.\]
Therefore, we have
\[b_{j-1,r}'|_{W}=0.\]
By the uniqueness in Lemma~\ref{j24-1},
we have $b_{j-1,r}'=0$, and we thus obtain
\[b_{j-1}=b_{j-1}''.\]
By repeating this argument, 
we have
\[(b_1,\dots,b_{n+1})=(b_1'',\dots,b_j'',b_{j+1},\dots, b_{n+1}).\]
Note that $\cF^{1,0}$ is a direct sum of constant sheaves and $\bC_{V_l}\rightarrow \bC_{V_l}[1]$ is zero.
Hence $b_1=b_1'$, and
we have
\[(b_1,\dots,b_n)=(0,b_2'',\dots,b_j'',b_{j+1},\dots, b_{n+1}).\]
Hence, we have $b_{2,l}|_{W}=b_{2,l}''|_{W}=0$.
By the uniqueness in Lemma~\ref{j27-2}, 
this implies $b_{2}=0$.
By repeating this argument, 
we obtain
\[(b_1,\dots,b_{n+1})=(0,\dots,0,b_j'',b_{j+1},\dots, b_{n+1}).\]
By Lemma~\ref{nov7-5}, we also have $b_j''=0$.
Set $j':=n-j+1$.
Since $b_{i}=0$ is equivalent to $b_{i,l}=0$ for $j+1\leq i\leq j'-1$, 
we inductively obtain
\[(b_1,\dots,b_{n+1})=(0,\dots,0,b_{j'},\dots, b_{n+1}).\]
By Lemma~\ref{jl8-22}, we have 
\[b_{j'}=b_{j'}'',\]
which implies
\[b_{j'+1,l}'|_{W}(=b_{j',r}'|_{W})=0.\]
Therefore, 
by Lemma~\ref{jan24-2},
we have
\[b_{j'+1}=b_{j'+1}''.\]
By using Lemma~\ref{j24-1} repeatedly again, we have 
\[(b_1,\dots,b_{n+1})=(0,\dots,0,b_{j'}'', b_{j+1}'',\dots, b_{n+1}'').\]
Since $b_{n+1}=b_{n+1}'$, we have $b_{n+1}''=0$.
By Lemma~\ref{j27-2}, we obtain
\[(b_1,\dots,b_{n+1})=(0,\dots,0,b_{j'}'',b_{j'+1}'',0,\dots, 0).\]
Since $b_{j',l}|_{W}=0$,
we have $\nu_l(b_{j'}'')=0$.
Moreover, by Lemma~\ref{j27-2},
we also have $\nu_r(b_{j'+1}'')=0$.
Therefore, by Lemma~\ref{jl18-23},
we obtain
$b_{j'}''=0$ and $b_{j'+1}''=0$.
This proves the injectivity.

\subsubsection{For $n\geq 2$ and $j=1$}

We may assume $n\geq 3$. 
We can show the case $n=2$ in the same way. This is the case \ref{s11-8-3} of Definition~\ref{s11-8}.
We use the same notation as in the proof of Lemma~\ref{m29-2} for $n\geq 4$ and $2\leq j\leq n/2$.

\subsubsection*{Surjectivity}
Take an $n$-tuple of morphisms 
\[(a_1,\dots,a_n)\in \bigoplus_{1\leq i\leq n}H^0\HOM_{{U_i\cap U_{i+1}}}(\block{1},H).\]
Since 
the morphism $a_1 $
can be extended to $\und{A_1}\to G^1$ on $V_r$,
we have 
\[(a_1,\dots,a_n)\sim (0,\wt{a_2},\dots,a_n),\]
where $\wt{a_2}$ is some morphism.
By repeating a similar procedure, we have
\begin{align*}
    (a_1,\dots,a_n)\sim (0,\dots,0,\wt{a_{n-1}},a_{n}),
\end{align*}
for some $\wt{a_{n-1}}$.
We extend ${\wt{a_{n-1}}}$ to $\wt{\wt{a_{n-1}}}\colon A_{1}\to G^{n-1}$.
By using the morphism 
\[\mathrm{FL}(\wt{\wt{a_{n-1}}})\circ (\ove{B_1}\to B_1)\]
where $(\ove{B_1}\to B_1)$ is a nonzero morphism, we obtain
\[(a_1,\dots,a_n)\sim (0,\dots,0,0,\wt{a_{n}}),\]
for some $\wt{a_n}$.
Since the restriction of $\CVT\to \CV$ to $W$ is $\CWT\to \CW$
and $(\wt{a_n})_r'$ (for a decomposition $\wt{a_n}=(\wt{a_n})_r'+(\wt{a_n})_r''$) is a direct sum of some $\CWT\to \CW$,
we have 
\[(a_1,\dots,a_n)\sim (0,\dots,0,0,(\wt{a_{n}})_r'').\]
By \ref{jl6-9} of Lemma~\ref{jl8-1}, Lemma~\ref{nov12-3-1} and Lemma~\ref{jl8-42},
we obtain
\[(a_1,\dots,a_n)\sim (0,\dots,0),\]
This proves the surjectivity.

\subsubsection*{Injectivity}
Take an $n+1$-tuple of morphisms 
\[(b_1,\dots,b_{n+1})\in \bigoplus_{1\leq i\leq n+1}H^0\HOM_{{U_i}}(\block{1},H),\]
such that its image by (\ref{oct31-5}) is zero.
Since there is no non-zero $\RG{T}\CV\to \CV[k]$ ($k\in \ZZ$),
we have $b_1=0$.
Therefore, $b_{2,l}|_{W}=0$, and hence $b_{2}=0$.
By repeating this,
we have
\[(b_1,\dots,b_{n+1})=(0,\dots,0,b_n,b_{n+1}).\]
Since $b_{n,r}\colon \ove{B_1}\to F^{n}$ satisfies
$\nu_l(b_{n,r})=0$,
we have $b_{n,r}'=0$ by the first assertion of Lemma~\ref{nov12-3-1}, i.e. $b_n=b_n''$.
Hence, $b_{n+1}'|_{W}$ is also zero and so is $b_{n+1}'$.
Eventually, we obtain
\[b''_{n,r}|_{W}=b_{n+1}''|_{W},\]
with $\nu_l(b''_{n,r})=0$.
By the second assertion of Lemma~\ref{nov12-3-1},
we also have 
\[\nu_l(b''_{n,r}\circ ( \CWT\to \ove{B_1}))=0.\]
On the other hand, by the second assertion of Lemma~\ref{jl8-42}, we have
\[\nu_r(b_{n+1}''\circ ( \CWT\to \CVT))=0.\]
Consequently, $\nu_l$ and $\nu_r$ of the zero extension of $b_{n,r}''|_{W}$ are both zero.
Therefore by Lemma~\ref{jl8-1},
$b''_{n,r}|_{W}=0$, and hence $b_{n+1}''|_{W}$ is also zero.
Then, Lemma~\ref{nov12-3-1} and Lemma~\ref{jl8-42},
we conclude that both $b''_{n,r}$ and $b_{n+1}''$ are zero.
This proves the injectivity

\subsubsection{For $n=2n_0+1$ $(n_0\geq 1)$ with $j=n_0+1$}
We proceed as in the previous cases and use the same notations. This is the case \ref{s11-8-4} of Definition~\ref{s11-8}.

\subsubsection*{Surjectivity} 
Take an $n$-tuple of morphisms
\[(a_1,\dots,a_n)\in \bigoplus_{1\leq i\leq n}H^0\HOM_{{U_i\cap U_{i+1}}}(\block{n_0+1},H).\]
By Lemma~\ref{j24-1} for $a_{n_0+2,r}'$ and $a_{n_0,l}'$,
we have
\[(a_1,\dots,a_n)\sim (a_1,\dots,a_{n_0-1},a_{n_0,l}'',a_{n_0+1},a_{n_0+2,r}'',a_{n_0+3},\dots  ,a_n).\]
Repeating this procedure,
we have
\[(a_1,\dots,a_n)\sim (a_1,a_2'',\dots,a_{n_0}'',a_{n_0+1},a_{n_0+2}'',\dots  ,a_{n-1}'',a_n).\]
Since $\CW\to \CW$ can be extended to $\CV\to \CV$,
we have
\[(a_1,\dots,a_n)\sim (a_1'',\dots,a_{n_0}'',a_{n_0+1},a_{n_0+2}'',\dots  ,a_n'').\]
By applying Lemma~\ref{j27-2} repeatedly,
we get
\[(a_1,\dots,a_n)\sim (0,\dots,0,a_{n_0}'',a_{n_0+1},a_{n_0+2}'',0,\dots  ,0),\]
for some $a_{n_0}''$ and $a_{n_0+2}''$.
For $a_{n_0}''$,
by Lemma~\ref{nov7-5},
there is a morphism $g\colon \und{P_{n_0+1}}\to \scF^{n_0+1,0}$ on $V_l$ (in $n_0+1$-th $\PR$)
such that $\mathrm{FL}(\nu_l(g))|_{W}=a_{n_0}''$.
Then, $g':=g\circ (\oveundP{n_0+1}\to \und{P_{n_0+1}})\colon \oveund{P_{n_0+1}}\to \scF^{n_0+1,0}$ satisfies $\mathrm{FL}(\nu_l(g'))|_{W}=a_{n_0}''$.
By the same argument for $a_{n_0+2}''$,
we have
\[(a_1,\dots,a_n)\sim (0,\dots,0,a_{n_0+1},0,\dots  ,0),\]
for some $a_{n_0+1}$.
By lemma~\ref{jl9-1}, we have $a_{n_0+1}=a_{n_0+1}''$.
Then, by Lemma~\ref{jl9-10},
we obtain 
\[(a_1,\dots,a_n)\sim (0,\dots,0),\]
which proves the surjectivity.

\subsubsection*{Injectivity}
Take an $n+1$-tuple of morphisms 
\[(b_1,\dots,b_{n+1})\in \bigoplus_{1\leq i\leq n+1}H^0\HOM_{{U_i}}(\block{n_0+1},H),\]
such that its image by (\ref{oct31-5}) is zero.
By Lemma~\ref{jl9-11},
we have $b_{n_0+1,r}'=0$, and hence $b_{n_0+1,l}'=0$.
Therefore, by using the uniqueness assertion of Lemma~\ref{j24-1},
we have $b_{n_0,r}'=0$.
Repeating this procedure,
we have
\[(b_1,\dots,b_{n+1})=(b_1,b_2'',\dots,b_{n}'',b_{n+1}).\]
Since $\CV\to \CV$ is zero if and only if its restriction to $W$ is zero,
the fact $b_1|_{W}=b_1'|_{W}=0$ implies $b_1=b_1'=0$.
Similarly, $b_{n+1}=0$.
From the uniqueness assertion of Lemma~\ref{j27-2},
we deduces $b_2=0$.
Therefore, we have
\[(b_1,\dots,b_{n+1})=(0,\dots,0,b_{n_0+1}'',b_{n_0+2}'',0,\dots,0).\]
By Lemma~\ref{j27-2} again,
$b_{n_0+1,l}''|_{W}=0$ implies $b_{n_0+1,l}''=0$, i.e.
$\nu_l(b_{n_0+1,r}'')=0$.
Similarly, we have $\nu_r(b_{n_0+2,l}'')=0$.
Since $b_{n_0+1,r}''|_{W}=b_{n_0+2,l}''|_{W}$,
by applying the uniqueness assertion of Lemma~\ref{jl9-10} we conclude that
$b_{n_0+1,r}''=0$ and $b_{n_0+2,l}''=0$,
i.e. $b_{n_0+1}''=0$ and $b_{n_0+2}''=0$.
This proves the injectivity.

\subsubsection{For $n=1$}
This is the case \ref{s11-8-5} of Definition~\ref{s11-8}.
In this case, 
(\ref{oct31-5}) is 
\[H^0\HOM_{{U_1}}(\block{1},H)\oplus H^0\HOM_{{U_2}}(\block{1},H)\to H^0\HOM_{{U_1\cap U_2}}(\block{1},H).\]
We will show it is bijective.
Let $H=(F^1,G^1)$.
Note that $F^1$ and $G^1$ are direct sum of some shifted constant sheaves.

\subsubsection*{Surjectivity}
Take $a\in H^0\HOM_{{W}}(\block{1}|_{W},F^1|_{W})$.
By Lemma~\ref{dec20-10}, $a$ is in fact $\wtundPo|_{W}\to \scF^{1,0}|_{W}$.
Since $\scF^{1,0}$ is a direct sum of some $\CV[1]$,
$\scF^{1,0}|_{W}$ is a direct sum of some $\CW[1]$.
Then, $a$ is decomposed into $a=a_1+a_2$ so that $\nu_l(a_1)=0$ and $\nu_r(a_2)=0$.
Then, \ref{jl9-16} of Lemma~\ref{dec20-10} implies that
$a$ is in the image of (\ref{oct31-5}).

\subsubsection*{Injectivity}
Take $(b_1,b_2)\in H^0\HOM_{{U_1}}(\wtundPo,F^1)\oplus H^0\HOM_{{U_2}}(\wtundPo,G^1)$ such that
its image is zero by (\ref{oct31-5}).
By \ref{jl9-15} and \ref{jl9-16} of Lemma~\ref{dec20-10},
we have $b_1=b_1''$, $b_2=b_2''$,
$\nu_l(b_1''\circ ((\wt{\undP{1}})_W\to \wt{\undP{1}}))=0$ and $\nu_r(b_2''\circ ((\wt{\undP{1}})_W\to \wt{\undP{1}}))=0$.
By \ref{jl9-18} of Lemma~\ref{dec20-10},
both 
$b_1''\circ ((\wt{\undP{1}})_W\to \wt{\undP{1}})$ and
$b_2''\circ ((\wt{\undP{1}})_W\to \wt{\undP{1}})$ are zero.
Hence,
by \ref{jl9-18} of Lemma~\ref{dec20-10} again,
we conclude that 
$b_1''$ and $b_2''$ are zero.

\subsubsection{For the general case}

We have already proved in the four cases for $n$ and $k$:
\begin{enumerate}
    \item $n\geq 4$ and $2\leq j\leq n/2$,
    \item $n\geq 2$ and $j=1$,
    \item $n=2n_0+1$ ($n_0\geq 1$) and $j=n_0+1$,
    \item $n=1$.
\end{enumerate}
Then, by the symmetry: $\block{j}=\block{n-j+1}^{\circ}$,
we conclude that all the remaining cases are also true.

\subsection{Lemmas for subsection~\ref{nov23-1}}
\label{dec20-3}

In this subsection, we gather several lemmas already used in the previous subsection.

\subsubsection{Morphisms between basic objects in $\Sh(\CC)$}

\begin{lemma}\label{jan27-102}
We have the following claims in the category $\HzshVz$.
\begin{enumerate}
\item \label{nov6-8-2}
For $s'\leq s$, we have diagrams:
\[\xymigidt{A_s}{\CW}{A_{s'},}{A_{s'-1}}{0,1}{\substack{0\times (s'-1)\\1\times (s'-1)}}\qquad
\xymigidt{A_s}{A_{s'-1}}{A_{s'}.}{\CW}{\substack{0\times (s'-1)\\1\times (s'-1)}}{0,1}
\]
Here, we follow the definition written just below Definition~\ref{s10-3-2}.

      \item  \label{nov6-2} For $s'<s$, we have diagrams:
\[\xyhidaridt{\CW}{A_s}{A_{s-1}}{A_{s'},}{1}{\substack{0\times s'\\1\times (s'-1)}}\qquad \xyhidaridt{A_{s-1}}{A_s}{\CW}{A_{s'}.}{\substack{0\times (s'-1)\\1\times s'}}{0}\]
In particular, a morphism $A_s\to A_{s'}[0]$ (resp. $A_s\to A_{s'}[1]$) always factors though $A_{s-1}\to A_{s'}[0]$ (resp. $A_{s-1}\to A_{s'}[1]$).

\item \label{nov6-4}
We have a diagram:
\[\xymigidt{\CC_0[-1]}{\CW}{A_s,}{A_{s-1}}{0}{1}\qquad
\xymatrix{\ &\phantom{\CW} \\\CC_0[-1] \ar[r]^-{\emptyset}& B_s. \\ &}
\]
In particular, for $s'<s$ and any morphism $A_s\to A_{s'}$, the composition $\CC_0[-1]\to A_s\to A_{s'}$ is zero.

\item \label{nov6-7}
We have
\[(A_s)_0=0,\qquad\mbox{and}\]
\[H^k((B_s)_{0})\simeq \left\{\begin{array}{cc}
     \CC&  k=0,1\\
     0& \mbox{othrewise}.
\end{array} \right.\]

    \end{enumerate}

    \end{lemma}

\begin{proof}
\ref{nov6-8-2},\ref{nov6-2},\ref{nov6-4} are proved by using the long exact sequences obtained by applying $\HOM$ functors to the distinguished triangles.
The first assertion of \ref{nov6-7} is clear.
By using the fact
\[H^k\RG{W}(\CW)_0\simeq \left\{\begin{array}{cc}
     \CC&  k=0,1\\
     0&\mbox{othrewise}
\end{array}\right.\]
 and
 the distinguished triangles obtained by applying $(\RG{W}(-))_0$ to (\ref{m16-3}) and (\ref{jan27-101}),
 we can show the second assertion of \ref{nov6-7}.
\end{proof}

\begin{lemma}\label{dec4-5}
    We have the following claims in the category $\HzshVz$.
\begin{enumerate}
    \item \label{nov6-9}
For $s'\leq s$, we have a diagram:
\[\xymigidt{A_s}{\CV}{P_{s'}.}{A_{s'-1}}{0,1}{\substack{0\times (s'-1)\\ 1\times (s'-1)}}\]
    
    \item \label{nov6-5}
We have diagrams:
\[\xymigidt{\CC_0[-1]}{\CV}{P_s,}{A_{s-1}}{\emptyset}{1}\qquad
\xymigidt{\CC_0[-1]}{\CC_0[-1]}{Q_s.}{B_{s-1}}{0}{\emptyset}
\]

\item \label{nov6-8}
We have
    \begin{align*}
        (P_s)_0&\simeq \CC_0,\qquad \mbox{and}\\
        (Q_s)_0&\simto \CC_0[-1].
    \end{align*}

\end{enumerate}
\end{lemma}

\begin{proof}
    For \ref{nov6-9}, since $A_s$ is a zero extension, to give $A_s\to P_{s'}[k]$ for $k\in \ZZ$ is equivalent to give $A_{s}\to A_{s'}[k]$.
Therefore, we have already known the dimension of the space of morphisms between $A_s$ to $P_{s'}[k]$ for each $k\in \ZZ$.
The diagrams immediately follows from this fact.

\ref{nov6-5} follows from \ref{nov6-4} of Lemma~\ref{jan27-102}.

For \ref{nov6-8},
the first assertion is clear.
For the second one, we consider the distinguished triangle:
\[\CC_0[-1]\to (Q_s)_0\to (B_{s-1})_0\to \CC_0.\]
Then, the result follows from the isomorphism $H^0((B_{s-1})_0)\simto H^0(\CC_0)$.

\end{proof}

\begin{lemma}\label{nov7-12}

We have the following claims in the category $\HzshVzm$.
    \begin{enumerate} 
        \item \label{j26-3}For $s\geq 1$, we have distinguished triangles:
\begin{align}
\label{nov7-3}    \RG{m}\CW\to A_s\to \undA{s}\to \RG{m}\CW[1], \\
    \RG{m}\CW\to P_s\to \und{P_s}\to \RG{m}\CW[1].\label{jan28-101}
\end{align}
Moreover, 
the morphism $\und{P_s}\to \RG{m}\CW[1]$ in the second triangle
fits into the commutative diagram:
\begin{align}\label{nov7-4}
    \xymigisita{\und{P_s}}{\und{A_{s-1}}}{\RG{m}\CW[1]},
\end{align}
where $\und{A_{s-1}}\to \RG{m}\CW[1]$ is the one in the first distinguished triangle (for $s-1$).

  \item \label{jl2-1}For $s\geq 1$, we have distinguished triangles: 
\begin{align*}
A_{s-1}\to \undA{s}\to \RG{T}\CW\to A_{s-1},\\
P_{s-1}\to \und{P_s}\to \RG{T}\CW\to P_{s-1}.   
\end{align*}

\item\label{nov6-13} 
For $1\leq s'\leq s$,
we have the diagram:
\[\xytandokusize{\und{A_s}}{A_{s'}}{1\times s'}{40}.\]

\item \label{nov6-14}
 For $1\leq s'<s$, We have the following diagrams:
        \begin{align}\label{jl2-5-2}
            \xytandokusize{\undA{s}}{\CV}{1}{20},\quad 
            \xytandokusize{\und{P_s}}{\CV}{\emptyset}{20}
            \end{align}
            \begin{align}
                            \label{jl2-4-1}
            \xyhidaridt{\CW}{\RG{T}\CW}{\RG{p}\CW[1]}{\CW,}{1}{\emptyset}\quad
            \xymatrix{&\\\RG{T}\CW\ar[r]^-{1\times s}&A_s,\\&}
            \end{align}
\begin{align}
                \xyhidaridt{A_{s-1}}{\und{A_s}}{\RG{T}\CW}{A_{s'}}{1\times s'}{\emptyset}\label{jl2-2}
        \end{align}
\end{enumerate}
\end{lemma}

\begin{proof}
    For \ref{j26-3}, 
    we define $A_{s}\to \undA{s}$ inductively so that
    we have a morphism between distinguished triangles:
\[\xymatrix{
\CW\ar[r]\ar[d]& A_s\ar[r]\ar[d]& A_{s-1}\ar[r]\ar[d]& \CW[1]\ar[d]\\
\CW\ar[r]& \undA{s}\ar[r]& \undA{s-1}\ar[r]& \CW[1].
}\]    
Then, applying the octahedral axiom to the diagram:
\[\xymigisita{\CW}{A_s}{\und{A_s}},\]
    we get the distinguished triangle~\ref{nov7-3}.
    The assertion for $P_s$ in \ref{j26-3} can be shown in the same way.
    
Similarly, we obtain \ref{jl2-1} by applying the octahedral axiom to the commutative diagrams:
    \[\xymigisita{A_{s-1}}{A_{s}}{\undA{s}},\quad \xymigisita{P_{s-1}}{P_{s}}{\und{P_{s}}}.\]

For \ref{nov6-13},
we have already seen the case where $s'=1$ in Definition-Lemma~\ref{oct31-4}.
For $s'\geq 2$, the assertion follows from induction with the following diagram:
\[\xymigidt{\und{A_s}}{\CW}{A_{s'}.}{A_{s'-1}}{1}{1\times (s'-1)}\]

We can show \ref{nov6-14} in the same way.

\end{proof}

\begin{lemma}\label{jan24-1lem}
\begin{enumerate}
\item \label{s10-8} For $s\geq 1$, we have distinguished triangles:
\begin{align}
    \CC_m[-1]\to \ove{A_s}\to A_s\to \CC_m,\label{jl2-7}\\
    \CC_m[-1]\to \ove{P_s}\to P_s\to \CC_m.
\end{align}

\item \label{s10-9} For $s\geq 1$, we have distinguished triangles:
\begin{align}
    \CC_{W\cap T}\to \ove{A_s}\to A_{s-1}\to \CC_{W\cap T}[1],\label{jl3-2}\\
    \CC_{T}\to \ove{P_s}\to A_{s-1}\to \CC_{T}[1],\label{jl3-3}
\end{align}
where the morphisms $A_{s-1}\to \CWT[1]$ and $A_{s-1}\to \CVT[1]$ in the triangles are (not unique) morphisms such that the following diagrams commute:
\[\xymatrix{\CW\ar[d]\ar[dr]&\\A_{s-1}\ar[r]&\CWT[1],}
\quad 
\xymatrix{\CW\ar[d]\ar[dr]&\\A_{s-1}\ar[r]&\CVT[1].}
\]

\item \label{nov8-2}
For $s\geq 1$, we have a diagram:
\[\xymigidt{\CWT}{\CW}{A_s.}{A_{s-1}}{0,1}{1\times s}\]

\item \label{nov7-10}
    For $s'\leq s$,
    we have a diagram:
    \[\xyhidaridt{\CWT}{\ove{A_s}}{A_{s-1}}{A_{s'},}{1\times (s'+1)}{\substack{0\times s'\\1\times (s'-1)}}
     \quad \xymigidt{\ove{A_s}}{A_{s'-1}}{A_{s'}.}{\CW}{\{0,1,1\}\times (s'-1)}{0,1,1}
    \]

\item \label{nov7-11}
For $s\geq 2$,
we have a diagram:
\[\xymigidt{\CC_0[-1]}{\CWT}{\oveA{s}.}{A_{s-1}}{0}{1}\]

\item \label{s10-10} For $s\geq 1$,
we have distinguished triangles:
\begin{align}
    \ove{A_s}\to \ove{P_s}\to \CC_0\to \ove{A_s}[1].\label{jl2-10}
\end{align}

    \end{enumerate}
\end{lemma}
\begin{proof}
For \ref{s10-8} and \ref{s10-9} are shown in the same way as in \ref{j26-3} of Lemma~\ref{nov7-12}.

For \ref{nov8-2}, the case $s=1$ is easy.
The case $s\geq 2$ follows from the (non-zero) commutative diagram:
\[\xymigisita{\CWT}{\CW}{\CW[1].}\]

For \ref{nov7-10},
    these diagrams are obtained by the fact that $\dim{H^0\HOMV(\ove{A_s},A_{s'})}=s'$ and $\dim{H^0\HOMV(\ove{A_s},A_{s'}[1])}=2s'$, which is deduced by the definition and induction.   

For \ref{nov7-11},
we note that the commutative diagram:
\[\xymigisita{\CC_0[-1]}{\CW}{\CW[1]}\]
induces a commutative diagram:
\[\xymigisita{\CC_0[-1]}{A_{s-1}}{\CWT[1],}\]
where $A_{s-1}\to \CWT[1]$ is the one in \ref{s10-9}.
Then, the assertion follows from the triangle (\ref{jl3-2}).

The remaining assertions are shown in the same way.

\end{proof}

We will introduce similar assertions for $\ovunA{s}$, $\ove{\und{P_s}}$ and $\wt{\und{P_1}}$  later.

\subsubsection{Nilpotent order of objects in $\sh{\CC}$}
We need to observe the conditions imposed on the objects in $\shVz$ appearing in $\mu \mathrm{sh}_C(X_\Gamma)$.
Let $F$ be an object in $\sh{V,0}$ on $V=\CC$ with a coordinate $t$.

\begin{definition}\label{dec25-2}

\begin{enumerate}
    \item We say that $F$ is \emph{unipotent} (at $0$) if 
    there exists $\ell\in \ZZ_{\geq 1}$ such that we have $(T-1)^\ell=0$ on $\psi_{t}F$.
    \item 
    We say that $F$ has the \emph{nilpotent order} $\leq s$ ($s\in \ZZ_{\geq 0}$) (at $0$) if we have $(T-1)^{s}=0$ on $\psi_{t}F$.
    \item We say that $F$ satisfies the condition $(N_s)$ if 
the following two conditions hold:
\begin{enumerate}
    \item $F$ has the nilpotent order $\leq s$,
    \item The Fourier transform $\mathrm{FL}(\nu_{0}(F))$ of the specialization has the nilpotent order $\leq s-1$.
\end{enumerate}
\end{enumerate}

\end{definition}

\begin{remark}\label{dec12-3-1}
\begin{enumerate}
    \item If the condition $(b)$ in the definition of $(N_s)$ holds,
    ($a$) also automatically holds.
    However, we leave ($a$) in the condition for convenience.

\item    Assume that $F$ is unipotent. 
   Since we have $\Var\circ \can=\frac{1}{2\pi\sqrt{-1}}\log(T-1)(=:N)$,
   $F$ has the nilpotent order $\leq s$ if and only if $N^s=(\Var\circ \can)^s=0$.
   Therefore, it follows that
$F$ satisfies the condition ($N_{s}$) if and only if $(\Var\circ \can)^s=0$ and $(\can\circ \Var)^{s-1}=0$.   

\end{enumerate}

\end{remark}

\begin{lemma}\label{dec4-1}
For an object $((F^1,G^1),(F^2,G^2),\dots,(F^n,G^n))\in \mu \mathrm{sh}_C(X_\Gamma)$,
the objects $F^i$ and $G^{n-i+1}$ are unipotent and satisfy the condition ($N_{i}$) for $i\leq n/2$.  
In particular,
the nilpotent orders of $G^{i}$ and $F^{n-i+1}$ are $\leq i$.
Moreover, when $n$ is odd: $n=2n_0+1$,
both $F^{n_0}$ and $G^{n_0}$ satisfy the condition $(N_{n_0})$.
\end{lemma}

\begin{proof}
    Since $F^1$ does not have any singular points,
    it must have the nilpotent order $\leq 1$, and hence so is $G^1$.
    Therefore, $F^2$ satisfies the condition ($N_2$).
    The assertion follows by induction.
\end{proof}

\begin{definition}
    \begin{enumerate}
        \item If a perverse sheaf $\scF$ on $V$ has a singular point at most at $0$, then we say $\scF$ is \emph{monodromic}.
        \item We write the category of monodromic perverse sheaves on $V$ as $\PervVz$, which is a subcategory of the category $\PervV$ of perverse sheaves on $V$.
        \item We write the category of monodromic unipotent perverse sheaves on $V$ as $\PervVzu$
    \end{enumerate}
\end{definition}

The following lemma follows from the definition.
\begin{lemma}\label{dec25-1}
    \begin{enumerate}
\item \label{jan27-103}$A_s$ and $B_s$ have the nilpotent orders $\leq s$ and satisfy the condition $(N_{s+1})$.
\item \label{jan27-104}$P_s$ has the nilpotent order $\leq s$ and satisfies the condtion $(N_s)$.
\item \label{jan27-105}$Q_s$ has the nilpotent order $\leq s-1$ and satisfies the condition $(N_{s+1})$.
            \item A unipotent perverse sheaf $\scF$ on $V$ with a singular point at most at $0$ satisfies the condition \Ns if and only if
            $\scF$ is a direct sum of several $P_{s'}[1]$ (${s'}\leq s$), $A_{s''}[1]$, $B_{s''}[1]$, $Q_{s''}[1]$ (${s''}\leq s-1$), or $\CC_0$.     
    \end{enumerate}
\end{lemma}
\begin{proof}
\ref{jan27-103}, \ref{jan27-104} and \ref{jan27-105} follows directly from the definition.
    The final assertion follows from the fact \ref{nov5-2} in Remark~\ref{nov5-3}.  
\end{proof}

The following proposition will be proved in Appendix~\ref{dec3-4}, which we use frequently in this subsection.

\begin{proposition}[Proposition~\ref{oct29-1} in Appendix~\ref{dec3-4}]\label{nov5-1}
For an object $F\in \shVz$, assume that $F$ has the nilpotent order $\leq s$ ($\ZZ_{\geq 1}$).
Then, $F$ can be expressed as a direct sum of several shifted perverse sheaves:
\[F=\bigoplus_{k\in \ZZ}\bigoplus_{j\in J_k}\scF_{j}[k],\]
where $J_k$ is an index set and $\scF_j$ is one of the following:
$\CC_0$, $A_{s'}[1]$, $B_{s'}[1]$, $P_{s'}[1]$ ($1\leq s'\leq s$), $Q_{s''}[1]$ ($2\leq s''\leq s+1$).
In particular, if $F$ satisfies the condition $(N_s)$ for some $s\in \ZZ_{\geq 1}$,
then $\scF_{j}$ is a direct sum of several
$P_{s'}[1]$ ($s'\leq s$),
$A_{s''}[1]$, $B_{s''}[1]$, $Q_{s''}[1]$ ($s''\leq s-1$),
or $\CC_0$, with some shifts.
\end{proposition}

\begin{remark}
The similar statement does not hold for $F\in \sh{\PR}$.
\end{remark}

We sometimes use Proposition~\ref{nov5-1} with the following lemma, which is easily proved.

\begin{lemma}\label{nov22-1}
    For $\scF,\scG\in \mathrm{Perv}(V,0)$, there is no non-zero morphism $\scF\to \scG[k]$ for $k\neq 0,1$.
\end{lemma}

We introduce convenient ``decompositions" for $A_s[1]$, $B_s[1]$, $P_s[1]$, $Q_s[1]$ and $\CC_0$.
In order to deal with them simultaneously,
we state it in the following form.

\begin{definition-lemma}\label{j24-2-2} 
    Let $\scF\in \PervVzu$ be a unipotent monodromic perverse sheaf on $V$ with a coordinate $t$, with the can-var description $(\psi,\phi,c,v)$, i.e. $\psi:=\psi_{t}\scF$, $\phi:=\psi_{t}\scF$, $c\colon \psi\to\phi$ and $v\colon \phi\to \psi$. 
We define the perverse sheaves $\scF_1$, $\scF_2$
which correspond to the can-var descriptions 
\begin{align*}
    (\psi,\Ima{c},c,v),\ (0,\phi/\Ima{c}, 0,0),
\end{align*}
respectively.
Then, we have the following exact sequence (and hence also the corresponding distinguished triangles):
\begin{align}\label{dec4-4}
    0\to \scF_1\to &\scF\to \scF_2\to 0.
\end{align}
The perverse sheaf 
$\scF_1$ (resp. $\scF_2$) is a direct sum of 
some $A_s[1]$ or $P_s[1]$ for some $s\geq 1$
(resp. the skyscraper sheaves $\CC_0$).
Moreover, if $\scF$ satisfies the condition $(N_s)$,
$\scF_1$ is a direct sum of 
some $A_{s'}[1]$ ($s'<s$) or $P_{s'}[1]$ ($s'\leq s$).
\end{definition-lemma}

\begin{proof}
    The exactness of the sequence (\ref{dec4-4}) is clear.
We can describe the pair $(\scF_1,\scF_2)$ concretely according to the type of $\scF$ (see Remark~\ref{nov5-3}) as follows.
 \begin{enumerate}
    \item If $\scF=\CC_0$, then the tuple is $(0,\CC_0)$.
    \item If $\scF=A_s[1]$, then the tuple is $(A_s[1],0)$.
    \item If $\scF=B_s[1]$, then the tuple is $(P_s[1],\CC_0)$.
    \item If $\scF=P_s[1]$, then the tuple is $(P_s[1],0)$.
    \item If $\scF=Q_s[1]$, then the tuple is $(A_{s-1}[1],\CC_0)$. 
\end{enumerate}
The desired assertions follow from these concrete expressions.
\end{proof}

\subsubsection{Lemmas for $n\geq 4$ and $2\leq j\leq n/2$}
In the rest of this subsection, we gather some properties of extensions of morphisms between (shifted) perverse sheaves on $\CS$ for the previous subsection.

\begin{lemma}\label{j24-1}
For $s\in \ZZ_{\geq 1}$, $\scF\in \PervVzu$ with the condition $(N_{s})$,
the morphism $A_s\to P_s$ induces an isomorphism:
\begin{align}\label{dec4-2}
H^0\HOM_{V}(P_s,\scF[-1])\simto H^0\HOM_{V}(A_s,\scF[-1]). 
\end{align}
In particular,
for a morphism $f\colon \calL_s\to \scF[-1]|_{W}$,
    there exists a unique extension $g\colon P_{s}\to \scF[-1]$ of it, i.e. $f|_{W}=g|_{W}$.
\end{lemma}

\begin{proof}
We use the distinguished triangle (\ref{j24-5}).
Since there is no non-zero morphism $\CC_0\to \scF[-1]$ by Lemma~\ref{nov22-1}, the morphism (\ref{dec4-2}) is injective.
So it remains to show the surjectivity.
We use Definition-Lemma~\ref{j24-2-2}.
By the condition $(N_s)$,
we may assume $\scF=\scF_1$, and $\scF_1[-1]=A_{s'}$ ($s'<s$) or $\scF_1[-1]=P_{s'}$ ($s'\leq s$).
If $\scF[-1]=P_{s'}$, since there is no non-zero $\CC_0[-1]\to P_{s'}$ by \ref{nov6-5} of Lemma~\ref{dec4-5},
the composition $f \circ (\CC_0[-1]\to A_s)$ is zero and hence (\ref{dec4-2}) is surjective.
Assume $\scF[-1]=A_{s'}$ ($s'<s$).
Then, by \ref{nov6-4} of Lemma~\ref{jan27-102},
the composition $f \circ (\CC_0[-1]\to A_s)$ is zero,
and hence (\ref{dec4-2}) is surjective.
\end{proof}

\begin{lemma}\label{j27-2}
    Let $\scF\in \PervVzu$ be an object with the nilpotent order $\leq s$.
   Then, we have an isomorphism induced by $A_s(\simeq (Q_{s+1})_W)\to Q_{s+1}$:
    \begin{align}\label{nov6-10}
        H^0\HOM_{V}(Q_{s+1},\scF)\simto H^0\HOM_{V}(A_{s},\scF)(=H^0\HOM_{{W}}(\scL_{s},\scF|_{W})).
    \end{align}
In particular,
for a morphism $f\colon \calL_s\to \scF|_{W}$,
there exists a unique morphism
$g\colon P_{s+1}\to \mathrm{FL}(\scF)$ on $V^{*}$ 
such that 
$\mathrm{FL}(\nu_0(g))|_{W}=f$.
\end{lemma}

\begin{proof}
    By \ref{nov6-8} of Lemma~\ref{dec4-5},
we have a distinguished triangle:
 \begin{align}\label{j25-5}
        (Q_{s+1})_{W}(=A_s)\to Q_{s+1}\to (Q_{s+1})_0(=\CC_0[-1])\to A_s[1].
    \end{align}
Then, the surjectivity of (\ref{nov6-10}) follows from Lemma~\ref{nov22-1}.

To show the injectivity, it is enough to show that
for any morphism $h\colon \CC_0[-1]\to \scF$,
there is a morphism $h'\colon A_s[1]\to \scF$ such that $h'\circ (\CC_0[-1]\to A_{s}[1])=h$.
We use Definition-Lemma~\ref{j24-2-2}.
Since there is no $\CC_0[-1]\to \scF_2$,
we may assume $\scF=\scF_1$, and $\scF_1=A_{s'}[1]$ ($s'<s$) or $\scF_1=P_{s'}[1]$ ($s'\leq s$).
For $\scF=A_{s'}[1]$,
there is a morphism $A_s[1]\to A_{s'}[1]$ such that the following diagram commute:
\begin{align}\label{nov6-12}
    \xymatrix{A_s[1]\ar[d]\ar[r]&A_{s'}[1]\ar[d]\\\CW[1]\ar[r]^{\mathrm{id}}&\CW[1].}
\end{align}
Then, the desired assertion follows from it with \ref{nov6-4} of Lemma~\ref{jan27-102}.
For $\scF=P_{s'}[1]$,
by \ref{nov6-5} of Definition-Lemma~\ref{nov6-6},
$h$ corresponds to $\CC_0[-1]\to A_{s'-1}[1]$.
Therefore, combining it with \ref{nov6-9} of Definition-Lemma~\ref{nov6-6} and (\ref{nov6-12}), we obtained the conclusion.
\end{proof}

To obtain a variant of this lemma for $\und{P_{s+1}}$,
we prepare some lemmas.

\begin{lemma}\label{jl8-30}
For $\scF\in \PervVzu$,
there is no non-zero morphism $\und{P_s}\to \scF[-1]$.
\end{lemma}

\begin{proof}
Note that
morphisms
$\und{P_s}\to \CC_0[-1]$ and $\und{P_s}\to \CV$ are zero.
    Moreover, there is no non-zero morphism $\und{P_s}\to \CW$,
    since $\und{A_{s-1}}\to \CW$ and $\CV\to \CW$ are zero.
Therefore, 
any morphism from $\und{P_s}$ to $P_{s'}$ or $A_{s'}$ is zero. 
Hence, by applying Definition--Lemma~\ref{j24-2-2},
we obtain the desired result.   
\end{proof}

\begin{lemma}\label{nov7-5}
    For $s\geq 2$, $\scF\in \PervVzu$ with the condition $(N_s)$
and a morphism $f\colon \calL_{s-1}\to \mathrm{FL}(\scF)$ on $V^*$,
there exists a unique morphism
$g\colon \undP{s}\to \scF$ on $V$ such that $\mathrm{FL}(\nu_0(g))|_{W}=f$.
\end{lemma}
\begin{proof}
Since we already have Lemma~\ref{j27-2}, it suffices to show that the morphism ${P_s}\to \und{P_s}$ induces an isomorphism 
\begin{align}
    \label{nov7-1}
    H^0\HOM_{{V}}(\und{P_{s}},\scF)\simto H^0\HOM_{{V}}(P_{s},\scF)
\end{align}
by the distinguished triangle (\ref{jan28-101}) of Lemma~\ref{nov7-12}.
Since there is no non-zero morphism $\CC_m[-2](=\RG{m}\CC_W)\to \scF$,
the morphism (\ref{nov7-1}) is surjective.

To show the injectivity,
we use Definition-Lemma~\ref{j24-2-2}.
If $\scF=\scF_2$, 
the morphism (\ref{nov7-1}) is clearly bijective.
In the general case, consider the following diagram:
\[\xymatrix{0\ar[r]&H^0\HOM_V(\und{P_s},\scF_1)\ar[r]\ar[d]&H^0\HOM_V(\und{P_s},\scF)\ar[r]\ar@{->>}[d]&H^0\HOM_V(\und{P_s},\scF_2)\ar[d]^-{\rotatebox{90}{$\sim$}}\\
0\ar[r]&H^0\HOM_V({P_s},\scF_1)\ar[r]&H^0\HOM_V({P_s},\scF)\ar[r]&H^0\HOM_V({P_s},\scF_2).
}
\]
Here, the horizontal sequences are exact except for the right ends.
From this diagram, we may assume $\scF=\scF_1$.
Moreover, 
one can see that the injectivity of 
 (\ref{nov7-1}) for $\scF=P_{s'}[1]$ ($s'\leq s$) follows from
 that for $\scF=A_{s'}[1]$ ($s'<s$).

By the distinguished triangle (\ref{jan28-101}),
it suffices to see that the surjectivity of the morphism: 
\[\HOM_V(P_s[1],A_{s'}[1])\to \HOM_V(\RG{m}\CW[1],A_{s'}[1]).\]
Note that both of the dimensions of $\HOM_V(P_s[1],A_{s'}[1])$ 
and 
$\HOM_V(\RG{m}\CW[1],A_{s'}[1])$
are $s'$,
and there is no non-zero morphism $\undP{s}[1]\to A_{s'}[1]$ by \ref{nov6-13} of Lemma~\ref{nov7-12}.
Therefore,
the surjectivity follows from 
the exact sequence:
\begin{align*}
    \HOM_V(\undP{s}[1],A_{s'}[1])(=0)\to \HOM_V(P_s[1],A_{s'}[1])\to \HOM_V(\RG{m}\CW[1],A_{s'}[1]).
\end{align*}
This completes the proof.

\end{proof}

\begin{lemma}\label{jl8-21}
        For $\scF\in \PervVzu$
        and a morphism $f\colon A_s\to \scF[k]$ $(k=-1,0)$,
        there exists a non-unique morphism $g\colon \ove{B_s}\to \mathrm{FL}(\scF)[k]$ on $V^*$ 
        such that $\mathrm{FL}(\nu_0(g))=f$.      
        \end{lemma}

        \begin{proof}
We set $h:=\mathrm{FL}(f)\colon B_s\to \mathrm{FL}(\scF)[k]$ on $V^*$.
Note that the morphism $\ove{B_s}\to B_s$ induces an isomorpshim $\nu_{0}(\ove{B_s})\simto \nu_0(B_s)$.
Therefore, the compostion $g:=h\circ (\ove{B_s}\to B_s)$ is the desired one.           
        \end{proof}

\begin{lemma}\label{jl8-22}
    For $\scF\in \PervVzu$ and a morphism $f\colon \ove{B_s}\to \scF[-1]$ (or $f\colon \ove{P_s}\to \scF[-1]$),
    if $\nu_0(f)=0$ then $f=0$.
\end{lemma}

\begin{proof}
By applying $\RG{W}(-)$ to (\ref{jl2-7}),
we have the distinguished triangle:
\begin{align}\label{jl2-8-2}
    \CC_m[-1]\to \ove{B_s}\to B_s\to \CC_m.
\end{align}   
   Remark that there is neither non-zero morphism $\CC_m[-1]\to \scF[-1]$ nor $\CC_m\to \scF[-1]$.
Therefore, 
we obtain the isomorphism $\HOM_{\HzshV}({B_s},\scF[-1])\simeq \HOM_{\HzshV}(\ove{B_s},\scF[-1])$.
Then, the corresponding morphism $g\colon B_s\to \scF[-1]$ is zero since $\nu_0(g)=\nu_0(f)=0$.
Hence, $f$ is also zero.   
\end{proof}

\begin{lemma}\label{jan24-2}
     For $\scF\in \PervVzu$ satisfying the condition $(N_s)$,
     we have an isomorphism induced by the restriction:
     \[H^0\HOMV(\oveP{s},\scF[-1])\simto H^0\HOMV(\oveA{s},\scF[-1]).\]
\end{lemma}

\begin{proof}
For the surjectivity,
    by the distinguished triangle (\ref{jl2-10}),
    it is enough to show that the pull back of the given morphism by $\CC_0[-1]\to \ove{A_s}$ is zero.
We use Definition-Lemma~\ref{j24-2-2} and we may assume $\scF=\scF_1$ and $\scF_1=A_{s'}[1]$ ($s'<s$) or $\scF_1=P_{s'}[1]$ ($s'\leq s$).
For $A_{s'}[1]$, this can be shown by using \ref{nov7-10} and \ref{nov7-11} of Lemma~\ref{jan24-1lem}.
The case of $P_{s'}[1]$ follows from the case $\scF=A_{s'}[1]$. 

For the injectivity,
by the distinguished triangle (\ref{jl2-10}) again,
it suffices to show $H^0\HOMV(\CC_0,\scF[-1])=0$.
However, this follows from Lemma~\ref{nov22-1}.
\end{proof}

\begin{definition}
    For $\PR$ with three points $l,r,m$,
    we set
    \[\CC_{l,m,r}:=\Csan,\quad \CC_{l,r}:=\Cl\oplus \Cr.\]
\end{definition}

\begin{lemma}\label{nov8-1}
    For $\PR$ with three points $l,r,m$,
We have the following exact sequences
 \begin{align}\label{jl6-3}
      0\to H^0\HOMP(\CWT[1],\CW[1])\to H^0\HOMP(\CC_{l,m,r},\CW[1])\to H^0\HOMP(\CC_{\PR},\CW[1])\to 0, \mbox{and}
 \end{align}   
 \begin{align}\label{jl6-4}
0\to  H^0\HOMP(\CWT,\CW[1]) \to H^0\HOMP((\CC_{l,m,r})[-1],\CW[1])\to H^0\HOMP(\CC_{\PR}[-1],\CW[1])\to 0.
 \end{align}
 In particular, a morphism $\CWT\to \CW[1]$ is determined by its pull back by
$(\CC_{l,m,r})[-1]\to \CC_{W\cap T}$. 
Conversely, for a triple $(f^l,f^m,f^r)$ of morphisms
$f^l\colon \Cl[-1]\to \CW[1]$, $f^m\colon \Cl[-1]\to \CW[1]$ and 
$f^r\colon \Cl[-1]\to \CW[1]$,
if the pull back of it by $\CC_{\PR}\to \CC_{l,m,r}$ is zero,
this defines the morphism $\CWT\to \CW[1]$.
\end{lemma}

\begin{proof}
We have the following diagrams (some of them have been already used):
\[
\xymigidt{\CC_0}{\CC_0[-1]}{\CC_{\CS},}{\CC_{\CC}}{1}{2}
\xymigidt{\CC_{\PR}}{(\CC_{l,r})[-1]}{\CW,}{\CC_{\PR}}{1}{2}
\xymatrix{\phantom{\CW}&\\
\CC_m[-1]\ar[r]^{1}&{\CW},\\
&
}
\xyhidaridt{\CC_m[-1]}{\CWT}{\CW}{\CW.}{1}{0,1}
\]
    Applying $H^0\HOMP(- ,\CW[1] )$ to the distinguished triangle 
    \[\CWT\to \CC_{\PR}\to \Csan\to \CWT[1],\]
    we have a long exact sequence:
    \begin{align*}
       0\to &H^0\HOMP(\CWT[1],\CW[1])\\
       \to &H^0\HOMP(\CC_{l,m,r},\CW[1])\to H^0\HOMP(\CC_{\PR},\CW[1])\to H^0\HOMP(\CWT,\CW[1]) \\\to&H^0\HOMP((\CC_{l,m,r})[-1],\CW[1])\to 
       H^0\HOMP(\CC_{\PR}[-1],\CW[1])
\to 0.
    \end{align*}
    Because we have already known all the dimensions of these vector spaces,
    this sequence must be decomposed into two exact sequences (\ref{jl6-3}), (\ref{jl6-4}).    
\end{proof}

\begin{lemma}\label{nov12-2}
Let $(f^l,f^m,f^r)\in H^0\HOMP(\CC_{l,m,r},\CW[1])$ be a triple corresponding to a morphism $f\colon \CWT\to \CW[1]$ in the sense of Lemma~\ref{nov8-1}.
Then, the specialization $\nu_l(f)$ (resp. $\nu_r$) at $l$ (resp. $r$) is zero
if and only if $f^l=0$ (resp. $f^r=0$).    
\end{lemma}

\begin{proof}
    Note that we have the following diagram:
    \[\xyhidaridt{\CC_{l,r}[-1]}{\CW}{\CC_{\PR}}{\CW.}{0,1}{\emptyset}
    \]
Therefore, if $\nu_l(f)=0$ if and only if $\nu_l(f)\circ (\Cl[-1]\to \CW)=0$.
The latter condition is equivalent to $f^l=f\circ (\Cl[-1]\to \CWT)=0$,
since $\nu_l(\Cl[-1]\to \CWT)=\Cl[-1]\to \CW$.
\end{proof}

\begin{lemma}\label{jl8-1}
\begin{enumerate}
    \item For $f\in  H^0\HOMP(\CWT,\CW[1])$, if $\nu_l(f)=0$ and $\nu_r(f)=0$, then we have $f=0$.\label{jl6-8}
    \item \label{jl6-9}
Moreover, there exist morphisms $g_1,g_2\in H^0\HOMP(\CWT,\CW[1])$
such that $\nu_{l}(g_1)\neq 0$, $\nu_{r}(g_1)= 0$, $\nu_{l}(g_2)= 0$ and $\nu_{r}(g_2)\neq 0$.
In particular,
the $2$-dimesional $\CC$-vector space $H^0\HOMP(\CWT,\CW[1])$
    is generated by $g_1$ and $g_2$.
\end{enumerate} 
\end{lemma}

\begin{proof}
First we show \ref{jl6-8}.
It is easy to see that $H^0\HOMP(\CC_{\PR\setminus \{m\}}[-1], \CW[1])=0$.
Then, by applying $H^0\HOMP(- ,\CW[1])$ to a distinguished triangle $\CC_{\PR\setminus \{m\}}\to \CC_{\PR}\to \CC_{m}\to \CC_{\PR\setminus \{m\}}[1]$,
we have an exact sequence:
\[\dots \to H^0\HOMP(\Cm[-1],\CW[1])\to H^0\HOMP(\CC_{\PR}[-1],\CW[1])\to H^0\HOMP(\CC_{\PR\setminus \{m\}}[-1],\CW[1])(=0)\to 0.\]
Therefore, since $\dim H^0\HOMP(\Cm[-1],\CW[1])=\dim H^0\HOMP(\CC_{\PR}[-1],\CW[1])=1$,
we have an isomorphism 
    \begin{align}
        \label{l6-7}
        H^0\HOMP(\Cm[-1],\CW[1])\simeq H^0\HOMP(\CC_{\PR}[-1],\CW[1]).
    \end{align}
For a triple $(f^l,f^m,f^r)\in H^0\HOMP(\CC_{l,m,r}[-1],\CW[1])$ corresponding to a morphism $f\colon \CWT\to \CW[1]$,
the image of it by $H^0\HOMP((\CC_{l,m,r})[-1],\CW[1])\to H^0\HOMP(\CC_{\PR}[-1],\CW[1])$ is zero.
Then, if $f^l=0$ and $f^r=0$,
this means that the composition $(\Cm[-1]\to \CW[1])\circ (\CC_{\PR}[-1]\to \Cm[-1])$ is zero, and hence $f^m=0$ by (\ref{l6-7}).  

Second, we show \ref{jl6-9}.
It is easy to see that $H^0\HOMP(\CC_{\PR\setminus \{l\}}[-1],\CW[1])=0$.
Therefore, in the same way as in the proof of the first assertion,
we have an isomorphism
 \begin{align}
        \label{l6-11}
        H^0\HOMP(\Cl[-1],\CW[1])\simeq H^0\HOMP(\CC_{\PR}[-1],\CW[1]).
    \end{align}
We also have a similar isomorphism for $\Cr[-1]$.
Combining (\ref{l6-7}) and (\ref{l6-11}),
we can find a non-zero triple $(f^l,f^m,0)\in H^0\HOMP(\CC_{l,m,r}[-1],\CW[1])$
whose image by $H^0\HOMP((\CC_{l,m,r})[-1],\CW[1])\to H^0\HOMP(\CC_{\PR}[-1],\CW[1])$ is zero, which means we obtain a morphism $g_1\colon \CWT[-1]\to \CW[1]$ such that $\nu_l(g_1)\neq 0$ and $\nu_r(g_1)=0$.
We can also find $g_2$ in the same way,
and clearly $g_1$ and $g_2$ are linearly independent.
This proves \ref{jl6-9}.
\end{proof}

\begin{definition}
    For $F,G\in \shV$, we define $H^0\HOMV(F,G)_0$
    as the subset of $H^0\HOMV(F,G)$ consisting those satisfying $\nu_0=0$.
\end{definition}

\begin{lemma}\label{nov8-3}
For $s'\leq  s$,
we have the following exact sequence:
    \begin{align}\label{jl6-1}
        0\to H^0\HOMV(\oveA{s},A_{s'-1}[1])_0\to H^0\HOMV(\oveA{s},A_{s'}[1])_0\to H^0\HOMV(\oveA{s},\CW[1])_0\to 0.
    \end{align}
   In particular, every morphism $\oveA{s}\to A_{s'}[1]$ with $\nu_0=0$ is determined inductively by 
    $s'$-tuple of $\CWT\to \CW[1]$ with $\nu_0=0$, if we fix a basis of $H^0\HOMV(\oveA{s},A_{s'}[1])_0$. 
\end{lemma}

\begin{proof}
We remark that we already have an exact sequence by \ref{nov7-10} of Lemma~\ref{jan24-1lem}:
\begin{align}\label{jan29-1}
        0\to H^0\HOMV(\oveA{s},A_{s'-1}[1])\to H^0\HOMV(\oveA{s},A_{s'}[1])\to H^0\HOMV(\oveA{s},\CW[1])\to 0.
\end{align}
The injectivity of $H^0\HOMV(\oveA{s},A_{s'-1}[1])_0\to H^0\HOMV(\oveA{s},A_{s'}[1])_0$ directly follows from this.

Let us show the exactness at $H^0\HOMV(\oveA{s},A_{s'}[1])_0$ of the sequence (\ref{jl6-1}).
By (\ref{jan29-1}), it is enough to show that $\nu_0$ of a morphism which $f\colon \oveA{s}\to A_{s'-1}[1]$ satisfies $\nu_0((A_{s'-1}[1]\to A_{s'}[1])\circ f)=0$ is zero.
By the distinguished triangle $A_{s'-1}\to A_{s'}\to \CW\to A_{s'-1}[1]$,
there is a morphism $f_1\colon A_s(=\nu_0(A_s))\to \CW$ such that
$(\CW\to A_{s'-1}[1])\circ f_1=\nu_0(f)$.
Take the pull back $f_2:=f_1\circ (\oveA{s}\to A_s)$.
Then, since $\nu_0(f_2)=f_1$, we have
$\nu_0(f-(\CW\to A_{s'-1}[1])\circ f_2)=0$.
On the other hand, by \ref{nov7-10} of Lemma~\ref{jan24-1lem},
$(\CW\to A_{s'-1}[1])\circ f_2$ is zero,
and hence $\nu_0(f)=0$.

Next, we show the exactness at $H^0\HOMV(\oveA{s},\CW[1])_0$ of the sequence (\ref{jl6-1}).
By (\ref{jan29-1}), it is enough to show the following: If a morphism $g\colon \oveA{s}\to A_{s'}[1]$ satisfies 
$\nu_0((A_{s'}[1]\to \CW[1])\circ g)=0$, then there exists $h\colon \oveA{s}\to A_{s'}[1]$ such that
$(A_{s'}[1]\to \CW[1])\circ h=(A_{s'}[1]\to \CW[1])\circ g$ and
$\nu_0(h)=0$.
There is a morphism $g_1\colon A_s\to A_{s'-1}[1]$ such that
$(A_{s'-1}[1]\to A_{s'}[1])\circ g_1=\nu_0(g)$.
Take the pull back $g_2:=g_1\circ (\oveA{s}\to A_{s}) $
and set $h:=g-(A_{s'-1}[1]\to A_{s'}[1])\circ g_2$.
Then, we have $\nu_0(h)=0$,
and $(A_{s'}[1]\to \CW[1])\circ h=(A_{s'}[1]\to \CW[1])\circ g$.

The final assertion follows from \ref{nov7-10} of Lemma~\ref{jan24-1lem}, i.e. to give a morphism $\oveA{s}\to \CW[1]$ is equivalent to giving $\CWT\to \CW[1]$, and a morphism $\oveA{s}\to A_{s'}$ is determined by define $s'$-tuple of $\oveA{s}\to \CW[1]$ if we fix a basis of $H^0\HOMV(\oveA{s},A_{s'}[1])$. 
\end{proof}

\begin{corollary}\label{jl8-7}
For $\PR$ with three points $l,r,m$ and $s'\leq  s$, there exists a basis $f_1,\dots,f_{s'}$, $g_1,\dots,g_{s'}$ of $H^0\HOMP(\oveA{s},A_{s'}[1])$ with the following properties:
\begin{align}
    \label{jl8-5}
    \nu_l(f_i)=0\ (1\leq i\leq s')\quad  \mbox{and}\quad  \nu_{r}(g_i)=0\ (1\leq i\leq s').
\end{align}
\end{corollary}

\begin{proof}
We use induction with respect to $s'\in \ZZ_{\geq 1}$.
If $s'=1$,
this follows from the first diagram of \ref{nov7-10} of Lemma~\ref{jan24-1lem} and Lemma~\ref{jl8-1}.
For general $s'$, assume that we have a basis $f_1,\dots,f_{s'-1}, g_1,\dots,g_{s'-1}\colon \oveA{s-1}\to A_{s'-1}[1]$ with the properties (\ref{jl8-5}).
By the second diagram of \ref{nov7-10} of Lemma~\ref{jan24-1lem},
the compositions of them (we will use the same symbols $f_1,\dots,f_{s'-1}, g_1,\dots,g_{s'-1}$) with $(A_{s'-1}[1]\to A_{s'}[1])$ are still linearly independent in $H^0\HOMP(\oveA{s},A_{s'}[1])$.
Moreover, take a basis $f_{s'},g_{s'}$ of $H^0\HOMP(\oveA{s},\CW[1])$ with $\nu_l(f_{s'})=0$ and $\nu_r(g_{s'})=0$.
By Lemma~\ref{nov8-3},
we take $\wt{f_{s'}}, \wt{g_{s'}}\colon \oveA{s}\to A_{s'}[1]$ so that $(A_{s'}[1]\to \CW[1])\circ \wt{f_{s'}}=f_{s'}$ and
$(A_{s'}[1]\to \CW[1])\circ \wt{g_{s'}}=g_{s'}$
with
$\nu_l(\wt{f_{s'}})=0$ and $\nu_r(\wt{g_{s'}})=0$.
Then,
$f_1,\dots,f_{s'-1},\wt{f_{s'}}$,
$g_1,\dots,g_{s'-1},\wt{g_{s'}}$
form a desired basis of $H^0\HOMP(\ove{A_s},A_{s'}[1])$.
\end{proof}

\begin{lemma}\label{jl3-11}
For $\scF\in \PervVzu$ satisfying the condition $(N_s)$ ($s\in \ZZ_{\geq 2}$)
    and a morphism $f\colon \ove{A_s}\to \scF$ such that $\nu_0(f)=0$.
Then, there exists a unique morphism $g\colon \ove{P_s}\to \scF$ such that
$\nu_0(g)=0$ and $g|_{W}=f|_{W}$.
\end{lemma}

\begin{proof}
We use Definition-Lemma~\ref{j24-2-2}.
Since $H^0\HOMV(\ove{P_s},\scF_2[k])_0=0$ and $H^0\HOMV(\ove{A_s},\scF_2[k])_0=0$ for any $k$,
we may assume $\scF=\scF_1$.
By Definition-Lemma~\ref{j24-2-2},
$\scF_1$ is a direct sum of some $A_{s'}[1]$ ($s'<s$) or $P_{s'}[1]$ ($s'\leq s$).
Hence, the assertion follows from the following Lemma~\ref{jl4-1} and Lemma~\ref{nov8-5}.
\end{proof}
Note that to give $f$ is equivalent to give $f|_{W}$.

\begin{lemma}
\label{jl3-14}
For $\scF\in \PervVzu$ and $f\colon \nu_{0}(\ove{P_s})(=P_s)\to \scF[k]$ (resp. $\nu_0(\oveA{s})=A_s\to \scF[k]$) (on $T_0V$) for any $k\in \ZZ$,
        there exists $g\colon \ove{P_s}\to \scF[k]$ (resp. $\oveA{s}\to \scF[k]$) such that $\nu_0(g)=f$.
        \end{lemma}
        \begin{proof}
 For $f\colon \nu_{0}(\ove{P_s})(=P_s)\to \scF[k]$,
 the morphism $g:=f\circ (\ove{P_s}\to P_s)$ satisfies $\nu_0(g)=f$.            
 The proof for $\oveA{s}$ is the same.
        \end{proof}
        
        \begin{corollary}
          \label{jl3-15}For $2\leq s'<s$ and $h\colon \ove{P_s}\to \CW[1]$ satisfying $\nu_0(h)=0$,
    there exists $\wt{h}\colon \ove{P_s}\to A_{s'}[1]$ such that 
    $(A_{s'}[1]\to \CW[1])\circ \wt{h}=h$ and $\nu_0(\wt{h})=0$.  
        \end{corollary}
\begin{proof}
    For $h\colon \ove{P_s}\to \CW[1]$ such that $\nu_0(h)=0$,
    we can take its lift, i.e. $h_1\colon \ove{P_s}\to A_{s'}[1]$ such that $(A_{s'}[1]\to \CW[1])\circ h_1=h$, since any morphism $\ove{P_s}\to A_{s'-1}[2]$ is zero.
    Since the morphism $(A_{s'}[1]\to \CW[1])\circ \nu_0(h_1)=\nu_0(h)$ is zero,
    $\nu_0(h_1)$ factors through some morphism $h_2\colon \nu_0(\ove{P_s})\to A_{s'-1}[1]$.
    Applying Lemma~\ref{jl3-14} to $h_2$,
    there exists $h_3\colon \ove{P_s}\to A_{s'-1}[1]$ such that $\nu_0(h_3)=h_2$.
    Then, the desired morphism is 
    $h_1-(A_{s'-1}[1]\to A_{s'}[1])\circ h_3$.    
\end{proof}

\begin{lemma}\label{jl4-1} 
In the situation of Lemma~\ref{jl3-11}, and with $\scF=A_{s'}[1]$ for $s'<s$,
the assertion of Lemma~\ref{jl3-11} is true.
The same claim is valid also for $\scF=\CV[1]$.
\end{lemma}

\begin{proof}
We consider the morphism
\begin{align}\label{nov8-4}
    H^0\HOMV(\oveP{s},A_{s'}[1])_0\to H^0\HOMV(\oveA{s},A_{s'}[1])_0.
\end{align}
The dimension of the right hand side is $s'$ by Lemma~\ref{nov8-3},
and it is easy to see that the dimension of the left hand side is $s'$.
Therefore, it is enough to show (\ref{nov8-4})
is surjective.

Assume $s'=1$.
If a morphism $f\in H^0\HOM_V(\oveA{s},\CW[1])$ satisfies $\nu_0(f)=0$, i.e. $f\circ (A_s\to \oveA{s})=0$,
the pullback $f\circ (\CC_0[-1]\to \oveA{s})$ is zero.
Therefore, by using an exact sequence:
\begin{align*}
H^0\HOM_V(\CC_0,\CW[1])\to H^0\HOM_V(\oveP{s},\CW[1])&\to H^0\HOM_V(\oveA{s},\CW[1])\to
H^0\HOM_V(\CC_0[-1],\CW[1]),
\end{align*}
induced by the distinguished triangles (\ref{jl2-10}),
there is a morphism
$g\in H^0\HOM_V(\oveP{s},\CW[1])$
such that $g\circ (\oveA{s}\to \oveP{s})=f$.
On the other hand, by the exact sequence induced by
the distinguished triangle (\ref{j24-5}),
we have an isomorphisms
\begin{align}
    H^0\HOM_V(\CC_0,\CW[1])&\simto H^0\HOM_V(P_s,\CW[1]),\label{jan29-4}
\end{align}
Therefore, we can take $h\colon \CC_0\to \CW[1]$ such that
$h\circ (P_s\to \CC_0)=\nu_0(g)$.
Then, the morphism $g':=g-h\circ (\oveP{s}\to \CC_0)\in H^0\HOMV(\oveP{s},\CW[1])$
satisfies $\nu_0(g')=0$ and $g'\circ (\oveA{s}\to \oveP{s})=f$.
This proves the surjectivity of (\ref{nov8-4}) for $s'=1$.

Assume $s'\geq 2$ and take $f\in H^0\HOMV(\oveA{s},A_{s'}[1])_0$.
Since $\nu_0((A_{s'}[1]\to \CW[1])\circ f)$ is also zero,
there exists a unique morphism $g\colon \ove{P_s}\to \CW[1]$ such that
$g\circ (\ove{A_s}\to \ove{P_s})=(A_{s'}[1]\to \CW[1])\circ f$ and $\nu_0(g)=0$ by the isomorphism (\ref{nov8-4}) for $s'=1$.
By Lemma~\ref{jl3-15},
there is a morphism $\wt{g}\colon \ove{P_s}\to A_{s'}[1]$
such that 
    $(A_{s'}[1]\to \CW[1])\circ \wt{g}=g$ and $\nu_0(\wt{g})=0$. 
  We set $h:= (f-\wt{g}\circ (\oveA{s}\to \oveP{s}))$.
  Since the composition $(A_{s'}[1]\to \CW[1])\circ h$ is zero by the definition,
$h$ factors through some $i\colon \oveA{s}\to A_{s'-1}[1]$.  
 By the induction assumption, we can take $\wt{i}\colon \ove{P_{s}}\to A_{s'-1}[1]$ such that 
 $\wt{i}\circ (\oveA{s}\to \ove{P_{s}})=i$ and
 $\nu_{0}(\wt{i})=0$.
Therefore, 
we have
\[(h=)f-\wt{g}\circ (\oveA{s}\to \oveP{s})=(A_{s'-1}[1]\to A_{s'}[1])\circ \wt{i}\circ (\oveA{s}\to \oveP{s}).\]
Hence, we get
\[f=(\wt{g}+(A_{s'-1}[1]\to A_{s'}[1])\circ \wt{i})\circ (\oveA{s}\to \oveP{s}).\]
 Since $\nu_0(\wt{g}+(A_{s'-1}[1]\to A_{s'}[1])\circ \wt{i})$ is zero, 
this proves the surjectivity of (\ref{nov8-4}).

The claim for $\scF=\CV[1]$ can be shown in the same (simpler) way.
\end{proof}

\begin{lemma}\label{nov8-5}
In the situation of Lemma~\ref{jl3-11}, and with $\scF=P_{s'}[1]$ for $s'\leq s$,
the assertion of Lemma~\ref{jl3-11} is true.
\end{lemma}

\begin{proof}
It is easy to see that, for any morphism $\oveP{s}\to A_{s'-1}$,
the composition $(A_{s'-1}\to \CV[1])\circ (\oveP{s}\to A_{s'-1})$ is zero.
Therefore, we have exact sequences:
\begin{align}
 0\to H^0\HOMV(\oveP{s},\CV[1])\to H^0\HOMV(\oveP{s},P_{s'})\to H^0\HOMV(\oveP{s},A_{s'}[1])\to 0,\\
 0\to H^0\HOMV(\oveA{s},\CV[1])\to H^0\HOMV(\oveA{s},P_{s'})\to H^0\HOMV(\oveA{s},A_{s'}[1])\to 0.
\end{align}
Moreover, in the same way as in the proof of \ref{nov8-3},
we also have
\begin{align}
    0\to H^0\HOMV(\oveP{s},\CV[1])_0\to H^0\HOMV(\oveP{s},P_{s'})_0\to H^0\HOMV(\oveP{s},A_{s'}[1])_0\to 0,\\
     0\to H^0\HOMV(\oveA{s},\CV[1])_0\to H^0\HOMV(\oveA{s},P_{s'})_0\to H^0\HOMV(\oveA{s},A_{s'}[1])_0\to 0.
\end{align}
Then, combining them with Lemma~\ref{jl4-1},
we have the isomorphism 
\[H^0\HOMV(\oveP{s},P_{s'})_0\simto H^0\HOMV(\oveA{s},P_{s'})_0.\]
This completes the proof.
\end{proof}

We need a version of Lemma~\ref{jl3-11} for $\ove{B_s}$ instead of $\ove{P_s}$.
The proof is the same as for Lemma~\ref{jl3-11}, so we omit the details.

\begin{lemma}\label{jl5-1}
For $\scF\in \PervVzu$ with the nilpotent order $\leq s$ and a morphism $f\colon \ove{A_s}\to \scF$ such that $\nu_0(f)=0$.
Then, there exists a unique morphism $g\colon \ove{B_s}\to \scF$ such that
$\nu_0(g)=0$ and $g|_{W}=f|_{W}$.
\end{lemma}
This completes the proof of Lemma~\ref{jl3-11}.

Recall the notation: For $\PR$ with three points $l,r,m$, we set $V_r=\PR\setminus \{l\}$,
$V_l=\PR\setminus \{r\}$.

\begin{lemma}\label{jl18-23} 
Let $\scF$ (resp. $\scG$) be an object in $\mathrm{Perv}(V_r,0)$ (resp. $\mathrm{Perv}(V_l,0)$)
and $f\colon \ove{\calL_s}\to \scF|_{W}$ a morphism on $W$.
Assume that $\scF|_{W}\simeq \scG|_{W}$ and $\scF$ satisfies the condition $(N_s)$.
Then, there exists a unique pair of morphisms: $g\colon \ove{B_s}\to \scG$
and $h\colon \ove{P_s}\to \scF$ such that 
$\nu_{l}(g)=0$, $\nu_{r}(h)=0$ and $f=g|_{W}+h|_{W}$.
\end{lemma}

\begin{proof}
    By Lemma~\ref{jl8-7},
    $f$ is uniquely decomposed into $f=f_1+f_2$ so that $\nu_l(f_1)=0$ and $\nu_r(f_2)=0$.
    By Lemma~\ref{jl3-11} (resp. Lemma~\ref{jl5-1}), there exists a unique morphism $h\colon \oveP{s}\to \scF$ (resp. $g\colon \ove{B_s}\to \scG)$ such that $\nu_r(h)=0$ and $h|_{W}=f_2$ (resp. $\nu_l(g)=0$ and $g|_{W}=f_1$). 
    Then, this proves the assertion.
\end{proof}

\subsubsection{Lemmas for $n\geq 2$ and $j=1$}

\begin{lemma}\label{nov12-3-1}
    If the nilpotent order of $\scF\in \PervVzu$ is $\leq 1$ and
    the one of $\mathrm{FL}(\scF)$ is also $\leq 1$,    
we have 
\begin{align}
    &H^0\HOMV(\ove{B_1},\scF[-1])_{0}= 0,\quad \mbox{and}\\
    &H^0\HOMV(\ove{B_1},\scF)_{0}\simto H^0\HOMV(\CWT,\scF)_0.
\end{align}
\end{lemma}

\begin{proof}
We only prove the second assertion.
The first one can be shown in the same way.
By the condition,
$\scF$ is a direct sum of some $\CV[1]$, $\CW[1]$, $\RG{W}\CV[1]$ and $\CC_0$.
The assertion is clear for $\RG{W}\CV[1]$ and $\CC_0$.
For $\CV[1]$ or $\CW[1]$, the dimensions of both sides are equal to $1$.
Therefore, it is enough to show the surjectivity.
This follows from Lemma~\ref{nov8-1} with the distinguished triangle:
\[(\ove{B_1})_0[-1](\simeq \CC_0[-2]\oplus \CC_0[-1])\to \CWT\to \ove{B_1}\to (\ove{B_1})_0.\]
\end{proof}

\begin{lemma}\label{jl8-42} 
    We have an isomorphisms
    \begin{align}
        H^0\HOMV(\CVT,\CV)\simto& H^0\HOMV(\CWT,\CW),\quad \mbox{and}
\\\label{jl8-41}
        H^0\HOMV(\CVT,\CV[1])\simto& H^0\HOMV(\CWT,\CW[1])_0.
    \end{align}
\end{lemma}

\begin{proof}
The first one is clear.

For the second one,
we consider the diagram:
\[\xyhidaridt{\CC_m[-1]}{\CVT}{\CV}{\CV}{1}{0}.\]
Then, for $f\colon \CVT\to \CV[1]$,
since the pull back $f\circ (\CWT\to \CVT)\circ (\CC_0[-1]\to \CWT)$ is zero,
we have $\nu_0(f\circ (\CWT\to \CVT))=0$ by Lemma~\ref{nov12-2}.
Because the dimensions of both sides of (\ref{jl8-41}) are $1$ and (\ref{jl8-41}) is non-zero morphism,
this proves the isomorphism.
\end{proof}

\subsubsection{Lemmas for $n=2n_0+1$ ($n_0\geq 1$) and $j=n_0+1$}

For $\oveundA{s}$ and $\oveundP{s}$, one can show the following.

\begin{lemma}\label{jl9-1}
    For $s\in \ZZ_{\geq 2}$, we have diagrams:
    \[\xyhidaridt{\ove{A_{s-1}}}{\oveundA{s}}{\RG{T}\CW}{\CW,}{1,1}{\emptyset}\quad \xyhidaridt{\CWT}{\oveundA{s}}{\und{A_{s-1}}}{\CW,}{1,1}{\emptyset}\quad 
    \xyhidaridt{\oveP{s-1}}{\oveundP{s}}{\RG{T}\CW}{\CW,}{1,1}{\emptyset} \quad     
    \xyhidaridt{\CVT}{\oveundP{s}}{\undA{s-1}}{\CW.}{1}{1}
\]
\end{lemma}

\begin{corollary}\label{jl9-3} 
\begin{enumerate}
\item For $\PR$ with a three points $l,r,m$ and $s\in \ZZ_{\geq 2}$,
there exists a basis $f,g$ of $H^0\HOMP(\oveundA{s},\CW[1])$ such that
$\nu_l(f)=0$ and $\nu_r(g)=0$.
\item For $s\in \ZZ_{\geq 2}$,
there exists a basis $f,g$ of $H^0\HOMV(\oveundP{s},\CW[1])$ such that
$\nu_0(f)=0$, $f|_{W}\neq 0$, $\nu_0(g)\neq 0$, $g|_{W}=0$. 
\end{enumerate}
\end{corollary}

\begin{proof}
By Lemma~\ref{jl9-1},
to give $\oveundA{s}\to \CW[1]$ is equivalent to giving $\CWT\to \CW$.
Then, the first assertion follows from Lemma~\ref{jl8-1}.

For the second assertion,
we have the diagram:
\[\xyhidaridt{\CV}{\nu_0(\oveundP{s})}{A_{s-1}}{\CW.}{\emptyset}{0,1}\]
Then, the specialization at $0$ of the morphism $g:=(\undA{s-1}\to \CW[1])\circ (\oveundP{s}\to \undA{s-1})$, which is non-zero by the fourth diagram in Lemma~\ref{jl9-1}, is not zero.
Moreover, $g|_{W}=0$ by the second diagram in Lemma~\ref{jl9-1}.
Then, we take $h$ which is a lift of $\CVT\to \CW[1]$, i.e. $(\oveundP{s}\to \CW[1])\circ (\CVT\to \oveundP{s})=(\CVT\to \CW[1])$.
We can take the constant $c\in \CC$ so that $\nu_0(h-cg)=0$,
and hence $f:=h-cg$ has the desired properties.
\end{proof}

\begin{lemma}\label{jl9-11}
    For $\scF\in \PervVzu$,
    any morphism $\oveundP{s}\to \scF[-1]$ is zero.
\end{lemma}
\begin{proof}
According to Lemma~\ref{jl9-1},
the assertion is true if $\scF$ is $\CC_0$, $A_s[1]$, $B_s[1]$,
$P_s[1]$ or $Q_s[1]$.
Hence, we get the conclusion by \ref{nov5-2} of Remark~\ref{nov5-3} .
\end{proof}

\begin{lemma}\label{jl9-2}
 For $s,s'\in \ZZ_{\geq 1}$
and a morphism $f\colon A_s\to A_{s'}[1]$ (resp. $g\colon P_s\to A_{s'}[1]$),
there exists a morphism $\wt{f}\colon \oveundA{s}\to A_{s'}[1]$ (resp. $\wt{g}\colon \oveundP{s}\to A_{s'}[1]$) such that $\nu_0(\wt{f})=f$ (resp. $\nu_0(\wt{g})=g$).
\end{lemma}
\begin{proof}
    Consider $f':=f\circ (\oveA{s}\to A_s)$
    and the distinguished triangle $\RG{m}\CW\to \oveA{s}\to \oveundA{s}\to \RG{m}\CW[1]$.
Then, since there is no non-zero morphism $\RG{m}\CW\to A_{s'}[1]$,
$f'$ can be lifted to $\oveundA{s}\to A_{s'}[1]$, which has the desired property.
The proof for $g$ is similar.
\end{proof}

Corollary~\ref{jl9-5} and Corollary~\ref{lj9-4} below can be shown in the same way as Corollary~\ref{jl8-7} and Corollary~\ref{jl3-15} by Lemma~\ref{jl9-2}.

\begin{corollary}\label{jl9-5}
        For $\PR$ with three points $l,r,m$,
there exists a basis $f_1,\dots,f_{s'},g_1,\dots,g_{s'}$ of $H^0\HOMP(\oveundA{s},A_{s'}[1])$ such that
$\nu_l(f_i)=0$ and $\nu_r(g_i)=0$ ($1\leq i\leq s'$).
\end{corollary}

\begin{corollary}\label{lj9-4}
     For $s'<s$ and $f\colon \oveundP{s}\to \CW[1]$ with $\nu_0(f)=0$,
    there exists a morphism $g\colon \oveundP{s}\to A_{s'}[1]$ such that
    $\nu_0(g)=0$ and $(A_{s'}[1]\to \CW[1])\circ g=f$.
    \end{corollary}

\begin{corollary}
         For $s'<s$,
the restriction map is a isomorphism:
\[H^0\HOMV(\oveundP{s},A_{s'}[1])_0\simto H^0\HOMV(\oveundA{s},A_{s'}[1])_0.\]
\end{corollary}

\begin{lemma}\label{jl9-6}
    For $\scF\in \PervVzu$ with the condition $(N_s)$.
    Then,
    the restriction map induces the isomorphism
    \[H^0\HOMV(\oveundP{s},\scF)_0\simto H^0\HOMV(\oveundA{s},\scF)_0.\]
\end{lemma}

In the following corollary, we denote by $\oveundP{s}^{l}$ (resp. $\oveundP{s}^{r}$) the object $\oveundP{s}$ on $V_l$ (resp. $V_r$).  

\begin{corollary}\label{jl9-10}
For $\scF\in \mathrm{Perv}(V_l,0)_{\mathrm{unip}}$ and $\scG\in \mathrm{Perv}(V_r,0)_{\mathrm{unip}}$ with the property $(N_s)$ and an isomorphism $\scF|_{W}\simto \scG|_{W}$,
and a morphism $f\colon \oveundP{s}|_{W}(=\oveundA{s})\to \scF|_{W}$,
there exists a unique pair of morphisms 
$g\colon \oveundP{s}^{l}\to \scF$ and $h\colon \oveundP{s}^{r}\to \scG$
such that
$f=g|_{W}+h|_{W}$, $\nu_l(g)=0$ and $\nu_r(h)=0$.
\end{corollary}

\subsubsection{Lemmas for $n=1$}

\begin{lemma}\label{dec20-10}
\begin{enumerate}
    \item \label{jl9-15}We have a diagram:
    \[
\xyhidaridt{\CC_T}{\wtundPo}{\CC_m[-1]}{\CV,}{1}{\emptyset}\quad 
\xyhidaridt{\CWT}{\wtundPo|_{W}}{\CC_m[-1]}{\CW.}{1,1}{\emptyset}\]

\item \label{jl9-16}The restriction map induces an isomorphism: 
\begin{align}\label{jl9-17}
    \HOM_{V}(\wtundPo,\CV)\simto \HOM_{V}((\wtundPo)_W,\CW)_0.
\end{align}

\item \label{jl9-18}There is a basis $f,g$ of $H^0\HOMV((\wtundPo)_W,\CW)$ such that
$\nu_l(f)=0$ and $\nu_r(g)=0$.

\end{enumerate}

\end{lemma}

\begin{proof}
\ref{jl9-15} is easy.
Note that 
for $f\in \HOM_{\Hzsh{V}}(\wtundPo,\CV)$ we have $\nu_0(f)=0$.
Moreover, for $\CC_T\to \CV[1]$ there is a morphism $\CWT\to \CV[1]$ with a commutative diagram
\[\xymigisita{\CWT}{\CC_T}{\CV}.\]
Therefore,
(\ref{jl9-17}) is injective.
Since the dimensions of the both sides of (\ref{jl9-17}) are $1$,
it is an isomorphism.
\ref{jl9-18} is due to the second diagram of \ref{jl9-15} and Lemma~\ref{jl8-1}. 
\end{proof}

\subsection{Construction of microlocal skyscraper sheaves} \label{dec20-4}  

Next step is to construct objects corresponding to cotangent fiber under the Ganatra--Pardon--Shende equivalence.

\begin{definition-lemma}\label{jl11-1} 
        \begin{enumerate}
        \item \label{jl10-1}We define the morphism $\frakf_1\colon \undA{1}(=\RG{T}\CW)\to \oveA{1}(=\CWT)$ and $\frakf_1\colon \undP{1}(=\RG{T}\CV)\to \oveP{1}(=\CVT)$ so that we have (non-zero) commutative diagrams:
        \[\xysitamigiue{\RG{T}\CW}{\RG{m}\CW[1]}{\CWT},
        \xysitamigiue{\RG{T}\CV}{\RG{m}\CW[1]}{\CVT},
        \]
        where $\RG{m}\CW[1]\to \CWT$ is the composition $(\Cm[-1]\to \CWT)\circ (\RG{m}\CW[1]\simeq \Cm[-1])$.
        
        \item \label{jl10-2} We define the morphisms $\frakf_s\colon \und{A_s}\to \ove{A_{s}}$ and $\frakf_s\colon \und{P_s}\to \oveP{s}$ inductively as the compositions:
        \begin{align*}
            (\oveA{s-1}\to \oveA{s})&\circ (\undA{s-1}\xrightarrow{\mathfrak{f_{s-1}}} \oveA{s-1})\circ (\undA{s}\to \undA{s-1}),\quad \mbox{and}\\
            (\oveP{s-1}\to \oveP{s})&\circ (\undP{s-1}\xrightarrow{\mathfrak{f_{s-1}}} \oveP{s-1})\circ (\undP{s}\to \undP{s-1}),
        \end{align*}
respectively.
 Then, each $\mathfrak{f}_s$ is not zero.

 \item \label{jl10-3} The (zero-extension of the) restriction of $\frakf_s\colon \undP{s}\to \oveP{s}$ to $W$ is $\frakf_s\colon \undA{s}\to \oveA{s}$.

 \item We define $\mathfrak{f}_s\colon \und{A_s}\to \ove{B_s}$ as the image under the adjunction isomorphism $\Hom(\und{A_s}, \ove{B_s})\cong \Hom(\und{A_s}, \ove{A_s})$.

\item We also define the (non-zero) morphisms
$\mathfrak{f}_s\colon \und{\ove{P_s}}\to \und{\ove{P_s}}$ inductively as
\[(\oveP{s-1}\to \und{\ove{P_s}})\circ ({\und{P_{s-1}}}\xrightarrow{\frakf_{s-1}} \ove{{P_{s-1}}})\circ (\und{\ove{P_s}} \to \und{{P_{s-1}}}).\]

\item We define
$\mathfrak{f}_1\colon \wt{\undP{1}}\to \wt{\undP{1}}$
as the composition
\[(\CC_T\to \wt{\undP{1}})\circ (\CC_m[-1]\to \CC_T)\circ (\wt{\undP{1}}\to \CC_m[-1]).\]
 
 \item \label{jl10-4} We have $\nu_0(\mathfrak{f}_s)=0$.

    \end{enumerate}
    
\end{definition-lemma}

\begin{proof}
\ref{jl10-1} follows from the following diagram:
\begin{align}
    \label{jan29-99}
    \xyhidaridt{\CW}{\RG{T}\CW}{\RG{m}\CW[1]}{\CWT.}{1}{0,1}
\end{align}

Let us show \ref{jl10-2}.
By (\ref{jl2-2}) and a diagram (which can be shown by using (\ref{jan29-99})):
\[\xyhidaridt{\CW}{\undA{s}}{\undA{s-1}}{\CWT,}{1}{0,1\times s}\]
   we get diagrams:
 \[\xymigidt{\undA{s}}{\oveA{s-1}}{\oveA{s},}{\CW}{0,1\times(2s-1)}{1}\quad 
 \xyhidaridt{\CW}{\undA{s}}{\undA{s-1}}{\oveA{s-1}.}{1}{0,1\times 2(s-1)}
 \]
Therefore,
the composition
    $(\oveA{s-1}\to \oveA{s})\circ (\undA{s-1}\xrightarrow{\mathfrak{f}_{s-1}} \oveA{s-1})\circ (\undA{s}\to \undA{s-1})$ is not zero.
The same argument works for $\und{P_s}\to \ove{P_s}$.

The others can be shown in the same way.
\end{proof}

The morphism $\mathfrak{f}_s$ plays a role when we glue $\block{j}$ and $\block{j}^{\circ}[1]$ together.

\begin{lemma}\label{jl11-2}
For $1\leq j\leq n$,
we can define a canonical morphism $\block{j}^{\circ}\to \block{j}$ in $\mu \mathrm{sh}_{C_{\lc \boldsymbol{m}\rc}}(X_\Gamma)$,
which will be denoted by $\mathfrak{f}_j$ (using the same symbol for morphisms in Definition-Lemma~\ref{jl11-1}).
Similarly, we can also define $\mathfrak{f}_j\colon \block{j}^{\circ}\to \block{j}'$.
\end{lemma}

\begin{proof}
We only show it in the case $n\geq 4$ and $2\leq j< n/2$.
The same argument works in the other cases.
Remark that $\block{j}^{\circ}=\block{n-j+1}$.
We consider a morphism $\mathfrak{f}_j\colon \und{A_j}\to \ove{B_j}$ (resp.
$\mathfrak{f}_j\colon \undP{j}\to \ove{P_j}$) on $V_l$ (resp. $V_r$) on
$(n-j+1)$-th $\PR$.
Then, by \ref{jl10-4} of Definition-Lemma~\ref{jl11-1},
we can consider $(0,\mathfrak{f}_j)\in H^0\HOM_{{U_{n-j+1}}}(\block{j}^{\circ},\block{j})$ and
$(\mathfrak{f}_j,0)\in H^0\HOM_{{U_{n-j+2}}}(\block{j}^{\circ},\block{j})$,
and we thus define an element
$(b_1,\dots,b_{n+1})\in \bigoplus_{1\leq i\leq n+1}H^0\HOM_{{U_i}}(\block{j}^{\circ},\block{j})$
as
\[b_i=\left\{\begin{array}{ll}
     (0,\mathfrak{f}_j)& i=n-j+1 \\
     (\mathfrak{f}_j,0)&i=n-j+2\\
     (0,0)&\mbox{otherwise}. 
\end{array}\right.\]
Since the image of it by (\ref{oct31-5}) is zero,
we get a morphism $\block{j}^\circ\to \block{j}$ in the category $\mu \mathrm{sh}_{C_{\lc \boldsymbol{m}\rc}}(X_\Gamma)$.
\end{proof}

Remark that we can write $\mathfrak{f}_j$ as $\block{n-j+1}\to \block{n-j+1}^\circ$ since $\block{n-j+1}=\block{j}^\circ$.

\begin{definition-lemma}\label{jl12-3}
For $1\leq j\leq n$,
we define ${\scH_j^k}\ (\mbox{resp.\ } \und{\scH_j^k}) \in \mu \mathrm{sh}_{C_{\lc \boldsymbol{m}\rc}}(X_\Gamma)$ ($k\in \ZZ_{\geq 0}$) inductively as follows.
\begin{enumerate}
    \item We set $\scH_j^0:= \block{j}'$ (resp. $\und{\scH_j^0}:=\block{j}$).
    \item For odd $k$, we define $\scH_j^{k}$ so that it fits into the distinguished triangle in $\mu \mathrm{sh}_{C_{\lc \boldsymbol{m}\rc}}(X_\Gamma)$:
    \[\scH_j^{k-1}\xrightarrow{\iota} \scH_j^{k}\to \block{j}^{\circ}[k]\to \scH_j^{k-1}[1],\]
    where $\block{j}^{\circ}[k]\to \scH_k^{j-1}[1]$ is defined inductively so that the diagram blow commutes (for $k\geq 2$):
    \[\xymatrix{\block{j}^{\circ}[k]\ar[dr]_{\mathfrak{f}_j}\ar[r]&\scH_j^{k-1}[1]\ar[d]\\&\block{j}[k]},\]
    where $\mathfrak{f}_j$ is the one in Lemma~\ref{jl11-2}.
Similarly, we define $\und{\scH_j^k}$ so that we have
\[\und{\scH_j^{k-1}}\xrightarrow{\iota} \und{\scH_j^{k}}\to \block{j}^{\circ}[k]\to \und{\scH_j^{k-1}}[1].\]

\item 
For even $k$, we define $\scH_j^{k}$ so that it fits into the distinguished triangle in $\mu \mathrm{sh}_{C_{\lc \boldsymbol{m}\rc}}(X_\Gamma)$:
    \[\scH_j^{k-1}\xrightarrow{\iota} \scH_j^{k}\to \block{j}[k]\to \scH_j^{k-1}[1],\]
    where $\block{j}[k]\to \scH_j^{k-1}[1]$ is defined inductively so that the diagram blow commutes:
    \[\xymatrix{\block{j}[k]\ar[dr]_{\mathfrak{f}_{n-j+1}}\ar[r]&\scH_j^{k-1}[1]\ar[d]\\&\block{j}^{\circ}[k]}.\]
Similarly, we define $\und{\scH_j^k}$ so that we have
\[\und{\scH_j^{k-1}}\xrightarrow{\iota} \und{\scH_j^{k}}\to \block{j}[k]\to \und{\scH_j^{k-1}}[1].\]
\end{enumerate}

\end{definition-lemma}
In particular, we can express $\scH_j^k$ explicitly as an $n$-tuple of constructible sheaves on $\PR$.

Remark that $\scH_j^k$ and $\und{\scH_j^k}$ differ only in the first part: $\block{j}'$ and  $\block{j}$.
Let $m_j$ be ``$m$" in the $j$-th $\PR$.
We denote by $\CC_{m_j}$ an object $\mu \mathrm{sh}_{C_{\lc\boldsymbol{m}\rc}}(X_\Gamma)$ 
\[((0,0),\dots,(0,0),(\CC_{m_j},\CC_{m_j}),(0,0),\dots,(0,0)),\]
where $(\CC_{m_j},\CC_{m_j})$ is the $j$-th component.
Then, we have the following.
\begin{lemma}\label{jl12-1}
    We have a distinguished triangle:
    \[\CC_{m_j}[-2](=\RG{m_j}\CC_{\PR})\to \block{j}'\to \block{j}\to \CC_{m_j}[-1].\]
\end{lemma}

By Definition-Lemma~\ref{jl12-3},
we get inductive systems of objects in $\mu \mathrm{sh}_{C_{\lc \boldsymbol{m}\rc}}(X_\Gamma)$:
\begin{align}\label{jl11-4}
    \scH^0_j\xrightarrow{\iota} \scH^1_j\xrightarrow{\iota}\scH^2_j\xrightarrow{\iota}\scH^3_j\xrightarrow{\iota}\dots,\\
    \und{\scH^0_j}\xrightarrow{\iota} \und{\scH^1_j}\xrightarrow{\iota}\und{\scH^2_j}\xrightarrow{\iota}\und{\scH^3_j}\xrightarrow{\iota}\dots.   
\end{align}

\begin{definition}
For $1\leq j\leq n$, we define $\scH^\infty_j\ (\mbox{resp.}\ \und{\scH^\infty_j})\in \mu \mathrm{sh}_{C_{\lc \boldsymbol{m}\rc}}(X_\Gamma)$ as the homotopy colimit of (\ref{jl11-4}), i.e.
an object which fits into the distinguished triangle:
\begin{align}
    \label{nov14-2}
\bigoplus_{k=0}^{\infty}\scH^k_j\to \bigoplus_{k=0}^{\infty}\scH^k_j\to \scH^\infty_j\to \bigoplus_{k=0}^{\infty}\scH^k_j[1],
\end{align}
\[(\mbox{resp.}\ \bigoplus_{k=0}^{\infty}\und{\scH^k_j}\to \bigoplus_{k=0}^{\infty}\und{\scH^k_j}\to \und{\scH^\infty_j}\to \bigoplus_{k=0}^{\infty}\und{\scH^k_j}[1] )\]
where the first morphism is defined as the direct sum of $\mathrm{id}\colon \scH^k_j\to \scH^k_j$ (resp. $\und{\scH^k_j}\to \und{\scH^k_j}$) and $-\iota\colon \scH^k_j\to \scH^{k+1}_j$ (resp. $\und{\scH^k_j}\to \und{\scH^{k+1}_j}$).
\end{definition}

\begin{lemma}\label{dec5-2}
For $k\in \ZZ_{\geq 0}$, we have distinguished triangles:
    \begin{align}
        \bC_{m_j}[-2]\to &\scH^k_j\to \und{\scH^k_{j}}\to \bC_{m_j}[-1]\notag\\
        \bC_{m_j}[-2]\to &\scH^\infty_j\to \und{\scH^\infty_{j}}\to \bC_{m_j}[-1].\label{jl16-1}
    \end{align}
\end{lemma}

\begin{proof}
We define $ {\scH^0_{j}}\to \und{\scH^0_{j}}$ as $\block{j}'\to \block{j}$.
  For $k\geq 1$, we take $ {\scH^k_{j}}\to \und{\scH^k_{j}}$ inductively so that it induces a morphism of distinguished triangles:
    \begin{align*}
    \xymatrix{
          \scH_j^{k-1}\ar[r]\ar[d]& \scH_j^{k}\ar[r]\ar[d]& \block{j}^{\ast}[k]\ar[r]\ar[d]^{\mathrm{id}}& \scH_j^{k-1}[1]\ar[d]\\
            \und{\scH_j^{k-1}}\ar[r]& \und{\scH_j^{k}}\ar[r]& \block{j}^{\ast}[k]\ar[r]& \und{\scH_j^{k-1}}[1],
    }
    \end{align*}
    where $\block{j}^{\ast}$ means $\block{j}$ or $\block{j}^{\circ}$ depending on the parity of $i$.
    Then, the first triangle follows from Lemma~\ref{jl12-1}.
    The second triangle follows the obvious fact that
    the homotopy colmit of the inductive system
    \[\CC_{m_j}[-2]\xrightarrow{\mathrm{id}} \CC_{m_j}[-2]\xrightarrow{\mathrm{id}}  \CC_{m_j}[-2]\xrightarrow{\mathrm{id}} \dots\]
    is $\CC_{m_j}[-2]$.
\end{proof}

For $1\leq j\leq n$, let $W_j$ be $W$ of the $j$-th $\PR$.
For $H\in \mu \mathrm{sh}_{C_{\lc \boldsymbol{m}\rc}}(X_\Gamma)$, we denote by $H|_{W_j}$ the restriction $F^j|_{W}=G^j|_{W}$
for an expression $H=((F^1,G^1),\dots,(F^n,G^n))$.

\begin{lemma}\label{nov14-1}
We set $s(j,i):=\min({j,n-j+1,i,n-i+1})$.
Then, we have
\[\scH^k_j|_{W_i}\simeq \calL_{s(j,i)}\oplus \calL_{s(j,i)}[1]\oplus \dots \oplus \calL_{s(j,i)}[k-1]\oplus \wt{\calL_{s(j,i)}}[k],\]
where $\wt{\calL_{s(j,i)}}$ is ${\calL_{s(j,i)}}$ or $\overline{\calL_{s(j,i)}}$ depending on $j$, $i$, $k$.
\end{lemma}

\begin{proof}
If $i$ is neither $j$ nor $n-j+1$, the assertion is obvious (for $\wt{\calL_{s(j,i)}}={\calL_{s(j,i)}}$).

We only prove 
\[\scH_j^1|_{W_{i}}\simeq {\calL_j}\oplus {\calL_{j}}[1],\]
for $i=n-j+1$ and $1\leq j\leq n/2$.
The other cases can be shown similarly by induction.
By the definition, we have a distinguished triangle
\[\ove{\calL_j}(=\scH^{0}_j|_{W_{i}})\to \scH^1_j|_{W_{i}}\to \und{\calL_j}[1](=\block{j}[1]|_{W_{i}})\xrightarrow{\mathfrak{f}_j|_{W_i}} \ove{\calL_j}[1](=\scH^{0}_j[1]|_{W_{i}}).\]
Moreover,
recall the distinguished triangle:
\[\RG{m}\CW[1]\to \ove{\calL_j}\to \calL_j\to \RG{m}\CW[2].\]
Then, applying the commutative diagram:
\[\xymatrix{\und{\calL_j}\ar[r]^{\mathfrak{f}_j|_{W}}\ar[d]&\ove{\calL_j},\\\RG{m}\CC_{W}[1]\ar[ur]&}\]
to the octahedral axiom, we have a distinguished triangle:
\[\calL_j[1]\to \scH_j^1|_{W_j}\to \calL_j\to \calL_{j}[2].\]
However, since $\calL_j\to \calL_{j}[2]$ is always zero,
we have
\[\scH_j^1|_{W_i}\simeq \calL_j\oplus \calL_j[1].\]
   
\end{proof}

\begin{lemma}\label{dec5-7}
For $1\leq i\leq n$, we have
\[\scH^{\infty}_j|_{W_i}\simeq \bigoplus_{k=0}^{\infty}\calL_{s(j,i)}[k],\]
where $s(j,i)$ is the one defined in Lemma~\ref{nov14-1}.
In particular,
 $\scH_j^\infty$ is an object of $\mu \mathrm{sh}_C(X_\Gamma)$.
\end{lemma}

\begin{proof}
We set
\[(\scH_j^k|_{W_i})':= \bigoplus_{l=0}^{k}\scL_{s(j,i)}[l].\]
Then, we have a distinguished triangle in $\Hzsh{\CC_{W_i}}$:
\begin{align*}
    \scH_j^k|_{W_i}\to (\scH_j^k|_{W_i})'\to C_{j,k,i}[k]\to  \scH_j^k|_{W_i}[1],
\end{align*}
where $C_{j,k,i}$ is zero or $\CC_m$ depending on the cases in Lemma~\ref{nov14-1}.
Moreover, 
we have a distinguished triangle:
\[\bigoplus_{k=0}^{\infty}(\scH^k_j|_{W_i})'\to \bigoplus_{k=0}^{\infty}(\scH^k_j|_{W_i})'\to \bigoplus_{k=0}^{\infty}\calL_{s(j,i)}[k]\to \bigoplus_{k=0}^{\infty}(\scH^k_j|_{W_i})'[1],\]
where the first arrow is the identity minus the direct sum of the natural morphism $(\scH^{k}_j|_{W_i})'\to (\scH^{k+1}_j|_{W_i})'$.
We now consider the commutative diagram:
\[\xymatrix{
    \displaystyle\bigoplus_{k=0}^{\infty}(\scH^k_j|_{W_i})'\ar[r]\ar[d]&\displaystyle\bigoplus_{k=0}^{\infty}(\scH^k_j|_{W_i})'\ar[r]\ar[d]&   \displaystyle\bigoplus_{k=0}^{\infty}\calL_{s(j,i)}[k] \ar[d]\\
\displaystyle\bigoplus_{k=0}^{\infty}\scH^k_j|_{W_i}\ar[r]\ar[d]&\displaystyle\bigoplus_{k=0}^{\infty}\scH^k_j|_{W_i}\ar[r]\ar[d]&    \scH^\infty_j|_{W_i}\ar[d]\\
\displaystyle\bigoplus_{k=0}^{\infty} C_{j,k,i}[k]\ar[r]^{\mathrm{id}}&\displaystyle\bigoplus_{k=0}^{\infty} C_{j,k,i}[k]\ar[r]& 0
}.\]
Since all rows and the first two column are distinguished triangles,
the third column is a distinguished triangle by the 9-lemma for a triangulated category.
Hence, we have $\scH^\infty_j|_{W_i}\simeq \bigoplus_{k=0}^\infty \calL_{s(j,i)}[k]$.
    \end{proof}

\begin{theorem}\label{dec5-8}
For $H\in \mu \mathrm{sh}_{C}(X_\Gamma)$, we have a functorial isomorphism:
\[H^0\HOM_{\Ancore}(\scH^\infty_j, H)\simeq H^0\HOM_{\Ancore}(\CC_{m_j}[-2],H),\]
where the morphism is induced by the morphism $\CC_{m_j}[-2]\to \scH^\infty_j$ in (\ref{jl16-1}).
In particular,
we have
\[H^0\HOM_{\Ancore}(\scH^\infty_j, H)\simeq
H^0((H|_{W_j})_{m_j}),
\]
where $(H|_{W_j})_{m_j}\in \sh{m_j}$ is a stalk at $m_j\in W_j$ of $H|_{W_j}=F^j|_{W_j}$ 
for an expression $H=((F^1,G^1),\dots,(F^n,G^n))$.
\end{theorem}

\begin{proof}
    By the distinguished triangle (\ref{jl16-1}),
    to prove the first isomorphism, it is enough to show that 
    \[H^0\HOM_{\Ancore}(\und{\scH^\infty_j},H)=0,\]
    for any $H\in \mu \mathrm{sh}_C(X_\Gamma)$.
By the definition of $\und{\scH^\infty_j}$,
it suffices to show
\[ H^0\HOM_{\Ancore}(\und{\scH^k_j},H)=0\quad (k\in \ZZ_{\geq 1}),\]
which follows from the definition of $\und{\scH_j^k}$ and Lemma~\ref{m29-2}.

The second assertion follows from the fact:
For $L\in \sh{\CC}$ with $\mathrm{SS}(L)\subset T^*_\CC\CC$ we have
\[\HOM_{{\CC}}(\CC_0[-2],L)\simeq H^{0}(L_0),\]
where $L_0$ is a stalk at $0$ of $L$.
\end{proof}

\subsection{Hodge structure}\label{dec20-5} 
We discuss a ``Hodge structure" on $\scH^\infty_j$.
We enhance our discussion in the above to the Hodge setup.

For a constructible sheaf (or a cohomologically constructible complex) $\scF$, 
when we consider (and fix for $\scF$) a mixed Hodge module (or a complex of mixed Hodge modules) whose underlying object is $\scF$,
we use the same symbol $\scF$ to represent it.
For example, we say ``$\CC_\CC[1]$ is a pure Hodge module of weight $1$" instead of that ``the perverse sheaf $\CC_\CC[1]$ is a underlying perverse sheaf of a pure Hodge module of weight $1$". We write the (half) Tate-twisted $\CC_\CC[1]$ as $\CC_\CC[1](s/2)$ for $s\in \ZZ$.

We recall the Fourier transform of a monodromic mixed Hodge module and the definition of Hodge microsheaves on $X_\Gamma$. For more details, see Subsection~\ref{jan7-1} and Section~\ref{jan7-2}.
Let $V$ be a complex line $\CC$ with the origin $0$ with a coordinate $t$.
The category $\sqrt{\mathrm{MHM}^{c,\heartsuit}(V)}$ (resp. $\sqrt{\mathrm{MHM}^{c,\heartsuit}(V,0)}(=\sqrt{\mathrm{MHM}^{c,\heartsuit}_{\mathrm{mon}}(V)})$) is the abelian category of mixed Hodge modules (resp. with a possible singular point at the origin) on $V$ with the half Tate twists.
Similarly, $\sqrt{\mathrm{MHM}^{c,\heartsuit}(V,0)_{\mathrm{unip}}}$ is the subcategory of $\sqrt{\mathrm{MHM}^{c,\heartsuit}(V,0)}$ consisting of objects whose underlying perverse sheaves are unipotent, which are called unipotent monodromic mixed Hodge modules.
In the following, we omit the symbol $\sqrt{\quad }$.

For $\scF\in \MHMxh{V}$,
the specialization $\nu_0(\scF)\in \MHMxh{T_0V,0}_{\unip}$ is defined as the object corresponding to the tuple $(\psi_{t}\scF,\phi_t\scF,\can,\Var)$.

Recall that if $\scF$ is unipotent and monodromic, i.e. $\scF$ can be recovered from the tuple (=specialization) $(\psi_{t}\scF,\phi_t\scF,\can,\Var)$.
Then, the Fourier transform $\mathrm{FL}^{\mathrm{pre}}(\scF)$ of $\scF$ is defined as the object on the dual space $V^*$ corresponding to
\[(\phi_t\scF,\psi_t\scF(-1),-\Var,\can(-1)).\]
Moreover, $\mathrm{FL}(\scF)$ is defined as
\[\mathrm{FL}(\scF)=\mathrm{FL}(\scF)(1/2).\]
Then, we have 
\[\mathrm{FL}(\mathrm{FL}(\scF))=\scF.\]
This operation can be extended to the one on $\mathrm{MHM}(V,0)_{\unip}$.
By using this operation,
we have defined the category of Hodge microsheaves $\mu \mathrm{M}_{C}(X_{\Gamma})$ and $\mu \mathrm{M}_{C_{\lc \boldsymbol{m}\rc}}(X_{\Gamma})$.
We denote the (non-full) image of $\Ind M$ in \ref{ex:Hodge-Tate2} of Example~\ref{ex:Hodge-Tate} under the restriction $\bP^1 \to V$ by $\mathrm{M}(V)'$, and set $\stM(V):=\sqrt{\stM(V)'}$. The monodromic version is denoted by $\mathrm{M}(V,0)$.
Note that $\mathrm{FL}$ (resp. $\nu_0$) sends $\mathrm{M}(V,0)$ (resp. $\mathrm{M}(V)$) to $\mathrm{M}(V^{*},0)$ (resp. $\stM(T_0V,0)$).
Recall that $H^0\stM(V)$ is a saturated triangulated (non-full) subcategory of $H^0\mathrm{MHM}(V)$.
We use the same definition for $\bP^1$ and $W$.

The category $\muMGm$ is defined in the same way as in $\mu \mathrm{sh}_{C_{\{\boldsymbol{m}\}}}(X_{\Gamma})$.  
An object in $\muMGm$ can be expressed as
\[((F^1,G^1),\dots,(F^n,G^n)),\]
where 
$F^i$ (resp. $G^i$) is an object in $\stM(V_l)$ (resp. $\stM(V_r)$) on $V_l$ (resp. $V_r$) in $i$-th $\PR$ with the condition:
\[\mathrm{FL}(\nu_{r}(G^i))=\nu_l(F^{i+1})\quad (1\leq i\leq n-1).\]
Remark that this condition is equivalent to the condition
\[\nu_r(G^i)=\mathrm{FL}(\nu_l(F^{i+1}))\quad (1\leq i\leq n-1).\]

We use diagrams which is an enhancement of the ones defined in Definition~\ref{s10-3-2} as follows.
For a distinguished triangle $A'\to A\to A''\to A'[1]$ in $H^0\stM(V)$ and $B\in \stM(V)$, and $i, j,s_i,s_j\in \ZZ\ (i\neq j)$,
if 
there exist subspaces $L'_k\subset \HOM(A',B[k](s_k))$ and $L''_k\subset \HOM(A'',B[k](s_k))$ with an exact sequence
\[0\to L''_k\to \HOM(A,B[k](s_k))\to L'_k\to 0\] for $k= i,j$,
and $\HOM(A,B[k](s))=0$ for other $k$ and $s$,
then we write
\[
{\xyhidaridt{A'}{A}{A''}{B,}{i(s_i)\times d'_i,\ j(s_j)\times d_j'}{i(s_i)\times d''_i,\ j(s_j)\times d''_j}}\] 
where 
$d'_k=\dim{L'_k}$ and $d''_k=\dim{L''_k}$.
Here, above $\HOM$ means $\HOM_{\stM(V)}$.
We use the same diagrams for somewhat different situations.

The followings are standard facts. 
Some of them can be confirmed by using the description of mixed Hodge modules explained in Subsection~\ref{section:linearalgebraicHodge} and Definition~\ref{fev5-1}.

\begin{lemma} 
Let $V$ be a complex line $\CC$ with the origin $0$ and another point $m$.
We set $W=V\setminus \{0\}$ and $T=V\setminus \{m\}$.
\begin{enumerate}
    \item The object $\CC_V[1]$ (resp. $\CC_0$) is a pure Hodge module of weight $1$ (resp. $0$),
    $\CC_W[1]$ is a mixed Hodge module with $\Gr^{W}_w(\CC_W[1])=0$ for $w\neq 0,1$,
    and we have a distinguished triangle in $H^0\stM(V)$:
    \[\CW\to \CV\to \CC_0\to \CW[1].\]
    \item  
   We have
\begin{align*}
    \mathrm{FL}^{\mathrm{pre}}{(\CV[1])}\simeq \CC_0(-1),&\     \mathrm{FL}^{\mathrm{pre}}{(\CC_0)}\simeq \CV[1],\quad  \mbox{and hence}
    ,\\
    \mathrm{FL}{(\CV[1])}\simeq \CC_0(-1/2),&\ 
    \mathrm{FL}{(\CC_0)}\simeq \CV[1](1/2).
\end{align*}
    
\item The object $\RG{W}\CV[1]$ is a mixed Hodge module with $\Gr^W_w(\RG{W}\CV[1])=0$ for $w\neq 1,2$,
$\RG{0}\CV[1]$ is isomorphic to $\CC_0[-1](-1)$ in $H^0\stM(V)$
and we have a distinguished triangle in $H^0\stM(V)$:
\[\RG{0}\CV\to \CV\to \RG{W}\CV\to \RG{0}\CV[1].\]

    \item The morphism $\CW\to \CW[1]$ (defined up to constant) in $\Hzsh{W}$ is enhanced to be a morphism in $H^0\stM(W)$
    \[\CW\to \CW[1](1).\]

    \item By induction, we can enhance $A_s[1]$ to be an object in $\stM(V)$ so that
    we have distinguished triangles and diagrams in $H^0\stM(V)$:
    \begin{align*}
    \CW(s)\to &A_s\to A_{s-1}\to \CW[1](s),        \\
        A_{s-1}(1)\to &A_s\to \CW\to A_{s-1}[1](1),
    \end{align*}
    \[\xyhidaridt{\CW(s)}{A_s}{A_{s-1}}{\CW,}{1(s)}{0(0)}\xymigidt{\CW}{A_{s-1}(1)}{A_s}{\CW}{0(-(s-1))}{1(1)}.\]
Moreover, we enhance $\calL_s[1]$ to be an object on $W$ as
\[\calL_s[1]:=A_s[1]|_{W}.\]

    \item The mixed Hodge module structure on $B_s[1]$ is defined as the push forward of $A_s|_{W}[1]$ along $W\hookrightarrow V$.
Then, we have
\begin{align*}
    \mathrm{FL}^{\mathrm{pre}}{(A_s[1])}\simeq B_s[1],&\     \mathrm{FL}^{\mathrm{pre}}{(B_s[1])}\simeq A_s[1](-1),\quad  \mbox{and hence,}\\
    \mathrm{FL}{(A_s[1])}\simeq B_s[1](1/2),&\ 
    \mathrm{FL}{(B_s[1])}\simeq A_s[1](-1/2).
\end{align*}

\item 
We have diagrams in $H^0\stM(V)$:
\[\xyhidaridt{\CW(s-1)}{A_s}{A_{s-1}}{\CV,}{1(s)}{0(0)}\xyhidaridt{B_1(s-1)}{B_s}{B_{s-1}}{\CC_0.}{0(s-1)}{-1(-1)}.\]

\item 
 We can enhance $P_s[1]$ (resp. $Q_s[1]$) to be an object in $\stM(V)$ so that
    we have distinguished triangles in $H^0\stM(V)$
\begin{align*}
    \CV(s-1)&\to P_s\to A_{s-1}\to \CV[1](s-1)\qquad\qquad\\
    (\mbox{resp.\ } \CC_0[-1](s-2)&\to Q_s\to B_{s-1}\to \CC_0(s-2)),
\end{align*}

Moreover, we have
\begin{align*}
    \mathrm{FL}^{\mathrm{pre}}{(P_s[1])}&\simeq Q_s[1],\     \mathrm{FL}^{\mathrm{pre}}{(Q_s[1])}\simeq P_s[1](-1),\quad  \mbox{and hence,}
    \\
    \mathrm{FL}{(P_s[1])}&\simeq Q_s[1](1/2),\ 
    \mathrm{FL}{(Q_s[1])}\simeq P_s[1](-1/2).
\end{align*}

\item The objects $\ove{A_s}$, $\und{A_s}$ are enhanced to be objects in $\stM(V)$
inductively so that 
we have distinguished triangles and diagrams in $H^0\stM(V)$
\begin{align*}
\oveA{s-1}(1)\to &\oveA{s}\to \CW\to \oveA{s-1}[1](1)\qquad\qquad \\
(\mbox{resp.\ } \CW(s-1)\to &\undA{s}\to \undA{s-1}\to \CW[1](s-1)),
\end{align*}
\[\xymigidt{\CW}{\oveA{s-1}}{\oveA{s},}{\CW}{\emptyset}{1(1)}\quad \xyhidaridt{\CW}{\undA{s}}{\undA{s-1}}{\CW}{1(1)}{\emptyset}.\]

\item 
We take $\RG{T}\CW\to\CW[1](1)$ and $\RG{T}\CW\to \ove{A_{s}}[1](1)$ such that the diagram commutes:
\[\xymatrix{\CW \ar[dr]\ar[d]& \\\RG{T}\CW\ar[r]&\CW[1](1),}\quad 
\xymigisita{\RG{T}\CW}{\ove{A_{s}}(1)}{\CW[1](1).}
\]
Using it, we define
$\ove{\und{A_s}}\in \stM(V)$ so that it fits into the distinguished triangle:
\[\ove{A_{s-1}}(1)\to \ove{\und{A_{s}}}\to \RG{T}\CW\to \ove{A_{s-1}}[1](1).\]
Moreover, we also enhance $\und{\ove{P_s}}$ to be an object in $\stM(V)$ so that we have
\[\ove{P_{s-1}}(1)\to \ove{\und{P_{s}}}\to \RG{T}\CW\to \ove{P_{s-1}}[1](1).\]

\item 
The object $\wt{\und{P_1}}[1]$ is also enhanced to be an object in $\stM(V)$ so that we have
\[\CC_T[1](1)\to \wt{\und{P_1}}[1]\to \CC_m\to \CC_T[2](1).\]

\item The morphisms $A_s\to \und{A_s}$, $P_s\to \und{P_s}$, $\ove{A_s}\to A_s$, $\ove{P_s}\to P_s$, $\ove{A_{s-1}}\to \ove{A_{s}}$, $\und{A_{s}}\to \und{A_{s-1}}$ etc.
can be enhanced to the ones (without twists) in $H^0\stM(V)$.

\end{enumerate}    
\end{lemma}

By using these facts, let us equip $\block{j}$, $\scH^k_j$ and $\scH^\infty_j$ with some mixed Hodge module structures.

\begin{lemma}\label{jan29-7}
    \begin{enumerate}
        \item For $n\geq 2$ and $1\leq  j\leq n/2$, we put $w_j:=(n-2j+1)/2$.
        Then, the tuple:
\begin{align*}
    ((P_1((j-1)/2),Q_2((j-1)/2)),&\dots, (P_{j-1}(1/2),Q_{j}(1/2)),\\
    (\und{P_j},&\und{A_j}), 
    \\
    (B_j(1/2),A_j(1/2)), \dots&,(B_j(w_j-1/2),A_j(w_j-1/2)),
    \\
    (\ove{B_j}(w_j),&\ove{P_j}(w_j)),\\
    (Q_j(w_j+1/2),P_{j-1}(w_j+1/2)),\dots&,(Q_2((n-j)/2),P_1((n-j)/2))),
\end{align*}
defines an object in $\muMGm$ whose underlying object is $\block{j}$.
Here, in the case $j=1$ (resp. $n/2$),
the first and fifth part (resp. third) part are removed.
We denote it by the same symbol $\block{j}$.
The object $\block{j}^{\circ}$ is enhanced in the same way.
For $n/2<j\leq n$,
$\block{j}\in \muMGm$ is defined as $\block{n-j+1}^\circ$.

\item 
When $n$ is odd: $n=2{n_0}+1$ ${n_0}\in \ZZ_{\geq 1}$,
$\block{{n_0}+1}\in \muMGm$ is defined to be 
\begin{align*}
      ((P_1({n_0}/2),Q_2({n_0}/2)),\dots, &(P_{{n_0}}(1/2),Q_{{n_0}+1}(1/2)),\\
    (\und{\ove{P_{{n_0}+1}}},&\und{\ove{P_{{n_0}+1}}}),\\
    (Q_{{n_0}+1}(1/2),P_{{n_0}}(1/2)),\dots,&(Q_2({n_0}/2),P_1({n_0}/2))).
\end{align*}

\item When $n=1$, we define $\block{1}\in \muMGm$ as
\[\block{1}=(\wt{\und{P_1}},\wt{\und{P_1}}).\]

\item We can also enhance $\block{j}'$ to be an object in $\muMGm$ in a similar way.
\end{enumerate}    
\end{lemma}

\begin{lemma}\label{dec5-1}
    The morphisms $\mathfrak{f}_s$ in Definition-Lemma~\ref{jl11-1} are enhanced to be morphisms in $H^0\stM(V)$:
    \begin{align*}
        \undA{s}\to &\oveA{s}(-1),\quad           \und{P_s}\to \ove{P_s}(-1),\quad 
        \und{A_s}\to \ove{B_s}(-1),\\
        \und{P_s}\to &\ove{P_s}(-1),\quad     \und{\ove{P_s}}\to \ove{\und{P_s}}(-1),\quad                        \und{\wt{P_s}}\to \und{\wt{P_s}}(-1).
    \end{align*}
\end{lemma}
\begin{proof}
Recall that the morphism 
$\RG{T}\CW\to \CWT$ in $\Hzsh{V}$ satisfies the commutative diagram:
\[\xysitamigiue{\RG{T}\CW}{\RG{m}\CW[1]}{\CWT}.\]
The morphism $\RG{m}\CW[1](=\CC_m[-1])\to \CWT$ in $\Hzsh{V}$ is enhanced to be the following morphism in $H^0\stM(V)$: 
    \[\RG{m}\CW[1](=\CC_m[-1](-1))\to \CWT(-1),\]
    which induces a morphism in $H^0\stM(V)$: 
    \[\und{A_1}\to \ove{A_1}(-1).\]
    The remaining cases follow from this fact.
\end{proof}

\begin{lemma}\label{jan8-5}
For $1\leq j\leq n$,
we set 
\[\wt{j}:=\left\{
\begin{array}{cc}
    w_j+1 &1\leq j\leq  n/2  \\
    1&j=n_0+1\ \mbox{for $n=2n_0+1$}\\
 w_{n-j+1}+1&n/2+1\leq j\leq n,
\end{array}
\right.\]
where $w_j$ is defined in Lemma~\ref{jan29-7}.
Then, the morphism $\block{j}^{\circ}\to \block{j}$ in Lemma~\ref{jl11-2} can be enhanced to be the morphism in $\muMGm$
\[\block{j}^{\circ}(\wt{j}) \to \block{j},\]
    which is also written as $\mathfrak{f}_j$.
\end{lemma}

\begin{proof}
We only consider the case $1< j< n/2$.
We can see it in the same way in the other cases.
The morphism $\block{j}^{\circ}\to \block{j}$ in Lemma~\ref{jl11-2}.
was induced by $\mathfrak{f}_j$ on the $(n-j+1)$-st $\PR$ of $\Ancore$, i.e.
the pair of 
$\mathfrak{f}_j\colon  \und{A_j}\to\ove{B_j}$ and $\mathfrak{f}_j\colon \und{P_j}\to \ove{P_j}$.
Since the $(n-j+1)$-st component of 
$\block{j}^{\circ}(\wt{j})$
(resp. 
$\block{j}$) 
$\in \muMGm$ is
\[ (\und{A_j}(\wt{j}),\und{P_j}(\wt{j}))\]
\[ (\mbox{resp.}\ (\ove{B_j}(w_j),\ove{P_j}(w_j)),\]
the morphism $\und{A_j}\to \ove{B_j}(-1)$ and $\und{P_j}\to \ove{P_j}(-1)$ in Lemma~\ref{dec5-1} induce the desired morphism.
\end{proof}

\begin{lemma}\label{dec5-6}
   The object $\scH^k_j$ and $\und{\scH^k_j} \in \mu \mathrm{sh}_{C_{\{\boldsymbol{m}\}}}(X_\Gamma)$ are enhanced to be objects $\scH^{k,H}_j$ and $\und{\scH_j^{k,H}}$ in $\muMGm$ inductively so that:
   \begin{enumerate}
       \item $\scH^{0,H}_j=\block{j}'$ and $\und{\scH^{0,H}_j}=\block{j}$ as objects in  $\muMGm$.

\item When $k$ is odd, we have distinguished triangles in $H^0\muMGm$:
   \begin{align*}
    \scH^{k-1,H}_j\to \scH_j^{k,H}\to \block{j}^{\circ}[k](k\cdot \wt{j})\to \scH^{k-1,H}_j[1],
\end{align*}
and the commutative diagram (for $k\geq 2$):
\[\xymatrix{\block{j}^{\circ}[k](k\cdot \wt{j})\ar[r]\ar[dr]_{\mathfrak{f}_j}&\scH_j^{k-1,H}[1]\ar[d]\\&\block{j}[k]((k-1)\cdot \wt{j})}.\]
Similarly, we have
\begin{align*}
    \und{\scH^{k-1,H}_j}\to \und{\scH_j^{k,H}}\to \block{j}^{\circ}[k](k\cdot \wt{j})\to \und{\scH^{k-1,H}_j}[1].
\end{align*}

   \item When $k$ is even, we have   
 we have distinguished triangles in $H^0\muMGm$:
   \begin{align*}
    \scH^{k-1,H}_j\to \scH_j^{k,H}\to \block{j}[k](k\cdot \wt{j})\to \scH^{k-1,H}_j[1],
\end{align*}
and the commutative diagrams:
\[\xymatrix{\block{j}[k](k\cdot \wt{j})\ar[r]\ar[dr]_{\mathfrak{f}_{n-j+1}}&\scH_j^{k-1,H}[1]\ar[d]\\&\block{j}^\circ[k]((k-1)\cdot \wt{j})}.\]
Similarly, we have
\[ \und{\scH^{k-1,H}_j}\to \und{\scH_j^{k,H}}\to \block{j}[k](k\cdot \wt{j})\to \und{\scH^{k-1,H}_j}[1].\]

   \end{enumerate}
\end{lemma}

As we have seen in Lemma~\ref{nov14-1}, for $1\leq i\leq n$,
we can express $\scH^k_j|_{W_i}\in \Hzsh{W_i}$ as
\[\scH^k_j|_{W_i}\simeq \calL_s\oplus \calL_s[1]\oplus \dots\oplus \wt{\calL_s}[k]=\bigoplus_{l=0}^{k-1}\calL_s[l]\oplus \wt{\calL_s}[k],\]
    where $s$ is $s(j,i)$ defined in Lemma~\ref{nov14-1} and $\wt{\calL_s}$ is $\calL_s$ or $\ove{\calL_s}$ depending on $j,i,k$.
    As objects in $H^0\stM(W_i)$,
the twists of them are determined as follows.
    
\begin{corollary}\label{dec5-5}
For $1\leq i,j\leq n$ and $u\in \bZ_{\geq 0}$,
we define
\[d^{j,i}_u:=\left\{
\begin{array}{cc}
     |j-i|/2&  \mbox{$u$:even}\\
     |n-j+1-i|/2& \mbox{$u$:odd},
\end{array}
\right.\]
and 
\[e^{j,i}_u:=u\cdot \wt{j}+d^{j,i}_u.\]
Then,
as an object in $H^0\stM(W_i)$,
we have
\[\scH^{k,H}_j|_{W_i}\simeq \bigoplus_{u=0}^{k-1}\calL_{s(j,i)}[u](e_{u}^{j,i})\oplus \wt{\calL_{s(j,i)}}[k](e_{k}^{j,i}).\]
\end{corollary}

\begin{proof}
This follows from the fact that
\[\block{j}^{\ast}[u](u\cdot \wt{j})|_{W_i}={\wt{\calL_{s(j,i)}}}[u](u\cdot \wt{j}+d^{j,i}_u),\]
where we have set
\[\block{j}^{\ast}=\left\{\begin{array}{cc}
     \block{j}&  \mbox{$u$:even,}\\
     \block{j}^{\circ}&  \mbox{$u$:odd.}
\end{array}\right.\]
\end{proof}

\begin{proposition}
        \begin{enumerate}
        \item $\scH^\infty_{j}$ and $\und{\scH^\infty_{j}}\in \mu\mathrm{sh}_{C_{\lc \boldsymbol{m}\rc}}(X_\Gamma)$ are enhanced to be objects $\scH^{\infty,H}_{j}$ 
 and $\und{\scH^{\infty,H}_{j}}$ in $\muMGm$.

        \item With the notation in Corollary~\ref{dec5-5}, in $\stM(W_i)$ we have 
        \begin{align}\label{jan8-2}
        \scH^{\infty,H}_{j}|_{W_i}=\bigoplus_{u=0}^\infty\calL_{s(j,i)}[u](e_u^{j,i}).
        \end{align}
        In particular, $\cH_j^{\infty,H}\in \muMG$.
    \end{enumerate}
\end{proposition}

\begin{proof}
   By Lemma~\ref{dec5-6}, we have an inductive system in $H^0\muMGm$:
\begin{align*}
    \scH^{0,H}_j\to \scH^{1,H}_j\to \scH^{2,H}_j\to \dots,\\
    \und{\scH^{0,H}_j}\to \und{\scH^{1,H}_j}\to \und{\scH^{2,H}_j}\to \dots,
\end{align*}
and we thus obtain the colimit objects in $H^0\muMGm$, whose underlying objects are $\scH^\infty_j$ and $\und{\scH^\infty_j}$.

For the second assertion, we can apply the same argument as in the proof of Lemma~\ref{dec5-7}.
\end{proof}

We regard $\CC_{m_k}\in \stM(\PR)$ as the weight-zero object in $\muMGm$.
Then, we have a Hodge enhancement of Lemma~\ref{dec5-2} as follows.

\begin{lemma}
For $1\leq j\leq n$,
    we have a distinguished triangle in $H^0\muMGm$:
\[\CC_{m_j}[-2](-1)\to \scH^{\infty,H}_j\to \und{\scH^{\infty,H}_j}\to \CC_{m_j}[-1](-1).\]
\end{lemma}

\begin{theorem}[Theorem~\ref{thm:comparisonGinzburng}]\label{jan9-5}
    For $H\in \muMG$, we have a functorial isomorphism:
\[\HOM_{\muMG}(\scH^{\infty,H}_j, H)\simeq \HOM_{\muMGm}(\CC_{m_j}[-2](-1),H),\]
where the morphism is induced by the morphism $\CC_{m_j}[-2](-1)\to \scH^{\infty,H}_j$.
Moreover, for an expression
$H=((F^1,G^1),\dots,(F^n,G^n))$,
assume that we have a decomposition:
\[H|_{W_i}=F^i|_{W_i}=\bigoplus_{k\in \ZZ}\scF^k[k],\]
where $\scF^k$ is a direct sum of mixed Hodge modules without half-Tate twists whose underlying perverse sheaves are $1$-shifted local systems of finite rank on $W_i$. 
Then,
we have
\[H^0\HOM_{\muMG}(\scH^{\infty,H}_j, H)\simeq
(F^0\scF^{-1})|_{m_j}\cap (W_1\scF^{-1})|_{m_j},
\]
where we denote the Hodge (resp. weight) filtration of the underlying $D$-module by $F^\bullet\scF^{-1}$ (resp. $W_\bullet\scF^{-1}$) and $|_{m_j}$ is the sheaf theoretical restriction to $m_j$.
\end{theorem}

\begin{proof}
Note 
the vanishing $\HOM_{\mu \mathrm{sh}_{C_{\lc \boldsymbol{m}\rc}}{(X_\Gamma)}}(\und{\scH^k_{j}},H)=0$ for the underlying constructible sheaves in the proof of Theorem~\ref{dec5-8}.
Hence, by Corollary~\ref{theorem:microHodgewrapping}, we conclude that $\cH_j^{\infty,H}$ is the Hodge wrapping of $\bC_{m_j}$. This proves the first claim.

For the second part,
we first note that there is no non-zero morphism
$\CC_{m_j}[-2](-1)\to \scF^{k}[k]$ if $k\neq -1$.
Moreover, we have
\begin{align*}
   \HOM_{H^0\stM(W)}(\CC_{m_j}[-2](-1),\scF^{-1}[-1])
\simeq &\HOM_{H^0\stM(W)}((i_{m_j})_{!}\CC_{m_j}[-2](-1),\scF^{-1}[-1])\\
\simeq &\HOM_{H^0\stM(m_j)}(\CC_{m_j}[-2](-1),(i_{m_j})^!\scF^{-1}[-1]).
\end{align*}
Since the underlying perverse sheaf of $\scF^{-1}$ is a $1$-shifted constant sheaf around $m_j$,
we have
\[(i_{m_j})^!\scF^{-1}\simeq \psi_{m_j}\scF^{-1}[-1](-1),\]
where $\psi_{m_j}$ is the nearby cycle functor at $m_j$,
and
\begin{align*}
    F^\bullet(\psi_{m_j}\scF^{-1})&=(F^\bullet\scF^{-1})|_{m_j}\\
W_\bullet(\psi_{m_j}\scF^{-1})&=(W_{\bullet+1}\scF^{-1})|_{m_j}.
\end{align*}
Combining them, we obtain the desired result.
\end{proof}

\begin{corollary}\label{jan8-1}
    For $1\leq i, j\leq n$ and $k,s\in \ZZ$, we have
    \[H^0\HOM_{\muMG}(\scH^{\infty,H}_j, \scH^{\infty,H}_i(s/2)[k])\simeq \left\{ \begin{array}{cc}
         \CC& \left(\substack{k\leq 0,\  s\leq 0,\\e^{i,j}_{-k}+s/2\in \ZZ \ \mbox{and}\\0\leq -(e_{-k}^{i,j}+s/2)\leq s(i,j)-1}\right)\\
         0& \mathrm{otherwise}.
    \end{array}\right.\]
\end{corollary}

\begin{proof}
    
Recall that
the Hodge (resp. weight) filtration of the Tate-twisted constant Hodge module $\CW[1](s)$ ($s\in \ZZ$) jumps only at the degree $-s$ (resp. $1-2s$). 
Moreover, the mixed Hodge module $\calL_\ell[1]$ is decomposed into
\[\CW[1]\oplus\CW[1](1)\oplus \dots \oplus \CW[1](\ell-1)\]
around a point $m\in W$.
Therefore, 
we have
\begin{align}\label{jan8-4}
    F^0(\calL_\ell[1](s))|_{m}\cap W_1(\calL_\ell[1](s))|_{m}
=\left\{
\begin{array}{cl}
     \CC&  0\leq -s\leq \ell-1\\
     0& \mbox{otherwise}
\end{array}
\right.
\end{align}
By (\ref{jan8-2}), we have
\[\scH^{\infty,H}_i(s/2)[k]|_{W_j}=\bigoplus_{u=0}^{\infty}\calL_{s(i,j)}[u+k](e^{i,j}_u+s/2).\]
In this case, 
we have
\[\scF^{-1}=\left\{\begin{array}{cc}
     \calL_{s(i,j)}[1](e^{i,j}_{-k}+s/2)& k\leq 0  \\
     0& k>0.
\end{array}\right.
\] 
If $e^{i,j}_{-k}+s/2\notin \ZZ$,
then $H^0\HOM_{\muMG}(\scH^{\infty,H}_j, \scH^{\infty,H}_i(s/2)[k])$ vanishes.
So, we assume that $k\leq 0$ and $e^{i,j}_{-k}+s/2\in \ZZ$.
Then, by (\ref{jan8-4}), we obtain
\[F^0(\scF^{-1})|_{m_j}\cap W_1(\scF^{-1})|_{m_j}=
\left\{\begin{array}{cc}
     \CC&  0\leq -(e_{-k}^{i,j}+s/2)\leq s(i,j)-1\\
     0& \mbox{otherwise}.
\end{array}\right.
\]
Since $e_{-k}^{i,j}$ is nonnegative,
we have $s\leq 0$ if $0\leq -(e_{-k}^{i,j}+s/2)\leq s(i,j)-1$.
This completes the proof.
\end{proof}

By using this corollary, we give a proof of Lemma~\ref{jan8-8} and Lemma~\ref{jan9-1}.

\begin{proof}[Proof of Lemma~\ref{jan8-8}]
    Remark that $\wt{j}$ in Lemma~\ref{jan8-5} is greater than or equal to $1$.
Therefore, for $k\leq 0$ we have
\begin{align}\label{jan8-6}
    e_{-k}^{i,j}\geq -k.
\end{align}
Hence, 
if $0\leq -(e_{-k}^{i,j}+s/2)$, we have
\[-k\leq -s/2.\]
In particular, $-k\leq -s$.
Then, if we set $C^{a,b}:=\bigoplus_{k,s\in \ZZ}H^0\HOM_{\muMG}(\scH^{\infty,H}_j, \scH^{\infty,H}_i(s/2)[k])$ for $a=k-s$ and $b=-s$,
then we have
\[
C^{a,b}=0\quad (a<0\ \mbox{or}\ b<0).
\]
Moreover, by (\ref{jan8-6}),
If $s=0$ and $k<0$, then $0\leq -(e_{-k}^{i,j}+s/2)$ does not hold, and hence
\[
C^{a,0}=0\quad (a>0).
\]
Finally, we assume $a=0$ and $b>0$, that is, $k=s$ and $s<0$.
By (\ref{jan8-6}) again,
we have
\[e_{-k}^{i,j}+k/2\geq -k+k/2=-k/2>0,\]
and hence
\[C^{0,b}=0\quad (b>0).\]
We then have
\[\bigoplus_{a,b}H^\bullet(B)^{a,b}=\bigoplus_{k,s\in \ZZ}H^0\HOM_{\muMG}(\bigoplus_{j=1}^n\scH^{\infty,H}_j, \bigoplus_{j=1}^n\scH^{\infty,H}_j(s/2)[k]),\]
where $a=k-s$ and $b=-s$.
Then, this is Adams connected.
By Lemma~\ref{lem:adams-connected},
this implies Lemma~\ref{jan8-8}.
\end{proof}

\begin{proof}[Proof of the second part of Lemma~\ref{jan9-1}]
Let us first compute 
\[H^0\HOM_{\muMG}(\bigoplus_{j=1}^n\scH^{\infty,H}_j, \bigoplus_{j=1}^n\scH^{\infty,H}_j(0)[0]),\] the augmentation module in $\mathrm{Mod}(B)$.
If $i\neq j$,
then $e^{j,i}_0=|j-i|/2\neq 0$, and hence the condition $0\leq -e^{j,i}_{0}$ does not hold.
Therefore, by Corollary~\ref{jan8-1}, we have
\begin{align}
    H^0\HOM_{\muMG}(\bigoplus_{j=1}^n\scH^{\infty,H}_j, \bigoplus_{j=1}^n\scH^{\infty,H}_j(0)[0])&=
\bigoplus_{j=1}^nH^0\HOM_{\muMG}(\scH^{\infty,H}_j, \scH^{\infty,H}_j(0)[0])\notag\\
&\simeq 
\bigoplus_{j=1}^n\CC \cdot 1_{\scH^{\infty,H}_j}.\label{jan22-2}
\end{align}

On the other hand, by Theorem~\ref{jan9-5}, we have
\begin{align*}
    H^0\HOM_{\muMG}(\scH^{\infty,H}_j, \CC_{\PR_i}(s/2)[k])\simeq& 
H^0\HOM_{\muMGm}(\CC_{m_j}[-2](-1), \CC_{\PR_i}(s/2)[k])\\
\simeq& \left\{\begin{array}{cc}
     \CC& i=j,\ s=0,\ k=0 \\
     0&  \mbox{otherwise}.
\end{array}\right.
\end{align*}
Therefore, we have
\begin{align}\label{jan22-1}
    \bigoplus_{k,s\in \ZZ}H^0\HOM_{\muMG}(\bigoplus_{j=1}^n\scH^{\infty,H}_j, \bigoplus_{i=1}^{n}\CC_{\PR_i}(s/2)[k])\simeq \bigoplus_{j=1}^n\CC.
\end{align}

From (\ref{jan22-2}) and (\ref{jan22-1}), we conclude that the image of 
$\bigoplus_{i=1}^{n}\CC_{\PR_i}$ under (\ref{jan9-3}) is (quasi) isomorphic to the augmentation module
$H^0\HOM_{\muMG}(\bigoplus_{j=1}^n\scH^{\infty,H}_j, \bigoplus_{j=1}^n\scH^{\infty,H}_j(0)[0])$.
\end{proof}

\subsection{The morphisms between $\bC_{\PR}$ in $\muMG$}\label{jan8-10}
In this subsection, we compute the right hand side of (\ref{jan7-3}):
\[\bigoplus_{s\in \ZZ}\mathrm{Hom}_{\muMG}(\bigoplus_{i=1}^n\CC_{\PR_i},\bigoplus_{i=1}^n\CC_{\PR_i}(s/2))\] and give a proof of Corollary~\ref{apr-1} and Lemma~\ref{jan7-4-2}.
The object $\CC_{\PR_i}\in \mu \mathrm{sh}_{C_{\lc \boldsymbol{m}\rc}}{(X_{\Gamma})}$ is expressed as 
\[((0,0),\dots,(0,\CC_r[-1](-1/2)),\CC_{\PR},(\CC_l[-1](-1/2),0),\dots,(0,0)),\]
where $\CC_{\PR}$ is on the $i$-th $\PR$.

\begin{lemma}\label{dec27-2} Let us consider $\CC_{\PR_i}$ and $\CC_{\PR_j}$ for $1\leq i,j\leq n$.
\begin{enumerate}
    \item If $j$ is not $i-1$, $i$ or $i+1$,
    then we have $H^k\HOM_{\muMG}(\CC_{\PR_i},\CC_{\PR_j}(s/2))=0$ for $k,s\in \ZZ$.
    \item If $j=i$, then we have
    \[H^k\HOM_{\muMG}(\CC_{\PR_i},\CC_{\PR_i}(s/2))\simeq \left\{\begin{array}{cc}
         \CC& k=0,\ s=0 \\
         \CC& k=2,\ s=2\\
         0&\mbox{otherwise}
    \end{array}\right.\]
    \item If $j=i-1$ or $j=i+1$, then we have
    \[H^k\HOM_{\muMG}(\CC_{\PR_i},\CC_{\PR_j}(s/2))\simeq \left\{\begin{array}{cc}
         \CC& k=1,\ s=1 \\
         0&\mbox{otherwise}
    \end{array}\right.\]
    
\end{enumerate}
\end{lemma}
\begin{proof}
    The morphisms in $H^0\stM(\PR)$ (for a point $0\in \PR$):
    \begin{align*}
        \CC_{\PR}\to \CC_{\PR}[0](0),\quad \CC_{\PR}\to \CC_{\PR}[2](1)\quad \mbox{and}\quad 
        \CC_{\PR}\to \CC_{0}[0](0)
    \end{align*}
defines the morphisms
    \begin{align}\label{dec27-1}
        \CC_{\PR_i}\to \CC_{\PR_i}[0](0),\quad
        \CC_{\PR_i}\to \CC_{\PR_i}[2](1)\quad \mbox{and}\quad 
        \CC_{\PR_i}\to \CC_{\PR_{j}}[1](1/2)
    \end{align}
respectively for $j=i-1\mbox{\ or\ } i+1$ in $H^0\muMG$.
This implies the assertion.
\end{proof}

\begin{proof}[Proof of Corollary~\ref{apr-1}] 
The proof is now immediate from (\ref{mar19-1}),  Definition~\ref{def:AdamsKoszul} and
Lemma~\ref{dec27-2}.
\end{proof}

We denote the morphism $\CC_{\PR_i}\to \CC_{\PR_i}[0](0)$ (resp.
$\CC_{\PR_i}\to \CC_{\PR_i}[2](1)$,
$\CC_{\PR_i}\to \CC_{\PR_{i+1}}[1](1/2)$,
$\CC_{\PR_i}\to \CC_{\PR_{i-1}}[1](1/2)$) by $\frak{i}_i$ (resp.
$\frak{d}_i$, $\frak{e}_{i,i+1}$, $\frak{e}_{i,i-1}$).

\begin{lemma}\label{dec27-3}
 In $H^0\muMG$,  we have
    \begin{align*}
        \frak{e}_{i+1,i}\circ \frak{e}_{i,i+1}=\frak{e}_{i-1,i}\circ \frak{e}_{i,i-1}=\frak{d}_i.
    \end{align*}
\end{lemma}
\begin{proof}
    This follows from the following non-zero commutative diagram in $H^0\stM(\PR)$:
    \[\xymigisita{\CC_{\PR}}{\CC_0[0](0)}{\CC_{\PR}[2](1)}.\]
\end{proof}

\begin{proof}[Proof of Lemma~\ref{jan7-4-2}] 
By Lemma~\ref{dec27-2},
the algebra
$\bigoplus_{k,s\in \ZZ}\mathrm{Ext}^k_{\muMG}(\bigoplus_{i=1}^n\CC_{\PR_i},\bigoplus_{i=1}^n\CC_{\PR_i}(s/2))$
is actually degenerate to the diagonal:
\[\bigoplus_{k\in \ZZ}\mathrm{Ext}^k_{\muMG}(\bigoplus_{i=1}^n\CC_{\PR_i},\bigoplus_{i=1}^n\CC_{\PR_i}(k/2))=:\bigoplus_{k\in \ZZ}D^k.\]
Then,
it follows from
Lemma~\ref{dec27-2} and Lemma~\ref{dec27-3} that
$\bigoplus_{k\in \ZZ}D^k$ (the degree of $D^k$ is $k$) is isomorphic to $A_{\Gamma}$ defined in Subsection~\ref{dec27-5};
for a $i$-th vertex $v_i$ in $\Gamma=A_n$,
$e_{v_i}$ corresponds to $\frak{i}_i$ as an element of the degree $0$,
$e_{v_i,v_{j}}$ corresponds to
$\frak{e}_{i,j}$ for $j=i+1$ or $i-1$ as an element of the degree $1$,
and
$w_{v_i}$ corresponds to $\frak{d}_i$ as an element of the degree $2$.
Moreover, since the dga $\bigoplus_{s\in \ZZ}\mathrm{Hom}_{\muMG}(\bigoplus_{i=1}^n\CC_{\PR_i},\bigoplus_{i=1}^n\CC_{\PR_i}(s/2))$ is formal,
this is also isomorphic to $A_{\Gamma}$.
\end{proof}

\subsection{McBreen--Webster's result}
Here we give an account for the closely related construction by McBreen--Webster's result. First, we give a brief explanation of the context.
Their result is related to the case when affine $A_n$-plumbings ($=:\widehat A_n$-plumbings). The resulting plumbing space is an example of multiplicative toric hyperK\"aher (a.k.a. multiplicative hypertoric) varieties.

For such a class of varieties, McBreen--Webster and Gammage--McBreen--Webster~\cite{McBreenWebster, McBreen} provides a construction closely related to our construction in the above. 

Let us fix a toric data $t$. Associated to this data, we first have the additive toric variety $\frakM_\bC$ with the affinization morphism $\frakM_\bC\rightarrow \Spec H^0(\frakM_\bC, \cO)$. We denote the category of coherent sheaves set-theoretically supported on $\pi^{-1}(0)$ by $\mathrm{Coh}(\frakM_\bC)_0$.

Associated to the mirror $t^\vee$ of the toric data $t$, we have the Dolbeault hypertoric variety $\frakD$. In $\cite{McBreenWebster}$, McBreen--Webster constructed a certain category of the deformation quantization modules $DQ$ over $\frakD$. Locally, an object of $DQ$ can be considered as a $\cD$-module. So, we can speak about the Hodge module version of $DQ$. We then denote the Hodge version by $\mu M$. See the references for the details.
\begin{theorem}[\cite{McBreenWebster}]
The following equivalences hold:
\begin{equation}
    \begin{split}
        D^b\mathrm{Coh}(\frakM_\bC)_0&\cong DQ,\\
        D^b\mathrm{Coh}_{\bC^*}(\frakM_\bC)_0&\cong \mu M.
    \end{split}
\end{equation}
\end{theorem}

Let us describe their $\mu M$ in our language
in the case corresponding to $\widehat A_n$-plumbing. Since $C$ is a chain of nodal $\bP^1$, the microsheaf category $\mu \mathrm{sh}_C(X_\Gamma)$ is a gluing up of $\Sh(\bP^1, \{0, \infty\})$.

Instead of $\mu \mathrm{sh}_C(X_\Gamma)$, we consider the unipotent version: We use $\Sh_{\mathrm{unip}}(\bP^1, \{0, \infty\})$ the subcategory consisting of the objects whose monodromy is unipotent as the pieces of $\mu \mathrm{sh}_C(X_\Gamma)$. We denote the resulting glued-up category by $\mu \mathrm{sh}_{C,\mathrm{unip}}(X_\Gamma)$. An object $\cE\in \mu \mathrm{sh}_{C,\mathrm{unip}}(X_\Gamma)$ is said to be the nilpotent order $\leq N$ (or $N$-unipotent, for short) if the monodromy of the restriction to each $W=\bC^*\subset X_\Gamma$ satisfies $(\id-T)^N=0$.

\begin{definition}
    For $i\in \{1,2,...,n\}$ and $N\in \bZ_{\geq 0}$, an object $\cH^{N}_{u,i}$ is an \emph{$N$-unipotent microlocal skyscraper sheaf} if there exists a functor isomorphism $\Hom_{\mu \mathrm{sh}_{C,\mathrm{unip}}(X_\Gamma)}(\cH^{N}_{u,i}, -)\cong \Hom_{\mu \mathrm{sh}(X_\Gamma)}(\bC_{m_i},-)$ on the $N$-unipotent objects in $\mu \mathrm{sh}_{C, \mathrm{unip}}(X_\Gamma)$.
\end{definition}

We will consider $\widehat A_n$ as a quotient of $A_\infty$ (or, in other words, we regard $A_\infty$ as the universal covering of $\widehat A_n$). For this reason, we first study $\Gamma=A_\infty$. We can similarly define the unipotent microsheaves and unipotent microlocal skyscraper sheaves.

We describe the $N$-unipotent microlocal skyscraper sheaf at $\bP^1_i$: 
\begin{lemma}
The following description gives $\cH_{u,i}^N$:
\begin{equation}
    ...,(A_N, B_N),(A_N, B_N),(A_N, A_N),(B_N, A_N), (B_N, A_N),....
\end{equation}
In the quiver description in the sense of Remark~\ref{remark:quiver}:
\begin{equation}
 \cdots\bC[y]/y^{N}\overset{y}{\underset{\id}{\rightleftharpoons}} \bC[y]/y^{N}\overset{y}{\underset{\id}{\rightleftharpoons}}\bC[y]/y^{N}\overset{\id}{\underset{y}{\rightleftharpoons}} \bC[y]/y^{N}\overset{\id}{\underset{y}{\rightleftharpoons}} \bC[y]/y^{N}\cdots
\end{equation}
where the central $\bC[y]/y^N$ is placed on $\bP_i^1$.
\end{lemma}
\begin{proof}
Given $\cE$ an $N$-unipotent object in $\mu \mathrm{sh}_{C, \mathrm{unip}}(X_\Gamma)$. Then it also has a quiver description by Remark~\ref{remark:quiver}. We denote the object given by the above quiver description by $\cH$. Suppose we have a morphism $\cH\rightarrow \cE$. Then it is determined by $\bC[y]/y^N\rightarrow \cE_i$ on $\bP_i^1$. Also, it is given by the image of $1$.

On the other hand, given an element of $\cE_i$, we can define a morphism $\bC[y]/y^N\rightarrow \cE_i$. This gives the desired bijection.
\end{proof}
Now let's consider $\widehat A_n$. There exists a proper push-forward functor $\pi_n\colon \mu \mathrm{sh}_C(X_{A_\infty})\rightarrow \mu \mathrm{sh}_C(X_{\widehat A_n})$ associated to the universal covering $X_{A_\infty}\to X_{\widehat A_n}$. 
\begin{lemma}
    The object $\pi_n(\cH_{u,i}^N)$ is the $N$-unipotent microlocal skyscraper sheaf at $\bP^1_i$
\end{lemma}
\begin{proof}
    For any $\cE\in \mu \mathrm{sh}_C(X_{\widehat A_n})$, we have
    \begin{equation}
    \begin{split}
        \Hom_{\mu \mathrm{sh}_C(X_{\widehat A_n})}(\pi_n(\cH_i^\infty), \cE) &\cong \Hom_{\mu \mathrm{sh}_C(X_{A_\infty})}(\cH_i^\infty, \pi_n^{-1}(\cE))\\
        &\cong \cE_{m_i}.
    \end{split}
    \end{equation}
\end{proof}
One can also see that $\pi_n(\cH^N_{u,i})$ is $\cP^{(n)}$ in \cite{McBreenWebster}. On the other hand, in \cite{gammage2024homologicalmirrorsymmetryhypertoric}, Gammage--McBreen--Webster prove the homological mirror symmetry between the additive toric variety. 

\begin{theorem}[\cite{gammage2024homologicalmirrorsymmetryhypertoric}]
    Associated to the same data $t$, we can associate the multiplicative toric varieties $\frakA$ and $\frakA^\vee$. Then we have
    \begin{equation}
        D^b\mathrm{coh}(\frakA)\cong \mathrm{Perf}\cW(\frakA^\vee).
    \end{equation}
    Moreover, there exists a subvariety $1$ of $\frakA$ such that the completion along $1$ of $D^b\mathrm{coh}(\frakA)$ is equivalent to $D^b\mathrm{coh}(\frakM_\bC)_0$.
\end{theorem}
McBreen~\cite{McBreen} communicates us that the cocore objects of $\frakA^\vee$ coincides with the completion of $\cP^{(n)}$ ($n\rightarrow \infty$) in a certain sense.

\subsection{Koszul duality for the category $\cO$ of $A_n$-plumbing of $T^*\bP^1$}\label{appendix:BLPW}

In this subsection, we discuss a variant of the above argument: As we discussed in \S~\ref{section:BLPW}, we use the relative core $C'$. In this case,
$\mu \mathrm{sh}_{C'}(X_\Gamma)$ is
\[\Sh(\bP^1_1,\{r\})\times_{\Sh(T_r\bP^1_1,0)} \Sh(\bP^1_2,\{l,r\})\times_{\Sh(T_r\bP^1_2,0)} \dots \times_{\Sh(T_r\bP^1_{n-2},0)} \Sh(\bP^1_{n-1},\{l,r\})\times_{\Sh(T_r\bP^1_{n-1},0)} \Sh(\bC^1,\{0\}).\]
For fixed points $m_i\in \bP^1_{i}$ and $m_{n}\in \bC^1$,
we set $\{\mathbf{m}\}=\{m_1,\dots,m_{n}\}$ and we define $\mu \mathrm{sh}_{C'_{\{\mathbf{m}\}}}$ similarly.
An object in $\mu \mathrm{sh}_{C'_{\{\mathbf{m}\}}}(X_\Gamma)$ can be expressed as a tuple:
\[H=((F^1,G^1),(F^2,G^2),\dots,(F^{n-1},G^{n-1}),F^{n}),\]
where
$(F^i,G^i)\in \Sh(\bP^1_i,\{l,r,m\})$ ($1\leq i\leq n-1$) and
$F^{n+1}\in \Sh(\bC,\{0\})$.
We also define $U_i$ and $\mu \mathrm{sh}_{C'}(U_i)$ ($1\leq i\leq n$) as in Subsection~\ref{dec20-1}.
Here, $U_{n}$ is the union of $V_r\subset \bP^1_{n-1}$ and $\bC^1$.
Then, for $H,H'\in \mu \mathrm{sh}_{C'_{\{\mathbf{m}\}}}(X_\Gamma)$ the map (\ref{s11-7-2}) is defined also in this case:
\begin{align}\label{feb25-1}
    \xymatrix{
\bigoplus_{i=1}^{n}H^k\mathrm{Hom}_{U_i}(H,H')
    \ar[r]&
    \bigoplus_{i=1}^{n-1}H^k\mathrm{Hom}_{U_{i}\cap U_{i+1}}(H,H')\\
(b_1,(b_{2,l},b_{2,r}),\dots,(b_{n-1,l},b_{n-1,r}),b_{n})\ar@{|->}[r]\ar@{}[u]|-{\rotatebox{90}{$\in$}}& 
(b_1|_{W}-b_{2,l}|_{W},b_{2,r}|_{W}-b_{3,r}|_{W},\dots, b_{n-1,r}|_{W}-b_{n}|_{W})\ar@{}[u]|-{\rotatebox{90}{$\in$}}.
}        
\end{align}

We follow the same route as in the previous cases.

\begin{definition}
\begin{enumerate}
    \item If $n\geq 2$ and $2\leq j\leq  n-1$,
    we define $\und{\cH_j}\in \mu \mathrm{sh}_{C'_{\{\mathbf{m}\}}}(X_\Gamma)$ (resp.${\cH_j}\in \mu \mathrm{sh}_{C'}(X_\Gamma)$)
    as
    \begin{align*}
        ((P_1,Q_2),(P_2,Q_3),\dots,(P_{j-1},Q_{j}),
        (\und{P_j},\und{A_j}),(B_{j},A_{j}),\dots,(B_{j},A_{j}),B_j),\\
        (\mbox{resp.\quad }        ((P_1,Q_2),(P_2,Q_3),\dots,(P_{j-1},Q_{j}),
        ({P_j},{A_j}),(B_{j},A_{j}),\dots,(B_{j},A_{j}),B_j)).
    \end{align*}
\item We define $\und{\cH_1}\in \mu \mathrm{sh}_{C'_{\{\mathbf{m}\}}}(X_\Gamma)$ (resp. ${\cH_1}\in \mu \mathrm{sh}_{C'}(X_\Gamma)$) as
\begin{align*}
((\und{P_1},\und{A_1}),(B_1,A_1),\dots,(B_1,A_1),B_1),\\
    (\mbox{resp. \quad} (({P_1},{A_1}),(B_1,A_1),\dots,(B_1,A_1),B_1)).
\end{align*}

\item We define $\und{\cH_{n}}\in \mu \mathrm{sh}_{C'_{\{\mathbf{m}\}}}(X_\Gamma)$ (resp. $\cH_{n}\in \mu \mathrm{sh}_{C'}(X_\Gamma)$)
as
 \begin{align*}
        ((P_1,Q_2),(P_2,Q_3),\dots,(P_{n-1},Q_{n}),\und{P_{n}}),\\
        (\mbox{resp.\quad }        ((P_1,Q_2),(P_2,Q_3),\dots,(P_{n-1},Q_{n}),P_{n})).
    \end{align*}
\end{enumerate}
\end{definition}

We will omit the subscript $\{\mathbf{m}\}$ below.

\begin{theorem}
    For any $H\in \mu \mathrm{sh}_{C'}(X_\Gamma)$ and $1\leq j\leq n$,
    we have the vanishing:
    \[H^k\HOM_{\mu \mathrm{sh}_{C'}(X_\Gamma)}(\und{\cH_j},H)=0\quad (k\in \ZZ).\]
\end{theorem}

\begin{proof}
    It suffices to show that the morphism (\ref{feb25-1}) is bijective.
    We can apply exactly the same proof as in the $A_n$-plumbing of $T^*\bP^1$.
\end{proof}

This means that $\cH_j$ is a microlocal skyscraper sheaf.
We can consider $\mu \mathrm{M}_{C'}$ also in this case as in Subsection~\ref{dec20-5}.
We gather some facts on it, which can be shown as in the previous case.

\begin{proposition}
    \begin{enumerate}
        \item 
        For $2\leq j\leq n-1$,
the tuple 
\begin{align*}
    ((P_1((j-1)/2),Q_2((j-1)/2)),(P_2,Q_3),\dots,(P_{j-1}(1/2),Q_{j}(1/2)),
        ({P_j},{A_j}),\\(B_{j}(1/2),A_{j}(1/2)),\dots,(B_{j}((n-1-j)/2),A_{j}((n-1-j)/2)),B_j((n-j)/2))
\end{align*}
defines an object in $\mu \mathrm{M}_{C'}(X_\Gamma)$, whose underlying object is $\cH_j$.   
We denote it by $\cH^H_j$.
        We can similarly enhance $\cH_1$ and $\cH_{n}$ to objects $\cH^H_1$ and  
 $\cH^H_{n}$ in $\mu \mathrm{M}_{C'}(X_\Gamma)$.
        
        \item We set $W_i=U_i\cap U_{i+1}$ and $W_{n}:=\CS\subset \CC$. 
        For $1\leq i,j\leq n$,
        in $\mathrm{M}(W_i)$,
we have
\[\cH_j^H|_{W_i}=
\calL_{\min{(i,j)}}(|j-i|/2)
=
\left\{
\begin{array}{cc}
     \calL_i((j-i)/2)&  1\leq i<j\\
     \calL_{j}((i-j)/2)& j\leq i\leq n.
\end{array}
\right.
\]

\item For $1\leq i,j\leq n$ and $k,s\in \ZZ$, we have
\[H^0\HOM_{\mu \mathrm{M}_{C'}(X_\Gamma)}(\cH_j^H,\cH_i^H(s/2)[k])
\simeq \left\{
\begin{array}{cc}
     \CC&\left(\substack{
          k=0  \\
           s/2+|j-i|/2\in \ZZ\\
           0\leq -(s/2+|j-i|/2)\leq {\min{(i,j)}}-1
     }\right)  \\
     0& \mbox{otherwise}
\end{array}
\right.
\]

    \end{enumerate}

\end{proposition}

By this proposition, 
for $1\leq i,j\leq n$, we can define (up to constant)
a non-zero morphism in
$H^0\HOM_{\mu \mathrm{M}_{C'}(X_\Gamma)}(\cH_j^H,\cH_i^H(s/2))$
for $s=-|j-i|-2\ell$ ($0\leq \ell\leq \min{(j,i)}-1$).
    We will see that this number coincides with the path
starting from $i$-th vertex and ending at $j$-th vertex in the basis of an algebra $L_\Gamma$ defined below (see Remark~\ref{may5-1}).

We set $\cH^H:=\bigoplus_{j}\cH^H_j$ and $\cH:=\bigoplus_j \cH_j $.
Then, by the saturatedness, we have
\[\HOM_{\mu \mathrm{M}_{C'}(X_\Gamma)}(\cH^H,\bigoplus_{s\in \ZZ}\cH^H(s/2))\simeq \mathrm{End}_{\mu \mathrm{sh}_{C'}(X_\Gamma)}(\cH).\]
The left hand side has a doubly graded decomposition:
\[B=\bigoplus_{a,b\in \ZZ} B^{a,b}=\bigoplus_{k\in \ZZ}\HOM^k_{\mu \mathrm{M}_{C'}(X_\Gamma)}(\cH^H,\bigoplus_{s\in \ZZ}\cH^H(s/2)),\]
where $a=k-s$ and $b=-s$.
For $\cE\in \mu \mathrm{M}_{C'}(X_\Gamma)$,
the Hom-complex $\bigoplus_{s\in \ZZ}\mathrm{Hom}_{\mu \mathrm{M}_{C'}(X_\Gamma)}(\cH^H,\cE(s/2))$
is an Adams-graded $B$-module,
where $\mathrm{Hom}^k_{\mu \mathrm{M}_{C'}(X_\Gamma)}(\cH^H,\cE(s/2))$ is of degree $(k-s,-s)$.
Since the object $\cH$ is a generator of the category $\mu \mathrm{sh}_{C'}(X_\Gamma)$,
we obtain the following equivalence as in Lemma~\ref{jan9-1}:
\begin{align}
 \mu \mathrm{M}_{C'}(X_\Gamma)\simeq  \mathrm{Mod}(B)\quad  ( \cE \mapsto  \mathrm{Hom}_{\mu \mathrm{M}_{C'}(X_\Gamma)}(\cH^H,\bigoplus_{s\in \ZZ}\cE(s/2))). 
\end{align}
Remark that the half Tate twist corresponds to the shift $\la 1\ra=(-1)[-1]$.
For $1\leq j\leq n$, we define $\cI_j^H$ as the ``constant sheaf on $\bP_j^1$":
\[\cI_j^H:=(0\dots, 0,\CC_r[-1](-1/2),\CC_{\bP^1_j},\CC_l[-1](-1/2),0,\dots,0)\quad(1\leq j\leq n-1),\]
and
\[\cI_{n}^H:=(0,\dots,\CC_r[-1](-1/2),\CC_{\CC}).\]
Then, we can see that the augmentation module $\mathbf{k}$ on the right hand side corresponds to
\[\cI^H:=\bigoplus_{i=1}^{n}\cI_i^H.\]
In this way, we get isomorphisms similar to (\ref{jan7-3}) and (\ref{mar19-1}):
\begin{align}
\bigoplus_{s\in \ZZ}\mathrm{Hom}_{\mu \mathrm{M}_{C'}(X_\Gamma)}(\cI^H,\cI^H(s/2))
\simeq 
\bigoplus_{s\in \ZZ}\mathrm{Hom}_{B}(\mathbf{k},\mathbf{k}\la s\ra).
    \\\label{mar18-1}
\mathrm{Ext}_{\mu \mathrm{M}_{C'}(X_\Gamma)}^k(\cI^H,\cI^H(s/2))
\simeq 
\mathrm{Ext}_{B}^{k-s}(\mathbf{k},\mathbf{k}(-s/2)).
\end{align}

Similarly to Lemma~\ref{dec27-2} and Lemma~\ref{dec27-3}, we obtain the following.
\begin{lemma}\label{apr9-2}
\begin{enumerate}
    \item If $i=j$, we have
    \[H^0\mathrm{Hom}_{\mu \mathrm{M}_{C'}(X_\Gamma)}(\cI_j^H,\cI_j^H(s/2)[k])\simeq \left\{\begin{array}{cl}
          \CC  & (1\leq i\leq n,\quad k=s=0)\\
          \CC & (1\leq i\leq n-1,\quad k=s=2)\\
          0& (\mbox{otherwise})
    \end{array}
        \right.\]
        We will denote by $\mathfrak{i}_i$ (resp. $\mathfrak{d}_i$) the morphism $\cI_i^H\to \cI_i^H(0)[0]$ (resp. $\cI_i^H\to \cI_i^H(1)[2]$).

    \item If $j=i+1$ or $j=i-1$, we have
    \[H^0\mathrm{Hom}_{\mu \mathrm{M}_{C'}(X_\Gamma)}(\cI_j^H,\cI_i^H(s/2)[k])\simeq \left\{\begin{array}{cl}
          \CC&  (1\leq i\leq n,\quad k=s=1)\\
          0& (\mbox{otherwise})
    \end{array}
    \right.\]
        We will denote by $\mathfrak{e}_{j,i}$ the morphism $\cI_j^H\to \cI_i^H(0)[0]$ for $i=j+1$ or $i=j-1$.
\end{enumerate}
\end{lemma}

By this lemma,
the algebra $\bigoplus_{s,k\in \ZZ}\mathrm{Ext}^k_{\mu \mathrm{M}_{C'}(X_\Gamma)}(\cI^H,\cI^H(s/2))$ is actually degenerate to the diagonal:
$\bigoplus_{k\in \ZZ}\mathrm{Ext}^k_{\mu \mathrm{M}_{C'}(X_\Gamma)}(\cI^H,\cI^H(k/2))$.
Together with (\ref{mar18-1}), we obtain the vanishing:
\[\mathrm{Ext}^k_B(\mathbf{k},\mathbf{k}(-s/2))=0\quad (k\neq 0).\]
Hence, we conclude the following.

\begin{corollary}
    $B$ is an Adams Koszul dga in the sense of \cite{HeWu}.
\end{corollary}

Next, we will describe the graded endomorphism algebras of $\scI^H$ and $\scH^H$.
We will introduce algebras which play the same role as $\cG_\Gamma$ and $A_\Gamma$ in this situation.

From here to Lemma~\ref{lem:endofknown}, we collect some known/folklore results. Namely, the properties of the algebras \( L_\Gamma \) and \( M_\Gamma \) including their Koszulity. One can read off these results from e.g. \cite{BLPWhypertoric}. For the reader's convenience, 
we will briefly recall their definitions and properties, and also give a proof of their Koszulity.

Let us consider the quiver $\Gamma$ again:
\begin{align}\label{may1-3}
    \xymatrix{
\bullet_1\ar@<0.5ex>[r]^{f_1}&\ar@<0.5ex>[l]^{g_1}\bullet_2\ar@<0.5ex>[r]^{f_2}&\ar@<0.5ex>[l]^{g_2}\bullet_{3}\ar@{}[r]|{\dots} &\bullet_{i}\ar@<0.5ex>[r]^{f_i}&\ar@<0.5ex>[l]^{g_i} \bullet_{i+1}\ar@{}[r]|{\cdots} &\bullet_{n-1}\ar@<0.5ex>[r]^{f_{n-1}}&\ar@<0.5ex>[l]^{g_{n-1}}\bullet_{n}.
}
\end{align}
We write $e_i$ the $i$-th vertex.

\begin{definition}\label{mar24-1}
We define the augmented $\mathbf{k}$-algebra $L_\Gamma$ (resp. $M_\Gamma$) as the quotient of the quiver path algebra of $\Gamma$ by the ideal generated by
  \[ g_{1}f_{1} ,\quad f_i g_i-g_{i+1} f_{i+1}\quad (1\leq i\leq n-2)
     \]   
        \[ (\mbox{resp.}  f_{n-1} g_{n-1},\quad f_i g_i-g_{i+1} f_{i+1},\quad
      f_{i+1} f_i,\quad  
      g_{i} g_{i+1}\ (1\leq i\leq n-2)). \]
    
     We define a grading on $L_\Gamma$ and $M_\Gamma$ so that 
     the degree of 
     the points $e_i$ is $0$ and 
     $f_i$ and $g_i$ is $1$.
\end{definition}

We write $L_\Gamma^i$ (resp. $M_\Gamma^i$) the $i$-th part of $L_\Gamma$ (resp. $M_\Gamma$).
Then, it is clear that we have
\[L_\Gamma=\bigoplus_{i=0}^{2(n-1)}L_\Gamma^i,\quad  M_\Gamma=\bigoplus_{i=0}^2M_\Gamma^i,\]
and we have
\[L_\Gamma^0\simeq M_\Gamma^0\simeq \mathbf{k}.\]

The following lemma can be shown by the direct computation.

\begin{lemma}
    \begin{enumerate}
        \item We define elements in $L_\Gamma$ as
\begin{align*}
    a_{s,t,u}&:=(f_{t-1} \cdots  f_{u}) (g_{u} \cdots  g_{t-1})(g_t\cdots g_{s-1})\\
    b_{s,t,u}&:=
        (f_{s-1} \cdots  f_{t})(f_{t-1} \cdots  f_{u}) (g_{u} \cdots  g_{t-1})
\end{align*}
for $u\leq t\leq s$
(modifying the definition somewhat if $s=t$ or $t=u$)
, which are expressed as the diagrams:
\[\xymatrix@C=50pt{\ar@/^10pt/[r]^-{f_{t-1} \cdots  f_{u}}\bullet_{u}&
\bullet_t\ar@/^10pt/[l]^{g_{u} \cdots  g_{t-1}}&
\ar[l]_-{g_t\cdots g_{s-1}}
\bullet_s,
}
\quad 
\xymatrix@C=50pt{\ar@/^10pt/[r]^-{f_{t-1} \cdots  f_{u}}\bullet_{u}&
\ar[r]^{f_{s-1} \cdots  f_{t}}\bullet_t\ar@/^10pt/[l]^{g_{u} \cdots  g_{t-1}}&
\bullet_s.
}
\]

Then, $\{a_{s,t,u},b_{s,t,u}\}$ forms a basis of $L_\Gamma$.

\item For $0\leq \ell\leq n$, we set $p_\ell:=n-\ell$,
which is the number of paths of length $\ell$ in $\Gamma$ with no loops and ignoring direction.
Then, 
we have
\[\dim{L_\Gamma^i}=\left\{
\begin{array}{cc}
    p_{\frac{i}{2}}+\displaystyle\sum_{\ell=0}^{\min{(i,n)}}2p_{\frac{i}{2}+\ell} &i\colon \mbox{odd}  \\
     \displaystyle\sum_{\ell=0}^{\min{(i,n)}}2p_{\frac{i+1}{2}+\ell}& i\colon \mbox{even}.
\end{array}
\right.\]

\item We define 
$L_\Gamma'$ (resp. $L_\Gamma''$, $L_\Gamma'''$) as the $L_\Gamma$-submodule of $L_\Gamma$ generated by $\{e_1,\dots,e_{n-1}\}$ (resp. $\{e_2,\dots,e_{n}\}, \{e_n\}$).
For $1\leq i\leq 2n$,
we set
\[q_i:=\left\{
\begin{array}{cc}
   \min{(\frac{i}{2}+1,n+1-\frac{i}{2})}  &  i:\mbox{even}\\
     \min{(\frac{i+1}{2}, n+1-\frac{i-1}{2})}& i:\mbox{odd}.
\end{array}
\right.\]
Then, we have
\begin{align*}
    \dim{{L_\Gamma'''}^i}&=q_i\\
    \dim{{L_\Gamma'}^i}&=\dim{L_\Gamma^i}-q_i,\\
    \dim{{L_\Gamma''}^i}&=\left\{\begin{array}{ll}
         \dim{L_\Gamma^i}-1&i\leq n-1  \\
         \dim{L_\Gamma^i}& i> n-1.
    \end{array}\right.
\end{align*}
\item For $0\leq i\leq 2(n-1)$, we have
\begin{align}
    \label{mar21-1}
    \dim{{L_\Gamma'}^i}+\dim{{L_\Gamma''}^i}- \dim{L_\Gamma^{i+1}}=\dim{{L_{\Gamma}'}^{i-1}}.
\end{align}
\end{enumerate}
\end{lemma}

Remark that both $L_\Gamma'$ and $L_\Gamma''$ are projective $L_\Gamma$-module since they are direct summand of $L_\Gamma$.
We write $\mathbf{k}=L_\Gamma^0$, which is a $L_\Gamma$-module.
\begin{lemma}
    We have the following projective resolution of $\mathbf{k}$:
    \[0\to L_\Gamma'(-2)\to L_\Gamma'(-1)\oplus L_\Gamma''(-1)\to L_\Gamma\to \mathbf{k}\to 0,\]
    where the symbol $(s)$ stands for the shift of the grading.
\end{lemma}

\begin{proof}
We define the morphism
    $L_\Gamma'(-1)\to L_\Gamma$ 
    (resp. $L_\Gamma''(-1)\to L_\Gamma$)
so that it sends $e_j$ to $g_{j}$
(resp. $e_j$ to $f_{j-1}$).
Then, this yields a surjective morphism for $i\geq 0$:
\[{L_\Gamma'}^i\oplus {L_\Gamma''}^i\twoheadrightarrow L_\Gamma^{i+1}.\]
We write the kernel of $L_\Gamma'(-1)\oplus L_\Gamma''(-1)\to L_\Gamma$ 
as $K(=\bigoplus_iK^i)$.
By the surjectivity above,
for $i\geq 1$, we obtain
\begin{align}\label{mar21-2}
    \dim K^i\leq \dim{{L_\Gamma'}^{i-1}}+\dim{{L_\Gamma''}^{i-1}}-\dim{L_\Gamma^{i}}=\dim{{L_\Gamma'}^{i-2}},
\end{align}
where the last equality is (\ref{mar21-1}).
We set an element $v_j\in L_\Gamma'\oplus L_\Gamma''$ ($1\leq j\leq n-1$) of degree $1$ as
\begin{align*}
    v_j=\left\{\begin{array}{lc}
         (0,g_1)& j=1\\
         (f_{j-1},g_{j})& 2\leq j\leq n-1
    \end{array}\right.
\end{align*}
Then, we define $L_{\Gamma}'(-2)\to L_\Gamma'(-1)\oplus L_\Gamma''(-1)$
so that it sends $e_j$ to $v_j$.
One can check that this yields an injective morphism
\[L_{\Gamma}'(-2)\hookrightarrow K.\]
Hence, for $i\geq 1$ we have
\[\dim{{L_{\Gamma}'}^{i-2}}\leq \dim{K^i},\]
and together with (\ref{mar21-2}), this leads to
\[\dim{K^i}=\dim{{L_{\Gamma}'}^{i-2}}.\]
Therefore, the sequence
\[0\to L_\Gamma'(-2)\to L_\Gamma'(-1)\oplus L_\Gamma''(-1)\to L_\Gamma\to \mathbf{k}\to 0\]
is exact.
\end{proof}

\begin{corollary}
    The algebras $L_\Gamma$ is a Koszul algebra (in the classical sense).
    \end{corollary}

By using the resolution of $\mathbf{k}$ above,
we can compute the graded algebra
$\mathrm{Ext}^\bullet_{L_F}(\mathbf{k},\mathbf{k})(=\mathrm{Ext}^0_{L_F}(\mathbf{k},\mathbf{k})\oplus \mathrm{Ext}^1_{L_F}(\mathbf{k},\mathbf{k}(-1))\oplus \mathrm{Ext}^2_{L_F}(\mathbf{k},\mathbf{k}(-2)))$
explicitly including the Yoneda algebra structure,
and indeed this is isomorphic to $M_\Gamma$ as a graded algebra.
Hence, by the definition, we conclude the following proposition.

\begin{proposition}\label{mar21-3}
    The algebras $L_\Gamma$ and $M_\Gamma$ are Koszul dual (in the classical sense) to each other.
\end{proposition}

Remark that if we set the bidegree of $f_i$ and $g_i$ in Definition~\ref{mar24-1} as $(1,1)$,
then $L_\Gamma$ and $M_\Gamma$ are also Koszul dual in the Adams Koszul sense.

Let us now return to $\mathcal{I}^H$ and $\mathcal{H}^H$.
The following can be easily checked.
\begin{lemma}\label{lem:endofknown}
For the morphisms $\frak{e}_{j,i}\colon \scI_j^H\to \scI_{i}^H(0)[0]$ and $\frak{d}_{i}\colon \scI_i^H\to \scI_i^H(1)[2]$ defined in Lemma~\ref{apr9-2},
we have
    \begin{align*}
        \frak{e}_{i+1,i}\circ \frak{e}_{i,i+1}=\frak{e}_{i-1,i}\circ \frak{e}_{i,i-1}&=\frak{d}_i\quad (1\leq i\leq n-1)\\
        \frak{e}_{n,n-1}\circ \frak{e}_{n-1,n}&=0.
    \end{align*}
\end{lemma}

Remark that $M_\Gamma$ is formal.
Then, in the same way as Subsection~\ref{mar24-10},
we have the following.
\begin{corollary}    \label{apr29-1}
\begin{enumerate}
    \item We have an isomorphism:
    \begin{align}
        \label{may1-1}
        \bigoplus_{s\in \ZZ}\mathrm{Hom}_{\mu \mathrm{M}_{C'}(X_\Gamma)}(\cI^H,\cI^H(s/2))\simeq M_\Gamma.
    \end{align}
    
    \item We have a quasi-isomorphism between Adams-graded dgas:
    \[\bigoplus_{s\in \ZZ}\HOM_{L_\Gamma}(\mathbf{k},\mathbf{k}(s))\simeq B.\]
\end{enumerate}
\end{corollary}
Applying Proposition~\ref{mar21-3}, we have the following.

\begin{corollary}\label{apr29-2}
        We have a quasi-isomorphism between Adams-graded dgas:
\begin{align}\label{may1-2}
     \bigoplus_{s\in \ZZ}\HOM_{\mu \mathrm{M}_{C'}(X_\Gamma)}(\cH^H,\cH^H(s/2))\simeq L_\Gamma.
\end{align}   
\end{corollary}

\begin{remark}\label{may5-1}
    In the isomorphisms in (\ref{may1-1}) and (\ref{may1-2}),
    the index ``$i$" of $e_i$ in the right hand sides
    corresponds to the index ``$i$" in
    $\cH^H(=\bigoplus_{i=1}^{n}\cH_i^H)$ and
    $\cI^H(=\bigoplus_{i=1}^{n}\cI_i^H)$ in the left hand sides. 
    Therefore, for example,
    $\bigoplus_{s\in \ZZ}\mathrm{Hom}_{\mu \mathrm{M}_{C'}(X_\Gamma)}(\cH^H_i,\cH^H_j(s/2))$ is isomorphic to
    the subspace of $L_\Gamma$ spanned by the paths starting at vertex $i$ and ending at vertex $j$ in (\ref{may1-3}). 
\end{remark}

In this way, we have clarified the graded algebra structures of $\mathcal{I}$ and $\mathcal{H}$, 
as well as the Koszul duality between them, 
in terms of the weights of mixed Hodge modules. 

\section{Appendix: a structure theorem for $\Hzsh{\CC,0}$} \label{dec3-4}

In this appendix, 
we summarize some basic properties of objects in $\Sh(\CC,0)$ and give a structure theorem for them.

The first lemma is basic.
\begin{lemma}\label{sep25-1}
Let $X$ be a contractible complex manifold. 
    An object $F\in \Sh(X)$ with $\mathrm{SS}(F)\subset T^*_{X}X$ can be expressed as the pullback of an object in $\mathrm{Vect}(\CC)$ by the morphism $X\to \mathrm{pt}$, where $\mathrm{Vect}(\CC)$ is the derived category of vector spaces.
In particular,
we have a non-canonical isomorphism:
\[F\simeq \bigoplus_{k\in \ZZ}H^k(F)[-k].\]
\end{lemma}

The following lemma is also well-known, but we give a sketch of the proof.

\begin{lemma}
     An object $F\in \Sh(\CS(=\CC\setminus \{0\}))$ with $\mathrm{SS}(F)\subset {T^*_{\CS}\CS}$ is isomorphic to a direct sum of shifted $\CC$-locally constant sheaves (infinite dimensional in general), i.e., we have
     \begin{align}
         \label{sep25-2}
         F\simeq \bigoplus_{k\in \ZZ}H^k(F)[-k]
     \end{align}
     and $H^k(F)$ is a locally constant sheaf.
\end{lemma}

\begin{proof}
    If $F$ is bounded and a stalk of each cohomology of $F$ is of finite rank,
    the assertion follows from the basis fact:
    for two local systems $\calL$, $\calL'$ on $\CS$ and a morphism $\calL\to \calL'$, the cone of it in $H^0\Sh(\CS)$ is a direct sum of shifted local systems.
    
In the general case, we consider a distinguished triangle:
\[F\to \RG{U_1}F\oplus \RG{U_2}F \to \RG{U_1\cap U_2}F\to F[1],\]
    where $U_1:=\{z\in \CS\ |\ \arg{z}\in (-3\pi/4,3\pi/4)\}$ and
    $U_2:=\{z\in \CS\ |\ \arg{z}\in (\pi/4,7\pi/4)\}$.
    If we set $U_3:=\{z\in \CS\ |\ \arg{z}\in (-\pi/4,\pi/4)\}$ and
    $U_4:=\{z\in \CS\ |\ \arg{z}\in (3\pi/4,5\pi/4)\}$,
    then we have a decomposition
    $U_1\cap U_2=U_3\cup U_4$.
    By Lemma~\ref{sep25-1},
    $F|_{U_i}$ ($i=1,2,3,4$)
are the direct sum of shifted constant sheaves.
Under the expression by such direct sums,
the morphism 
$\RG{U_1}F\oplus \RG{U_2}F \to \RG{U_3}F\oplus \RG{U_4}F(=\RG{U_1\cap U_2}F)$
is also a direct sum of morphims between shifted sheaves.
Hence,
it suffices to observe the morphism for each $k\in \ZZ$
\begin{align}\label{sep25-3}
    H^k\RG{U_1}F\oplus H^k\RG{U_2}F \to H^k\RG{U_3}F\oplus H^k\RG{U_4}F.
\end{align}
Moreover, 
by taking the basis suitably,
we may assume 
the submorphisms $H^k\RG{U_i}F\to H^k\RG{U_3}F$ ($i=1,2$) and 
$H^k\RG{U_1}F\to H^k\RG{U_4}F$
of the above morphism
are the direct sums of the (shifted) morphisms induced by identity maps on a vector space $V$ corresponding to $F|_{U_1}$.
The only non-trivial part is 
$H^k\RG{U_2}F\to H^k\RG{U_4}F$,
induced by some automorphism $\rho$ on $V$.
Let $\calL$ be the locally constant sheaf corresponding to the vector space $V$ with the automorphism $\rho$.
Then, the morphism
\[H^k\RG{U_1}\calL[-k]\oplus H^k\RG{U_2}\calL[-k] \to H^k\RG{U_3}\calL[-k]\oplus H^k\RG{U_4}\calL[-k]\]
is isomorphic to the morphism (\ref{sep25-3}).
The assertion follows from the uniqueness of (the isomorphism class of) the cone in a triangulated category.
\end{proof}

Let $t$ be a coordinate of $\CC$.

\begin{corollary}\label{sep30-1}
      For an object $F\in \Sh(\CS)$ with $\mathrm{SS}(F)\subset {T^*_{\CS}\CS}$,
assume that 
$F$ is unipotent, i.e.,
there exists $\ell\in \ZZ_\geq 1$ such that 
$(T-1)^\ell=0$ for the monodromy automorphism $T$ of $\psi_tF$.
Then, $F$ is isomorphic to a direct sum of (possibly infinitely many) shifted local systems $\calL_{s}[k]$ for some $k\in \ZZ$ and $1\leq s\leq \ell$,
where $\calL_{s}$ is the $\CC$-local system whose monodromy matrix is the unipotent Jordan block of size $s$.
\end{corollary}

\begin{proof}
    This follows from the fact that a vector space $V$ with unipotent and finite automorphism $\rho$ is isomorphic to a direct sum of finite dimensional vector spaces with unipotent automorphisms.
\end{proof}

The next lemma simply follows from the fundamental theorem on homomorphisms.

\begin{lemma}\label{oct28-1}
      Let $V$, $W$ be vector spaces (possibly infinite dimensional), $N$ the nilpotent operator on $W$ and $f$ a linear map $V\to W$ with the following properties:
      \begin{itemize}
          \item there exists $\ell\in \ZZ_{\geq 1}$ such that $N^{\ell}=0$,
          \item $N\circ f=0$.
      \end{itemize}
Then, there exist linearly independent vectors $\{v_i\}_{i\in I\sqcup J}$ of $V$ and $\{w_j\}_{j\in I\sqcup K}$ of $W$, where $I,J,K$ are index sets, and integers $\{k_j\}_{j\in I \sqcup K}$ 
such that we have the following:
\begin{enumerate}
\item $\{v_i\}_{i\in I\sqcup J}$ is a basis of $V$,
    \item $N^{k_j}w_j\neq 0$ and $N^{k_j+1}w_j= 0$ for $j\in I\sqcup K$,
\item $\{w_j,Nw_{j}\dots N^{k_j}w_j\ |\ j\in I\sqcup K\}$ is a basis of $W$,
\item $f(v_i)=\left\{\begin{array}{cc}
     N^{k_{i}}w_i& i\in I \\
     0& i\in J.
\end{array}\right.$
\end{enumerate}
\end{lemma}

Let us consider an object $F\in \sh{\CC,0}$ whose restriction $F|_{\CS}$ to $\CS$ is unipotent.
Then, we have a distinguished triangle in $H^0\sh{\CC,0}$:
\[F_{0}[-1]\to F_{\CS}\to F\to F_0.\]
In other words, we can regard $F$ as a cone of $F_{0}[-1]\to F_{\CS}$.
The morphism $F_{0}[-1]\to F_{\CS}$ is described concretely as follows.
Let $\{V_{i}\}_{i\in \ZZ}$ be a family of vector spaces, 
which are regarded as the skyscraper sheaves on $\bC$ supported at the origin $0$, 
$\ell\in \ZZ_{\geq 1}$ be a positive integer
and $\{L_i\}_{i\in \ZZ}$ be a family of direct sums $L_i=\bigoplus_{j\in J_i}A_{s_j}$ of
$A_{s_j}$ for $s_{j}\leq \ell$ (for the definition of $A_{s_j}$, see Lemma~\ref{oct31-1}).
    Moreover, let $f_i$ (resp. $g_i$) be a morphism $f_i\colon V_i[i]\to L_{i+1}[i+1]$ (resp. $g_i\colon V_i[i]\to L_{i+2}[i+2]$) in $H^0\Sh({\CC})$.
We put 
\begin{align}\label{oct25-1}
    h:=\bigoplus_{i\in \ZZ}(f_i\oplus g_i)\colon \bigoplus_{i\in \ZZ}V_i[i]\to \bigoplus_{j\in \ZZ}L_j[j].
\end{align}
Remark that since $\CC_0$ is a compact object in $\sh{\CC,0}$,
we have
\begin{align*}
\HOM_{H^0\sh{\CC}}(\bigoplus_{i\in \ZZ}V_i[i],\bigoplus_{j\in \ZZ}L_j[j])
\simeq \prod_{i\in \ZZ}\bigoplus_{j\in \ZZ}\HOM_{H^0\sh{\CC}}(V_i[i],L_j[j]).
\end{align*}
On the other hand, there is a non-zero morphism $\CC_0\to A_s[i]$ in $H^0\Sh(\CC)$ if and only if $i=1,2$.
Therefore, by Corollary~\ref{sep30-1}, 
the above morphism $F_{0}[-1]\to F_{\CS}$ can be expressed as (\ref{oct25-1}) for some $\{V_i\}$, $\{L_i\}$, $f_i$ and $g_i$.
So we discuss the cone of the morphism $h$ (\ref{oct25-1}) to classify $F$.

\begin{lemma}  \label{oct28-3}
If $f_i=0$ ($i\in \ZZ$),
the cone of $h$ in $H^0\Sh(\CC)$ is isomorphic to a direct sum of several (possibly infinitely many) $A_{s}$ or $Q_{s+1}$ ($1\leq s\leq \ell$) or $\CC_0$ with some shifts.    
\end{lemma}
\begin{proof}
It is enough to see what the cone of $g_i\colon V_i\to L_{i+2}[i+2]$ is.
By Lemma~\ref{oct28-1},
$g_i$ is decomposed into several
$\CC_0\to A_{s}[2]$ ($s\in \ZZ_{\geq 1}$) or the zero morphism from $\CC_0$.
Then, the assertion follows from the distinguished triangle
   \[\CC_0\to A_{s}[2]\to Q_{s+1}[2]\to \CC_0[1].\] 
\end{proof}

In general, for a vector space $V$ and a direct sum $\bigoplus_{k\in K}M^k[1]$,
where $M^k$ is $A_s$ or $Q_{s+1}$ ($1\leq s\leq \ell$) or $\CC_0$,
let $f$ be a morphism $f\colon V\to \bigoplus_{{k\in K}}M^k[1]$,
where we regard $V$ as a skyscraper in $\Sh(\CC)$.
Remark that the morphism $f$ corresponds to a linear map $V\to \bigoplus_{\substack{k\in K\\ M^k\neq \CC_0}}\CC$.
Then, we write $\Ker{f}\subset V$ (resp. $\Coim{f}$) as the kernel (resp. coimage) of this linear morphism. 
For our $V_i$ in (\ref{oct25-1}),
we fix a decomposition:
\[V_i\simeq \Ker{f_i}\oplus \Coim{f_i}.\]

\begin{lemma}\label{oct28-4}
There is a direct sum $M_j$ ($j\in \ZZ$) of several $A_{s}$ or $Q_{s+1}$ ($1\leq s\leq \ell$) or $\CC_0$ and morphisms $f_i'\colon \Coim{f_i}[i]\to M_{j+1}[j+1]$ with $\Ker{f_i'}=0$ and $g_j'\colon \Coim{f_i}[i]\to M_{j+2}[j+2]$ such that
the cone of $h$ in $H^0\Sh(\CC)$ is isomorphic to the cone of the morphism
   \begin{align}\label{oct28-5}
     h':=\bigoplus_{i\in \ZZ}(f_i'\oplus g_i')\colon   \bigoplus_{i\in \ZZ}\Coim{f_i}[i]\to \bigoplus_{j\in \ZZ}M_j[j].
   \end{align}

\end{lemma}

\begin{proof}
Consider a commutative diagram:
\[\xymatrix{{\bigoplus_{i\in \ZZ}\Ker{f_i}[i]}\ar[dr]\ar[r]&{\bigoplus_{i\in \ZZ}\Ker{f_i}[i]\oplus \Coim{f_i}[i]}\ar[d]^{h}\\
&{\bigoplus_{j\in \ZZ}L_j[j]}.}\]
The cone $\bigoplus_{i\in \ZZ}\Ker{f_i}[i]\to \bigoplus_{j\in \ZZ}L_j[j]$
    is a direct sum of several (infinitely many) $A_{s}$ or $Q_{s+1}$ ($1\leq s\leq \ell $) or $\CC_0$ with some shifts by Lemma~\ref{oct28-3}.
    Then, the assertion follows from the octahedron axiom for the above commutative diagram.    
\end{proof}

Remark that in particular for any $\CC_0$-component of $M_j$ the composition $(\bigoplus_{j\in \ZZ}M_j[j]\to \CC_0[j_0])\circ h'$ is zero.

\begin{lemma}
    Let $f$ (resp. $g$) be a morphism $\CC_0\to \bigoplus_{k\in K_2}A_{s_{1}(k)}[1]\oplus \bigoplus_{k\in K_2}Q_{s_{2}(k)}[1]$ (resp. $\CC_0\to \bigoplus_{k\in K_3}A_{s_{3}({k})}[2]$)
    , where $K_1$, $K_2$ and $K_3$ are index sets, $1\leq s_{1}(k),s_{3}(k)\leq \ell$ and $2\leq s_{2}(k)\leq \ell+1$.
    If $\Ker{f}=0$ (for the definition of $\Ker{f}$, see just before Lemma~\ref{oct28-4}), 
    then there exists a morphism $\bigoplus_{k\in K_2}A_{s_{1}(k)}[1]\oplus \bigoplus_{k\in K_2}Q_{s_{2}(k)}[1]\to \bigoplus_{k\in K_3}A_{s_{3}(k)}[2]$ such that the following diagram is commutative:
    \[\xymatrix{\CC_0\ar[r]^-{f}\ar[dr]_{g}&\displaystyle \bigoplus_{k\in K_2}A_{s_{1}(k)}[1]\oplus \bigoplus_{k\in K_2}Q_{s_{2}(k)}[1]\ar[d]\\ &\displaystyle\bigoplus_{k\in K_3}A_{s_{3}(k)}[2].}\]
\end{lemma}
\begin{proof}
    The assertion follows from the (non-zero) commutative diagrams:
    \[
    \xymatrix{&\CC_{\CS}[1]\ar[d]\\\CC_0\ar[r]\ar[ur]&A_{s}[1],}\quad
        \xymatrix{&\CC_{\CS}[1]\ar[d]\\\CC_0\ar[r]\ar[ur]&Q_{s}[1],}
    \quad\xymatrix{\CC_0\ar[r]\ar[dr]&\CC_{\CS}[1]\ar[d]\\&\CC_{\CS}[2].}\]
\end{proof}

Applying this lemma to $f_i'$ and $g_i'$, we obtain the following.

\begin{corollary}
    There exists a morphism $l_{i+1}\colon M_{i+1}[i+1]\to M_{i+2}[i+2]$ such that
    the following diagram commutes:
    \[\xymatrix{\Coim{f_i}[i]\ar[dr]_{g_i'}\ar[r]^-{f_i'}&M_{i+1}[i+1]\ar[d]^{l_{i+1}}\\ &M_{i+2}[i+2].
    }\]
\end{corollary}

We consider a new morphism $h''$, replacing $g_i'$ of $h'$ with $0$:
\begin{align}
     h'':=\bigoplus_{i\in \ZZ}(f_i'\oplus 0)\colon   \bigoplus_{i\in \ZZ}\Coim{f_i}[i]\to \bigoplus_{j\in \ZZ}M_j[j].
\end{align}
Then, we have a commutative diagram:
\[\xymatrix{\bigoplus_{i\in \ZZ}\Coim{f_i}[i]\ar[r]^-{h''}\ar@{}[d]|{\rotatebox{90}{$=$}}& \bigoplus_{j\in \ZZ}M_j[j] \ar[d]^{\bigoplus_{j\in \ZZ}(1+ l_j)}\\\bigoplus_{i\in \ZZ} \Coim{f_i}[i]\ar[r]^-{h'}&\bigoplus_{j\in \ZZ}M_j[j].}\]
Since the vertical arrow $\bigoplus_{j\in \ZZ}(1\oplus l_j)$ is a quasi-isomorphism, we have the following assertion.

\begin{lemma}
The cone of $h$ in $H^0\Sh(\CC)$ is isomorphic to the cone of $h''$ in $H^0\Sh(\CC)$.
In particular, the cone of $h$ is the direct sum of the cone of $f_i'\colon \Coim{f_i}[i]\to M_{i+1}[i+1]$. 
\end{lemma}

By using Lemma~\ref{oct28-1},
the cone of $f_i'$ in $H^0\Sh(\CC)$ is decomposed into several following objects:
\begin{enumerate}
    \item the cone of $\CC_0[i]\to  A_{s}[i+1]$ for some $1\leq s\leq \ell$,
    \item the cone of $\CC_0[i]\to  Q_{s}[i+1]$ for some $2\leq s\leq \ell+1$,
    \item $A_{s}[i+1]$ for some $1\leq s\leq \ell$
    \item $Q_{s}[i+1]$ for some $2\leq s\leq \ell+1$
    \item $\CC_0[i+1]$.   
\end{enumerate}
Note that the first one is $P_s[i+1]$ and the second one is $B_{s-1}[i+1]$.
In particular, we conclude the following.

\begin{proposition}\label{oct29-1}
     For an object $F\in \Sh(\CC,0)$,
assume that there exists $\ell\in \ZZ_{\geq 1}$ such that 
$(T-1)^\ell=0$ for the monodromy automorphism $T$ of $\psi_tF$.
Then, 
$F$ can be described as a direct sum of several shifted objects of the following five types of perverse sheaves:
\begin{enumerate}
    \item $\CC_0$
    \item $A_s[1]$ ($1\leq s\leq \ell$)
    \item $B_s[1]$ ($1\leq s\leq \ell$)
    \item $P_s[1]$ ($1\leq s\leq \ell$)
    \item $Q_s[1]$ ($2\leq s\leq \ell+1$).
\end{enumerate}
In particular, if $F$ satisfies the condition $(N_s)$ (see the Definition~\ref{dec25-1}),
then $F$ can be expressed as a direct sum of several shifted objects of 
$P_{s'}[1]$ ($s'\leq s$),
$A_{s''}[1]$, $B_{s''}[1]$, $Q_{s''}[1]$ ($s''<s$),
or $\CC_0$.
\end{proposition}

\begin{remark}\label{remark:quiver}
Proposition~\ref{oct29-1} implies the fact:
to give an object $F\in \Sh(\CC,0)$ whose nilpotent order is $\leq \ell$ 
is equivalent to
give a tuple $(V,W,c,v)$:
\[\xymatrix{V\ar@<-0.5ex>[r]_{v}&W\ar@<-0.5ex>[l]_{c}},\]
where $V$ and $W$ are (possibly infinite dimensional) vector spaces
and $c,v$ are linear morphisms with $(cv)^{\ell}=0$ and $(vc)^\ell=0$.

\end{remark}

\bibliographystyle{alpha}
\bibliography{brane.bib}

\noindent Tatsuki Kuwagaki: 
Department of Mathematics, Kyoto University, Kitashirakawa Oiwake-cho, Sakyo-ku, Kyoto 606-8502, Japan.

\noindent 
\textit{E-mail address}: \texttt{tatsuki.kuwagaki.a.gmail.com}

\medskip

\noindent Takahiro Saito:
Department of Mathematics, Chuo University,
1-13-27 Kasuga, Bunkyo-ku, Tokyo 112-8551, Japan.

\noindent
\textit{E-mail address}: \texttt{takahiro.saito27.a.gmail.com}

\end{document}